%% file: main.tex
\documentclass[sn-mathphys,Numbered]{sn-jnl}

\usepackage{graphicx}%
\usepackage{multirow}%
\usepackage{amsmath,amssymb,amsfonts}%
\usepackage{amsthm}%
\usepackage{mathrsfs}%
\usepackage{xcolor}%
\usepackage{textcomp}%
\usepackage{manyfoot}%
\usepackage{booktabs}%
\usepackage{listings}%

\input{heading}

\input{def}

\begin{document}

\title[Probability and moment inequalities for additive functionals of geometrically ergodic Markov chains]{Probability and moment inequalities for additive functionals of geometrically ergodic Markov chains}


\author[1]{\fnm{Alain} \sur{Durmus}}\email{alain.durmus@polytechnique.edu}

\author[1]{\fnm{Eric} \sur{Moulines}}\email{eric.moulines@polytechnique.edu}

\author[2]{\fnm{Alexey} \sur{Naumov}}\email{anaumov@hse.ru}

\author*[2]{\fnm{Sergey} \sur{Samsonov}}\email{svsamsonov@hse.ru}

\affil[1]{Ecole polytechnique, Paris, France}

\affil[2]{HSE University, Moscow, Russia}


\abstract{In this paper, we establish moment and Bernstein-type inequalities for additive functionals of  geometrically ergodic Markov chains. These inequalities extend the corresponding inequalities for independent random variables. Our conditions cover Markov chains converging geometrically to the stationary distribution either in weighted total variation norm or in weighted Wasserstein distances. Our inequalities apply to unbounded functions and depend explicitly on constants appearing in the conditions that we consider.}

\keywords{concentration inequalities for Markov chains, cumulant expansion}

\pacs[MSC Classification]{60E15, 60J20, 65C40}

\maketitle

\section{Introduction}
\label{sec:introduction}
\input{introduction}

\section{Main results}
\label{sec:main-results}
\subsection{Geometrically $V$-ergodic Markov chains}
\label{sec:geom-v-ergod}
First, we consider the case where the Markov kernel $\MK$ is $V$-uniformly geometrically ergodic. We impose the following assumptions on the Markov kernel $\MK$:
\begin{assumption}
\label{assG:kernelP_q}
There exist a measurable function $V: \Xset \to \coint{\rme,\infty}$, $\lambda \in (0,1)$, and $b \geq 0$ such that for any $x \in \Xset$, $\MK V(x) \leq \lambda V(x) + b$.
\end{assumption}
Note that, in contrast to the usual definition of Lyapunov functions in the Markov chain literature, we assume here that $V$ takes values in $\coint{\rme,\plusinfty}$, rather than in $\coint{1,\plusinfty}$. This choice allows us to avoid technical problems when we consider $W = \log
V$ later in this section.

\begin{assumption}
\label{assG:kernelP_q_smallset}
There are an integer $ m \geq 1$, $\epsilon \in (0,1)$, and $d \in \rset_{+}$, such that  the level set $\{x \in\Xset \, :\, V(x) \leq d\}$ is $(m,\epsilon)$-small and $\lambda+2b/(1+d)<1$. The quantities $\lambda$ and $b$ are defined in \Cref{assG:kernelP_q}.
\end{assumption}
The definition of $(m,\epsilon)$-small set can be found, e.g., in \cite[Definition~9.1.1.]{douc:moulines:priouret:soulier:2018}. In particular, it is known that the Markov kernel $\MK$ is known to be uniformly geometrically ergodic if and only if the entire space $\Xset$ is $(m,\epsilon)$-small, see \cite[Theorem~15.3.1]{douc:moulines:priouret:soulier:2018}. Under \Cref{assG:kernelP_q} and \Cref{assG:kernelP_q_smallset}, the Markov kernel $\MK$ admits a unique invariant probability measure $\pi$ satisfying $\pi(V) < \infty$. Moreover, \cite[Theorem~19.4.1]{douc:moulines:priouret:soulier:2018} implies that for any probability measure $\xi$ satisfying $\xi(V) < \infty$ and all $n \in \nset$,
\begin{align}
    \label{eq:V-geometric-coupling-general}
    \tvnorm{\xi \MK^n - \pi} \leq \Vnorm[V]{\xi \MK^n - \pi} \leq \cmconstv \{ \xi(V) + \pi(V) \} \ratev^n \eqsp,
  \end{align}
  where the constants $\ratev$ and $\cmconstv$ are given by
  \begin{equation}
    \label{eq:bornes-v-geometric-coupling}
      \begin{aligned}
        &\log \ratev = \frac{\log(1-\epsilon) \log\bar\lambda_m} { m\bigl(\log(1-\epsilon) +
          \log\bar\lambda_m-\log\bar{b}_m\bigr) } \eqsp ,\\
        &\bar\lambda_m = \lambda^m+2b_m/(1+d) \eqsp,
        \eqspp \bar{b}_m = \lambda^m b_m + d
        \eqsp, \eqspp b_m=b (1-\lambda^m) /(1-\lambda) \eqsp ,   \\
        &\cmconstv  = \ratev^{-m}\{\lambda^m+(1-\lambda^m)/(1-\lambda)\}\{1+\bar{b}_m/[(1-\epsilon)(1-\bar\lambda_m)]\} \eqsp .
  \end{aligned}
\end{equation}
Before proceeding with our main results, we introduce some additional quantities. For each $q \in \nset, u \in \{1, \dots, q-1\}$ and $\gamma \geq 0$, we introduce
\begin{equation}
\label{eq: B_u_q_def_new}
\ConstB_{\gamma}(u,q)
= \frac{(2q)!}{u!} \sum_{(k_1,\ldots,k_u) \in \scrE_{u,q}}   \prod_{i=1}^u (k_{i}!)^{\gamma + 2} \eqsp,
\end{equation}
where $\scrE_{u,q} = \{ (k_1,\ldots,k_u) \in \nset^u \, : \, \sum_{i=1}^u k_i = 2q\, ,\, k_i \geq 2\}$.
Note that the cardinality of $\scrE_{u,q}$ is $\binom{2q-u-1}{u-1}$ which implies and upper bound
\begin{equation}
\label{eq:bound-Balpha}
\ConstB_{\gamma}(u,q)
\leq \frac{(2q)!}{u!} \binom{2q-u-1}{u-1} \bigl((2q-2u+2)!\bigr)^{2+\gamma} 2^{(u-1)(2+\gamma)}\eqsp.
\end{equation}
We first establish a Rosenthal-type inequality for uniformly $V$-geometrically ergodic Markov chains. The leading term is the variance (under stationarity) scaled by the corresponding moment of the Gaussian distribution.
\begin{theorem}\label{th:rosenthal_V_q}
Assume \Cref{assG:kernelP_q}, \Cref{assG:kernelP_q_smallset}, and let $q \in \nset^*$. Then, for any function $g \in \mrl_{V^{1/(2q)}}$, 
\begin{equation}
\label{eq:main-rosenthal}
\PE_\pi[|S_n|^{2q}] \leq \momentGq[q] \{\PVar[\pi](S_n)\}^q + \Constmainros^{2q} \Vnorm[V^{1/(2q)}]{\bar{g}}^{2q}
\sum_{u=1}^{q-1} \frac{\ConstB_0(u,q) n^{u}}{\ratev^{u/2} \log^{2q-u}{(1/\ratev)}}\eqsp,
\end{equation}
where
\begin{equation}
\label{eq:constmainros}
\Constmainros = 2 \cmconstv \pi(V)\eqsp.
\end{equation}
\end{theorem}
\begin{proof}
The proof is postponed to \Cref{sec:proof-ros_v_q}.
\end{proof}
\begin{remark}
\label{remark:theo_1_scaling_mix_time}
Note that the bound \eqref{eq:main-rosenthal} scales homogeneously w.r.t. the factor $\{\log(1/\ratev)\}^{-1}$. Indeed, for any $g \in \mrl_{V^{1/(2q)}}$, we get applying \Cref{lem:geom_ergodicity_variance_bound}, that 
\[
\PVar[\pi](S_n) \leq 5 n c^{1/2} \ratev^{-1/2} \{\log{1/\ratev}\}^{-1} \pi(V)^{3/2} \Vnorm[V^{1/2}]{\bar{g}}^{2}\eqsp.
\]
Thus, applying \Cref{lem:scale_B_gamma} and dividing both parts of \eqref{eq:main-rosenthal} by $n^{2q}$, we obtain from \eqref{eq:main-rosenthal} that
\begin{equation}
\label{eq:main-rosenthal-simplified}
\PE_\pi\bigl[\bigl|\tfrac{S_n}{n}\bigr|^{2q}\bigr] \lesssim q^{q} \biggl(\frac{\ratev^{-1/2}}{n \log(1/\ratev)}\biggr)^{q} + q^{2q} \sum_{u=1}^{q-1} \biggl(\frac{\ratev^{-1/2} (2q-u)^{2}}{\rme n \log(1/\ratev)}\biggr)^{2q-u}\eqsp.
\end{equation}
Here $ a \lesssim b $ stands for $a \leq \operatorname{c} b$ where $\operatorname{c}$ prefactor, depending upon $c, \pi(V)$, and $\Vnorm[V^{1/(2q)}]{\bar{g}}$. That is, for $n = \kappa \ratev^{-1/2} / \log(1/\ratev)$, the r.h.s. of \eqref{eq:main-rosenthal-simplified} scales with $\kappa^{-1}$. 
\end{remark}
The above result can be extended using the construction of the exact distributional coupling for any initial distribution (see \Cref{sec:non-stationary-extension} for the necessary definitions). It is worth noting that in our approach it is not necessary to assume that the Markov chain is strongly aperiodic (unlike \cite{adamczak2015exponential}).
\begin{theorem}
\label{theo:changeofmeasure}
Assume \Cref{assG:kernelP_q}, \Cref{assG:kernelP_q_smallset} and let $q \in \nset^*$. Then, for any probability measure $\xi$ on $(\Xset,\Xsigma)$ satisfying $\xi(V) < \infty$ and
$g \in \mrl_{V^{1/(2q)}}$,
\begin{equation*}
\PE_\xi\big[ \big|S_n \big|^{2q} \big] \leq 2^{2q-1} \PE_\pi\big[ \big|S_n \big|^{2q} \big]  + 2^{6q-1} \Vnorm[V^{1/(2q)}]{\bar{g}}^{2q} \cmconstv \{ \xi(V) + \pi(V) \} \frac{q^{2q}}{\ratev (\log(1/\ratev))^{2q}}\eqsp.
\end{equation*}
\end{theorem}
\begin{proof}
  The proof is postponed to \Cref{sec:non-stationary-extension}.
\end{proof}

In the above results we fix $q \in \nsets$ and consider a function $g \in \mrl_{V^{1/(2q)}}$. Of course, with these assumptions we can only obtain moment bounds of order $2q$ or smaller and cannot control the exponential moments of $S_n$. In our next statement, we consider the case of the function $g \in \mrl_{W^{\gamma}}$, where $W = \log V$ and $\gamma \geq 0$. Note that when $\gamma = 0$, $\mrl_{W^{\gamma}}$ coincides with the set of bounded functions. In this case, in addition to the Rosenthal-type bound \eqref{eq:main-rosenthal}, we can formulate a counterpart of the Bernstein-type bound \eqref{eq:gene_Bernstein type bound}. We begin with the result that is a counterpart of \Cref{th:rosenthal_V_q}.
\begin{theorem}\label{th:rosenthal_log_V}
Assume \Cref{assG:kernelP_q}, \Cref{assG:kernelP_q_smallset} and let $\gamma \geq 0$, $q \in \nset^*$. Then for any $g \in \mrl_{W^\gamma}$, it holds
\begin{equation*}
\PE_{\pi}[|S_n|^{2q}] \leq \momentGq[q] \{\PVar[\pi](S_n)\}^q+  \Constmainros^{2q} (2\gamma)^{2\gamma q} \Vnorm[W^\gamma]{\bar{g}}^{2q} \sum_{u=1}^{q-1}  \frac{\ConstB_{\gamma}(u,q) n^{u}}{\ratev^{u/2} \log^{2q-u}{(1/\ratev)}}\eqsp,
\end{equation*}
where $\Constmainros$ is defined in \eqref{eq:constmainros}.
\end{theorem}
\begin{proof}
The proof is postponed to ~\Cref{sec:proof-ros_log_V}.
\end{proof}
Similar to \Cref{theo:changeofmeasure}, we provide a version of the above statement for any initial distribution.
\begin{theorem}
\label{theo:changeofmeasure-1}
Assume \Cref{assG:kernelP_q}, \Cref{assG:kernelP_q_smallset} and let $\gamma \geq 0$,  $q \in \nset^*$. Then for any probability measure $\xi$ on $(\Xset,\Xsigma)$, and any $g \in \mrl_{W^\gamma}$, it holds
\begin{equation}
     \PE_\xi\big[ \big|S_n \big|^{2q} \big] \leq 2^{2q-1} \PE_\pi\big[ \big|S_n \big|^{2q} \big]  + 2^{4q-2} \Vnorm[W^\gamma]{\bar{g}}^{2q} \cmconstv \{ \xi(V) + \pi(V) \} \operatorname{D}^{(1)}_{q,\gamma}\eqsp,
\end{equation}
where
\begin{equation}
\operatorname{D}^{(1)}_{q,\gamma}= \rme^{-1} \ratev^{-1} \{ \log(1/\ratev) \}^{1- 4 q} (4q-2) ! +  \ratev^{-1} \{ \log(1/\ratev) \}^{-1} (4 q \gamma/ \rme)^{4 q \gamma}\eqsp.
\end{equation}
\end{theorem}
\begin{proof}
  The proof is postponed to \Cref{sec:non-stationary-extension}.
\end{proof}
We can also obtain Bernstein-type bound. We start from the stationary case.
\begin{theorem}
\label{th:rosenthal_log_V_cor_2}
Assume \Cref{assG:kernelP_q}, \Cref{assG:kernelP_q_smallset} and let $\gamma \geq 0$. Then for any $g \in \mrl_{W^\gamma}$ and $t \geq 0$, it holds that 
\begin{equation}
\label{eq:bernstein_mc}
\PP_{\pi}(|S_n| \geq t) \leq 2\exp\biggl\{-\frac{t^2/2}{\PVar[\pi](S_n) + \ConstJ^{1/(\gamma+3)} t^{2-1/(\gamma+3)}}\biggr\}\eqsp.
\end{equation}
Moreover, for any $\delta > 0$ we get
\begin{equation}
\label{eq:high_prob_bound_W_ergodic}
\PP_{\pi}\biggl(|S_n| \geq 2\sqrt{\PVar[\pi](S_n)}\sqrt{\log(4/\delta)} + 4^{\gamma+3}\ConstJ\{\log(4/\delta)\}^{\gamma+3}\biggr) \leq \delta\eqsp.
\end{equation}
Here the constant $\ConstJ$ is given by
\begin{equation}
\label{eq:const_B_n_definition_main}
\ConstJ = \biggl( \frac{n \ratev^{-1/2} \{\log(1/\ratev)\}^{-1} \Constmainros^{2} \|\bar{g}\|_{W^{\gamma}}^{2}}{\PVar[\pi](S_n)} \vee 1\biggr) \frac{ 2^{1+3\gamma}\gamma^{3\gamma} \Constmainros \|\bar{g}\|_{W^{\gamma}}}{\log(1/\ratev)} \eqsp.
\end{equation}
\end{theorem}
\begin{proof}
The proof is postponed to ~\Cref{sec:proof_bernstein_bound}.
\end{proof}
Comparing \eqref{eq:bernstein_mc} with \eqref{eq: Bernstein type bound}, one can see that in the subexponential regime $t^{1/(\gamma+1)}$ is replaced by $t^{1/(\gamma+3)}$, as in \cite{doukhan2007probability}. This factor is caused by the dependence along the observations of the Markov chain. Similar to \eqref{eq: Bernstein type bound}, the constant $\ConstJ$ is not distribution-free, moreover, compared to \eqref{eq:gene_Bernstein type bound}, it is possible that $\ConstJ$ scales with $n$. However, the dependence in the distribution is fully explicit: one has to compare $\PVar[\pi](S_n)$ with $n \{\log(1/\ratev)\}^{-1} \Vnorm[W^\gamma]{\bar{g}}^{2}$.
\begin{remark}
\label{rem:lower_bound_variance}
For some particular instances of Markov chains, this comparison can be done explicitly. Indeed, let us choose a function $g$ with $\|\bar{g}\|_{W^\gamma} \leq 1$, $\pi(g^2) < \plusinfty$, and introduce for $\ell \in \nset$ the quantity
\[
\varsigma_\pi(g,\ell) = \PCov[\pi]( g(X_0),g(X_{\ell}))\eqsp.
\]
Under \Cref{assG:kernelP_q} and \Cref{assG:kernelP_q_smallset}, the Markov chain is$V$-uniformly geometrically ergodic, which implies $\sum_{\ell=-\infty}^\infty | \varsigma_\pi(g,\ell)| < \infty$.
Therefore, we can calculate the spectral density $f(g,\lambda)= (2\uppi)^{-1}\sum_{\ell = -\infty}^{\infty} \varsigma_\pi(g,\ell) \rme^{-\rmi \ell \lambda}$, for $\lambda \in [-\pi,+\pi]$. If we additionally assume that there is $f_{\min} \in \rset_{+}$ such that $f(g,\lambda) \geq f_{\min}$ for all $\lambda \in [-\pi,\pi]$, it is straightforward to show that $\PVar[\pi](S_n) \geq n f_{\min}$ and hence the constant
\[
\ConstJ \leq \biggl(\frac{\ratev^{-1/2} \{\log(1/\ratev)\}^{-1} \Constmainros^{2}}{f_{\min}} \vee 1\biggr) \frac{ 2^{1+3\gamma}\gamma^{3\gamma} \Constmainros}{\log(1/\ratev)}
\]
is independent of $n$.
\end{remark}
Finally, we provide a Bernstein-type bound for the case of an arbitrary initial distribution. For this proof we again use distributional coupling, but the reasoning is more complicated to obtain Weibulian dependence in the initial conditions.

\begin{theorem}
  \label{theo:prob_ineq_V_norm}
  Assume \Cref{assG:kernelP_q} and \Cref{assG:kernelP_q_smallset} and let $\gamma \geq 0$. Then, for any initial distribution $\xi$ on $(\Xset,\Xsigma)$, $g \in \mrl_{W^\gamma}$, and $t \geq 0$, it holds setting  $\varpi_{\gamma} = 1/(1+\gamma)$, that
  \begin{align*}
&\PP_{\xi}(|S_n| \geq t) \leq    \PP_{\pi}(|S_n| \geq t/4) \\
    & \quad + \parenthese{\frac{\rme^{-\log(1/\rho)t^{\varpi_{\gamma}}/(4^{1+\varpi_{\gamma}}\|\bar{g}\|_{W^\gamma}^{\varpi_{\gamma}}\varpi_{\gamma})}}{\ratev^{1/2}} +   \frac{\rme^{-(1+\gamma) t^{\varpi_{\gamma}}/(2^{1+2\varpi_{\gamma}}\|\bar{g}\|_{W^\gamma}^{\varpi_{\gamma}}\gamma)}}{1-\ratev} }\cmconstv \{ \xi(V) + \pi(V) \}\eqsp.
  \end{align*}
\end{theorem}
\begin{proof}
  The proof is postponed to \Cref{sec:non-stationary-extension}.
\end{proof}
It is worth noting that the exponent of the terms reflecting the dependence in the initial conditions is $1/(1+\gamma)$, as in \cite[Theorem~5.1]{adamczak2015exponential}, but without the assumption of strong aperiodicity. Finally, unlike \cite[Theorem~5.1]{adamczak2015exponential}, the dependence in the initial condition appears as a multiplicative factor rather than in the exponential rate.
\subsection{Geometrically ergodic Markov chains with respect to Wasserstein semi-metric}
\label{sec:geom-ergod-mark}
Now we extend the results obtained in \Cref{sec:geom-v-ergod} to the case of Markov kernels that are geometrically contracting for a weighted Wasserstein (pseudo)distance. The advantage of this setting is that we do not have to assume that the Markov kernel is irreducible. This is a significant advantage for the study of stochastic algorithms (which is one of the goals of this paper), but also for Markov chains in infinite dimensions; see \cite{hairer2011asymptotic,hairer:stuart:vollmer:2012,butkovsky:veretennikov:2013} and \cite[Chapter~20]{douc:moulines:priouret:soulier:2018} and references therein.
In this section we assume that $(\Xset,\distance)$ is a complete separable metric space and denote by $\Xsigma$ its Borel $\sigma$-field. Let $\cost:\Xset\times\Xset\to\rplus$ satisfy the following condition.
\begin{assumptionC}
  \label{ass:cost_fun}
$\cost$ is a lower semicontinuous symmetric function such that $\cost(x,x')=0$ for $x=x'$. Also, there is $\pcost \in\nsets$ such that for any $x,x'\in\Xset$, $(\distance(x,x') \wedge 1)^{\pcost} \leq \cost(x,x') \leq 1$.
\end{assumptionC}
A function $\cost$ satisfying \Cref{ass:cost_fun} is called distance-like.
For two probability measures $\xi$ and
$\xi'$ on $(\Xset,\Xsigma)$, we say that a probability measure $\nu$ on $(\Xset^2,\Xsigma^{\otimes 2})$ is a coupling of $\xi$ and $\xi'$, if for each $\msa \in \Xsigma$, $\nu(\msa \times \Xset) = \xi(\msa)$ and $\nu( \Xset \times \msa) = \xi'(\msa)$.
Denote by
$\couplingmeasure(\xi,\xi')$  the set of couplings of $\xi$ and
$\xi'$ on $(\Xset,\Xsigma)$,  and define
\begin{align*}
  \wasser[\cost]{\xi}{\xi'} = \inf_{\nu \in \couplingmeasure(\xi,\xi')} \int_{\Xset\times\Xset} \cost(x,x')
  \nu(\rmd x\rmd x')  \eqsp .
\end{align*}
We say that $\MKK$ is a Markov coupling of $\MK$ if for all $(x,x') \in \Xset^2$ and $\msa \in \Xsigma$, $\MKK((x, x'), \msa \times \Xset) = \MK (x, \msa)$ and $\MKK((x,x'), \Xset \times \msa) = \MK(x',\msa)$.
If $\MKK$ is a kernel coupling of $\MK$, then for every $n \in \nset$, $\MKK^n$ is a kernel coupling of $\MK ^n$ and for every $\nu \in \couplingmeasure (\xi,\xi')$, $\nu \MKK^n$ is a coupling of $(\xi \MK ^n,\xi'\MK^n)$, which means
$ \wasser[\cost]{\xi \MK ^n}{\xi' \MK ^n} \leq \int_{\Xset\times\Xset} \MKK^n\cost(x,x') \nu(\rmd x\rmd x')$.
For any  probability measure $\nu$ on $(\Xset^2,\Xsigma^{\otimes 2})$, we denote by
$\PP_{\nu}^\MKK$ (respectively $\PE_{\nu}^\MKK$) the probability (respectively the
expectation) on the canonical space $((\Xset^2)^\nset,(\Xsigma^{\otimes 2})^{\otimes \nset})$  such that the canonical process $\sequenceDouble{X}{X'}[n][\nset]$ is a Markov chain with initial
probability $\nu$ and Markov kernel $\MKK$.  By convention, we set
$\PE_{x,x'}^{\MKK} = \PE_{\delta_{x,x'}}^{\MKK}$ for all $(x,x') \in \Xset^2$.
Consider the following assumption, which weakens the $\distance$-small set condition of \cite{hairer2011asymptotic} by allowing the contraction to occur in $m \in \nset^*$ steps:
\begin{assumption}
  \label{assG:kernelP_q_contractingset_m}
 There exist a kernel coupling $\MKK$ of $\MK$, $m \in \nset, \minorwas \in (0,1)$, $\boundmetric \geq 1$ such that  \begin{equation}
  \label{eq:assG:kernelP_q_contractingset_m}
\MKK \metricc(x,x') \leq \boundmetric \metricc(x,x') \eqsp, \qquad    \MKK^m \metricc(x,x') \leq (1 - \minorwas \indi{\CKset}(x,x'))\metricc(x,x') \eqsp,
\end{equation}
where $ \CKset = \{V \leq d\} \times \{V \leq d\}$ with $\lambda+2b/(1+d)<1$ where $\lambda$ and $b$ are given in \Cref{assG:kernelP_q}.
\end{assumption}
Define for $x,x'\in\Xset$, $\bar V(x,x') = \{V(x) + V(x')\}/2$, $\bar \lambda_m = \lambda^m+ 2 b_m/(1+d)$,
$b_m = b(1-\lambda^m)/(1-\lambda)$, and $\bar d = (d+1)/2$. Consider the equation with unknown $\delta \geq 0$,
\begin{equation}
\label{eq: delta def}
(1-\minorwas)\lr{\frac{\bar \lambda_m+ b_m+\delta}{1+\delta}} =  \frac{\bar \lambda_m \bar{d}+\delta}{\bar{d}+\delta}\eqsp.
\end{equation}
Since necessarily, $b \geq 1$, note that the left-hand side of this equation is  a decreasing function of $\delta$, while the right-hand side is an increasing function. Hence, \eqref{eq: delta def} has a unique positive root (denoted by $\rootwas$) if $(1-\minorwas)(\bar \lambda_m+ b_m) > \bar \lambda_m$, and we define
\begin{equation}
\label{eq:delta_star_def}
\deltawas =
\begin{cases}
\rootwas & \text{ if } (1-\minorwas)(\bar \lambda_m+ b_m) > \bar \lambda_m\eqsp, \\
0 & \text{otherwise}\eqsp.
\end{cases}
\end{equation}
We first note that the assumptions \Cref{assG:kernelP_q} and \Cref{assG:kernelP_q_contractingset_m} imply the existence and uniqueness of an invariant distribution $\pi$, and second, that for any initial $\xi$-distribution, the $\xi \MK ^n$-iterates geometrically converge  to the invariant distribution $\pi$- for the pseudodistance $\wassersym[\metricc^{1/2} \bar{V}^{1/2}]$. This result generalizes the weak Harris theorem of \cite{hairer2011asymptotic} (see also \cite[Theorem~20.4.5]{douc:moulines:priouret:soulier:2018}).
\begin{proposition}
\label{prop:wasser:geo}
Assume \Cref{assG:kernelP_q}, \Cref{assG:kernelP_q_contractingset_m} and \Cref{ass:cost_fun}, and let $q \in \nset^*$. Then for $(x,x')\in \Xset^2$, $p \le 2q$, and $n\in \nset$, $n \geq m$ it holds
\begin{equation}
\label{eq: constraction}
\PE_{x,x'}^{\MKK}[\metricc^{1/2}(X_n, X_n') \bar V^{p/(4q)}(X_n, X'_n)] \leq  \boundmetric^{m/2} \vartconstwas^{p/(2q)}  \metricc^{1/2}(x,x') \bar{V}^{p/(4q)}(x,x') \ratewas^{np/(2q)}  \eqsp,
\end{equation}
where
\begin{equation}
\label{eq:def:rho}
\ratewas  = \Bigl(\frac{\bar \lambda_m \bar{d}+\deltawas}{\bar{d}+\deltawas} \Bigr)^{1/(2m)} < 1 \eqsp,  \quad \vartconstwas = (1 + b/(1-\lambda) + \deltawas)^{1/2} / \ratewas^m \eqsp.
\end{equation}
\end{proposition}
\begin{corollary}
\label{cor:wasserstein-convergence}
Assume \Cref{assG:kernelP_q}, \Cref{assG:kernelP_q_contractingset_m}, and \Cref{ass:cost_fun}. Then $\MK$ admits a unique invariant probability measure $\pi$ satisfying $\pi(V) < \infty$. Moreover, for all
  initial distributions $\xi$ and $n\in \nset$,
  \begin{equation}
    \label{eq:wasser:geo:bound:pi}
    \wasser[\cost]{\xi \MK^n}{\pi} \leq \wasser[\cost^{1/2} \bar V^{1/2}]{\xi \MK^n}{\pi}
    \leq  (1/\sqrt{2}) \boundmetric^{m/2} \vartconstwas  \ratewas^{n}    \lrb{\xi(V^{1/2})+\pi(V^{1/2})} \eqsp.
  \end{equation}
\end{corollary}
\begin{proof}
The proof is postponed to \Cref{sec:proof-ros_W_q}.
\end{proof}
For a measurable function  $\lyapW: \Xset \to \coint{1,\infty}$, set $\bar \lyapW(x,y) = (\lyapW(x) + \lyapW(y))/2$, and for  $\beta \in \rset_+$, define
\begin{equation*}
  \Nnorm[\beta, \lyapW]{f} = \max \bigg \{\sup_{\substack{x,x'\in\Xset\eqsp, \,\, x \neq x'}} \frac{|f(x) - f(x')|}{\metricc^{1/2}(x,x') \bar{\lyapW}^{\beta}(x,x')} , \, \sup_{x \in \Xset} \frac{|f(x)|}{\lyapW^{\beta}(x)} \bigg\}\eqsp,
\end{equation*}
and $\Lclass_{\beta,\lyapW} = \{f: \Xset \to \rset: \Nnorm[\beta,\lyapW]{f} < \infty\}$. The first main result of this section is a Rosenthal-type inequality for geometrically ergodic Markov chains in terms of the Wasserstein semi-metric. Again, the leading term is the stationary variance multiplied by the corresponding moment of a Gaussian random variable.
\begin{theorem}
\label{th:rosenthal_V_poly_wasserstein}
Assume \Cref{assG:kernelP_q}, \Cref{assG:kernelP_q_contractingset_m}, \Cref{ass:cost_fun}, and let $q \in\nsets$.
Then for any function $g \in \Lclass_{1/(4q), V}$,
\begin{equation}
\label{eq:wasser_scaling_main}
    \PE_\pi[|S_n|^{2q}] \leq \momentGq[q] \{\PVar[\pi](S_n)\}^q + \Constwasspoly^{2q} \Nnorm[1/(4q), V]{\bar{g}}^{2q} \sum_{u=1}^{q-1} \frac{\ConstB_0(u,q) n^{u}}{\ratewas^{u/2} \{\log(1/\ratewas)\}^{2q-u}}\eqsp,
\end{equation}
where $\ConstB_0(u,q)$ is defined in~\eqref{eq: B_u_q_def_new} and with $\vartconstwas$ in \eqref{eq:def:rho},
\begin{equation}
\label{eq:const_poly_class_wasserstein}
\Constwasspoly = 2\sqrt{2} \boundmetric^{m/2} \vartconstwas \{\pi(V)\}^{1/2}\eqsp.
\end{equation}
\end{theorem}
\begin{proof}
The proof is postponed to~\Cref{sec:proof:rosenthal_V_poly_wasserstein}.
\end{proof}

Note that the right-hand side of \eqref{eq:wasser_scaling_main} has the same homogeneous scaling with respect to the ratio $n/\log(1/\ratewas)$ as in the corresponding bound for the $V$-geometrically ergodic case \eqref{eq:main-rosenthal}. We can now extend this result to the non-stationary case in a similar way to \Cref{theo:changeofmeasure}.
We use here a coupling argument but unlike \Cref{theo:changeofmeasure} we do not use a distributional coupling but a coupling kernel together with the coupling inequality outlined in \Cref{prop:wasser:geo}.
\begin{theorem}
\label{theo:changeofmeasure_wasser}
Assume \Cref{assG:kernelP_q}, \Cref{assG:kernelP_q_contractingset_m}, \Cref{ass:cost_fun}, and let $q \in \nset^*$. Then, for any probability measure $\xi$ on $(\Xset,\Xsigma)$ satisfying $\xi(V^{1/2}) < \infty$ and $g \in \Lclass_{1/(4q), V}$, we get
\begin{equation*}
     \PE_\xi\big[ \big|S_n \big|^{2q} \big] \leq 2^{2q-1} \PE_\pi\big[ \big|S_n \big|^{2q} \big] + 2^{4q-1} \Nnorm[1/(4q), V]{\bar{g}}^{2q}\,\boundmetric^{m/2}  \vartconstwas \{\xi(V^{1/2}) + \pi(V^{1/2})\} \frac{q^{2q}}{\ratewas (\log(1/\ratewas))^{2q}}\eqsp.
\end{equation*}
\end{theorem}
\begin{proof}
  The proof is postponed to \Cref{sec:proof-crefth_wass_change_mease}.
\end{proof}

Finally, we provide a series of results where we replace the class $\Lclass_{1/(4q), V}$ by the class $\Lclass_{1, W^\gamma}$ for $\gamma \geq 0$. We first prove a Rosenthal-type inequality in the stationary case, which we then extend to the arbitrary inital distribution. Results below are the analogues of \Cref{th:rosenthal_log_V} and \Cref{theo:changeofmeasure-1}. The proof in the stationary case again involves an inequality on centered moments adapted to the weighted Wasserstein distance. The extension to the non-stationary case still requires a coupling inequality but more subtle than for \Cref{theo:changeofmeasure_wasser}.
\begin{theorem}\label{th:rosenthal_log_V_wasserstein}
  Assume \Cref{assG:kernelP_q}, \Cref{assG:kernelP_q_contractingset_m}, \Cref{ass:cost_fun}, let $\gamma \geq 0$, $q \in \nset^*$.  Then for any $g \in \Lclass_{1, W^\gamma}$, we get
\begin{equation*}
\PE_{\pi}[|S_n|^{2q}] \leq \momentGq[q] \{\PVar[\pi](S_n)\}^q+  \Constwasspoly^{2q} (2\gamma)^{2\gamma q} \Nnorm[1,W^\gamma]{\bar{g}}^{2q}\,\sum_{u=1}^{q-1} \frac{\ConstB_{\gamma}(u,q) n^{u}}{\ratewas^{u/2} \log^{2q-u}{(1/\ratewas)}}\eqsp,
\end{equation*}
where the constant $\Constwasspoly$ is defined in \eqref{eq:const_poly_class_wasserstein}.
\end{theorem}
\begin{proof}
The proof is postponed to~\Cref{sec:proof_ros_W_log}.
\end{proof}

\begin{theorem}
\label{theo:changeofmeasure-1_wasser}
Assume \Cref{assG:kernelP_q}, \Cref{assG:kernelP_q_contractingset_m}, \Cref{ass:cost_fun}. Then for any probability measure $\xi$ on $(\Xset,\Xsigma)$, $\gamma \geq 0$, $q \in \nsets$ and function $g \in \Lclass_{1, W^\gamma}$, it holds
\begin{equation*}
     \PE_\xi\big[ \big|S_n \big|^{2q} \big] \leq 2^{2q-1} \PE_\pi\big[ \big|S_n \big|^{2q} \big] + 2^{2q-1} \Nnorm[1,W^\gamma]{\bar{g}}^{2q} \operatorname{D}^{(2)}_{q,\gamma} \eqsp,
\end{equation*}
\begin{multline*}
  \operatorname{D}^{(2)}_{q,\gamma}=
  \boundmetric^{m/2}  \vartconstwas \{ \xi(V^{1/2}) + \pi(V^{1/2}) \} \ratewas^{-1} \defEns{ \bigl(\frac{2\sqrt{2}}{\log(1/\ratewas)}\bigr)^{4q} (4q-1)! +  \frac{(8q\gamma/\rme)^{4q\gamma}}{\log(1/\ratewas)}}\eqsp.
\end{multline*}
\end{theorem}
\begin{proof}
The proof is postponed to \Cref{sec:proof-crefth-1_wass:theo:changeofmeasure-1_wasser}.
\end{proof}
We conclude with a Bernstein-type inequality. The following results extend \Cref{th:rosenthal_log_V_cor_2}
and \Cref{theo:prob_ineq_V_norm}. The proof of \Cref{th:rosenthal_log_V_cor_2} is straightforward due to the centered moment inequality. The non-stationary extension \Cref{th:rosenthal_log_V_cor_2_wasserstein_non_statio} requires more effort to obtain the correct dependence in the initial conditions (which is the same as in the $V$-geometric-ergodic case).
\begin{theorem}
\label{th:rosenthal_log_V_cor_2_wasserstein}
Assume \Cref{assG:kernelP_q}, \Cref{assG:kernelP_q_contractingset_m}, \Cref{ass:cost_fun}. Then,  for any $\gamma \geq 0$, $g \in \Lclass_{1, W^\gamma}$, and  $t \geq 0$,
\begin{equation}
\PP_{\pi}(|S_n| \geq t) \leq 2\exp\biggl\{-\frac{t^2/2}{\PVar[\pi](S_n) + \ConstJW^{1/(\gamma+3)} t^{2-1/(\gamma+3)}}\biggr\}\eqsp,
\end{equation}
where $\ConstJW$ is given by
\begin{equation}
\label{eq:const_J_n_definition_main_was}
\ConstJW = \biggl( \frac{n \ratewas^{-1/2} \{\log(1/\ratewas)\}^{-1} \Constwasspoly^{2} (2\gamma)^{4\gamma} \Nnorm[1, W^{\gamma}]{\bar{g}}^{2}}{\PVar[\pi](S_n)} \vee 1\biggr) \frac{2 (2\gamma)^{2\gamma} \Constwasspoly \Nnorm[1, W^{\gamma}]{\bar{g}}}{\log(1/\ratewas)}\eqsp.
\end{equation}
\end{theorem}
\begin{proof}
The proof is postponed to ~\Cref{sec:proof_bernstein_bound_wasserstein}.
\end{proof}

\begin{theorem}
\label{th:rosenthal_log_V_cor_2_wasserstein_non_statio}
Assume \Cref{assG:kernelP_q}, \Cref{assG:kernelP_q_contractingset_m}, \Cref{ass:cost_fun}. Then,  for any probability measure $\xi$ on $(\Xset,\Xsigma)$ satisfying $\xi(V^{1/2}) < \infty$,  $\gamma \geq 0$,  function $g \in \Lclass_{1, W^\gamma}$, and  $t \geq 0$, it holds that
\begin{align*}
&\PP_{\xi}(|S_n| \geq t) \leq
  \PP_{\pi}(|S_n| \geq t/2) \\
&\quad   +  \exp\parenthese{-\frac{\log(1/\ratewas) t^{\varpi_{\gamma}}}{2^{3+\varpi_{\gamma}} \Nnorm[1,W^\gamma]{\bar{g}}^{\varpi_{\gamma}} \varpi_{\gamma}}}\defEns{1+(-\log(\ratewas)/4)\frac{[ \boundmetric^{m/2}  \vartconstwas \{\pi(V^{1/2}) + \xi(V^{1/2})\}]^{1/2}}{\ratewas^{1/4}(1-\ratewas^{1/4})}} \\
  &\quad   +\exp\parenthese{-\frac{ (1+\gamma)\upsilon_{\gamma} t^{\varpi_{\gamma}}}{2^{5+\varpi_{\gamma}} \Nnorm[1,W^\gamma]{\bar{g}}^{\varpi_{\gamma}} \gamma }}\defEns{1+ \upsilon_{\gamma} \sup_{a \geq \rme} \{a^{4^{-1}\upsilon_{\gamma}}\log(a)\} \frac{[\boundmetric^{m/2}  \vartconstwas \{\pi(V^{1/2}) + \xi(V^{1/2})\}]^{\upsilon_{\gamma}}}{1-\ratewas^{\upsilon_{\gamma}}}}\eqsp,
\end{align*}
where $\varpi_{\gamma} = 1/(1+\gamma)$ and $\upsilon_{\gamma} = 1\wedge(2\gamma)^{-1}$\,.
\end{theorem}
\begin{proof}

The proof is postponed to \Cref{sec:proof-crefth:r_th:rosenthal_log_V_cor_2_wasserstein_non_statio}.
\end{proof}
\subsection{Related works}
\label{sec:related-works}
Moment bounds and the concentration of the additive function of Markov chains have been studied in many papers using a wealth of different techniques; the list of papers below does not claim to be exhaustive, but rather provides a selection of existing results and related theoretical tools. \cite{dedecker2015subgaussian} used coupling techniques to obtain Azuma-Hoeffding type inequality (the variance parameter is not considered) for geometrically ergodic Markov chains and bounded functions $g$ \footnote{\cite{dedecker2015subgaussian} considered separately bounded functions, which is more general than additive functionals}; this result was extended to unbounded functions by \cite{wintenberger2017exponential} but with random normalization.
 In \cite{marton1996measure}, Hoeffding inequalities are derived using Marton coupling. \cite{samson2000concentration} extends Marton's information-theoretic approach to obtain Gaussian concentration results for uniformly ergodic Markov chains and $\Phi$-mixing processes.
Probability bounds for Markov kernels that are contractive with respect to a Wasserstein distance are presented in \cite{joulin:ollivier:2010}. However, additional conditions are needed involving quantities such as \textit{granularity} and \textit{local dimension}, which are difficult to evaluate in most applications.

Using Kato's perturbation theory on the spectrum of bounded operators on Hilbert space \cite{kato:2013}, \cite{lezaud:1998} establishes Chernoff-type bounds for Markov chains on general state spaces and bounded functions $g$. This work was followed by
\cite{paulin2015concentration,fan:jiang:sun:2018:hoeffding,fan:jiang:sun:2018:bernstein}, which establish Hoeffding and Berstein probability bounds using spectral methods for Markov chains and bounded functions $g$ under the assumption that $\MK$ admits a positive absolute spectral gap.
Note also that geometric ergodicity assumptions (see \Cref{assG:kernelP_q} and \Cref{assG:kernelP_q_smallset}) do not necessarily imply the existence of a spectral gap (see \cite{kontoyiannis2012geometric}).

\cite{kontoyiannis2003spectral,kontoyiannis2005large} develop the theory of multiplicative regular Markov chains based on multiplicative drift conditions which strengthen the classical Foster-Lyapunov drift conditions.
These conditions, introduced by \cite{varadhan:1984}, play a key role in the study of large deviations of additive functions of Markov chains. Multiplicative drift conditions are generally difficult to verify; see the discussion in \cite[Section~3.1]{adamczak2015exponential}. The bounds reported in these papers are not quantitative: the bounds depend on the multiplicative Poisson equation, which amounts to solving an eigenvalue problem for an operator associated with $\MK$.

\cite{clemenccon2001moment,bertail2010sharp,Adamczak2008,adamczak2015exponential,bertail2018new,lemanczyk2020general}
use regenerative decompositions to obtain, among others, moment bounds and Bernstein inequalities under \Cref{assG:kernelP_q} and \Cref{assG:kernelP_q_smallset}.
These techniques are based on the Numellin splitting construction (see \cite{athreya1978new} and \cite{Nummelin1978AST}), which allows the sum $S_n$ to be decomposed into a random number of single-valued blocks of random length. The regenerative decomposition allows one to derive exponential inequalities for additive functionals of Markov chains from the concentration of a (random) sum of one-dependent random variables, at the cost of some very non-trivial technical work. \cite[Theorem~1]{adamczak2015exponential} provides a Bernstein-type inequality for a $V$-uniformly geometrically ergodic strongly aperiodic Markov chain and unbounded functions. \cite[Theorem~1]{lemanczyk2020general} extends the result to aperiodic Markov chains, but is restricted to bounded functions and does not give an explicit expression for constants.


Moment bounds and Bernstein-type inequalities have also been obtained under various conditions of weak dependence/mixing; see \cite{doukhan:louhichi:1999,doukhan2007probability,merlevede2011bernstein}. In general, these results are not directly comparable because the bounds depend on different types of weak dependence/mixing coefficients instead of drift conditions and local minorization/contraction conditions. However, the connections between weak dependence/mixing assumptions and $V$-geometric ergodicity are discussed in detail in \cite{adamczak2015exponential}. The results based on weak dependence / mixing methods are more appropriate for the stationary case. The extensions for the non-stationary case are less accurate than those given in our work (the way the bounds depend on the initial conditions). Finally, note that our proof for the stationary case is based on the argument developed by \cite{doukhan2007probability}, which we adapt to the Markov case. Compared to that work, we replace a covariance bound with an accurate bound for centered moments.

\input{examples}

\input{proof}

\section*{Declarations}

\textbf{Conflict of interest. } The authors have no competing interests to declare that are relevant to the content of this article.

\noindent\textbf{Data availability. } Data sharing is not applicable to this article as no datasets were generated or analyzed during the current study.

\newpage
\bibliography{biblio}
\end{document}

%% file: heading.tex
\usepackage{hyperref}
\usepackage{mathrsfs}
\usepackage{nicefrac}
\usepackage{fancyhdr}
\usepackage{setspace}
\usepackage{lastpage}
\usepackage{upgreek}
\usepackage{graphicx}	
\usepackage[lofdepth,lotdepth]{subfig}
\usepackage{amsmath,amsthm}	
\usepackage{amssymb,dsfont,bbm}	
\usepackage[ruled,vlined]{algorithm}  
\usepackage{xcolor}  
\usepackage{comment}
\usepackage{bm}
\usepackage{xargs}
\usepackage[shortlabels]{enumitem}

\usepackage{aliascnt}
\usepackage{cleveref}
\usepackage[textwidth=3cm,textsize=footnotesize]{todonotes}

\makeindex 

%% file: def.tex
\def\mrl\mathrm{L}

\def\MK{{\rm Q}}
\def\MKK{{\rm K}}

\def\PE{\mathrm{E}}
\def\msa{\mathsf{A}}
\def\Xset{\mathsf{X}}
\def\Xsigma{\mathcal{X}}
\def\Yset{\mathsf{Y}}
\def\Ysigma{\mathcal{Y}}

\def\maxind{\varkappa}
\def\distance{\mathsf{d}}
\newcommand{\ensemble}[2]{\left\{#1\,:\eqsp #2\right\}}
\newcommand{\set}[2]{\ensemble{#1}{#2}}
\DeclareMathAlphabet{\mathpzc}{OT1}{pzc}{m}{it}
\def\lyapW{\mathpzc{W}}
\def\cost{\metricc}
\def\rplus{\rset_+}

\def\vectpr{h}
\def\sigmaconst{\tilde{\sigma}}
\def\rootwas{\zeta}
\def\Xcoupling{\check{X}}

\newcommand{\1}{\mathbbm{1}}
\def\rset{\mathbb{R}}

\def\PP{\mathrm{P}}

\def\nset{\mathbb{N}}
\def\Constmainros{\mathsf{C}_0}

\def\Constroslogmom{\mathsf{C}_2}
\def\Constwasspoly{\mathsf{C}_1}

\newcommand{\absLigne}[1]{ \vert #1  \vert}
\def\eqsp{\;}
\newcommand{\Vnorm}[2][1=V]{\| #2 \|_{#1}}
\newcommand{\Nnorm}[2][1=V]{[ #2]_{#1}}
\def\mrl{\mathrm{L}}

\newcommand{\eqspp}{\ \ }
\def\Gammabf{\mathbf{\Gamma}}
\newcommand{\momentGq}[1][1=q]{\mathrm{m}_{\mathrm{G},#1}}
\def\rme{\mathrm{e}}
\newcommandx{\PVar}[1][1=]{\ensuremath{\operatorname{Var}_{#1}}}
\newcommandx{\PCov}[1][1=]{\ensuremath{\operatorname{Cov}_{#1}}}

\def\metricc{\mathsf{c}}

\def\gapindex{\mathsf{g}}
\def\maxgap{\operatorname{gap}}
\def\diam{\operatorname{diam}}
\newcommand{\card}[1]{\operatorname{card}(#1)}
\newcommand{\lr}[1]{\left( #1 \right)}
\newcommand{\lrb}[1]{\left[ #1 \right]}

\def\Lclass{\mathcal L}
\newcommand{\indi}[1]{\1_{#1}}
\newcommand{\indiacc}[1]{\1_{\{#1\}}}

\def\rmi{\mathrm{i}}
\newcommand{\expeLigne}[1]{\PE [ #1 ]}
\newcommand{\ps}[2]{\langle#1,#2 \rangle}
\def\bnu{\boldsymbol{\nu}}
\newcommand{\PEC}{\overline{\PE}}
\def\bk{{\boldsymbol k}}
\newcommand{\abs}[1]{\left\vert #1 \right\vert}
\newcommand{\ConstB}{\operatorname{B}}
\newcommand{\ConstC}{\operatorname{C}}

\newcommandx{\ConstD}[1][1={q,\lyapW}]{\operatorname{D}_{#1}}
\newcommandx{\ConstDW}[1][1={q,\lyapW}]{\mathfrak{D}_{#1}}
\newcommandx{\ConstJ}[1][1={n,W^\gamma}]{\operatorname{J}_{#1}}
\newcommandx{\ConstJW}[1][1={n,W^\gamma}]{\mathfrak{J}_{#1}}

\newcommandx{\arate}[1][1={q,\lyapW}]{{\alpha_{#1}}}
\newcommandx{\wrate}[1][1={q,\lyapW}]{{\beta_{#1}}}

\newcommandx{\ConstG}[1][1={n,\lyapW}]{\operatorname{G}_{#1}}
\newcommandx{\ConstGW}[1][1={n,\lyapW}]{\mathfrak{G}_{#1}}

\newcommandx{\ConstM}[1][1={n,\lyapW}]{\operatorname{M}_{#1}}
\newcommandx{\ConstMW}[1][1={n,\lyapW}]{\mathfrak{M}_{#1}}

\newcommandx{\boldb}[1][1={q}]{\mathsf{b}_{#1}}

\newcommand{\CKset}{\bar{\mathsf{C}}}

\def\Lclass{\mathcal L}

\def\rmd{\mathrm{d}}

\newcommand{\tvnorm}[1]{\left\Vert #1 \right\Vert_{\mathrm{TV}}}
\newcommandx{\CPE}[3][1=]{\PE_{#1}\bigl[\bigl. #2 \, \bigr| #3 \bigr]}

\newcommand{\coint}[1]{\left[#1\right)}
\newcommand{\ocint}[1]{\left(#1\right]}
\newcommand{\ooint}[1]{\left(#1\right)}
\newcommand{\ccint}[1]{\left[#1\right]}
\renewcommand{\iint}[2]{\{#1,\ldots,#2\}}

\newcommandx\dsequence[4][3=,4=]{\ensuremath{\{ (#1_{#3}, #2_{#3})~:~\#3 \in #4 \}}}
\newcommandx\sequence[3][2=,3=]
{\ifthenelse{\equal{#3}{}}{\ensuremath{\{ #1_{#2}\}}}{\ensuremath{\{ #1_{#2}: #2 \in #3 \}}}}
\newcommandx\sequenceDouble[4][3=,4=]
{\ifthenelse{\equal{#4}{}}{\ensuremath{\{ (#1_{#3},#2_{#3})\}}}{\ensuremath{\{ (#1_{#3},#2_{#3})~:~#3 \in #4 \}}}}
\newcommand{\restric}[2]{\left(#1\right)_{#2}}
\newcommandx{\sequencen}[2][2=n\in\nset]{\ensuremath{\{ #1, \eqsp #2 \}}}

\def\constlemqnpi{\zeta}
\def\minorwas{\varepsilon}
\def\mcf{\mathcal{F}}
\def\mcg{\mathcal{G}}
\newcommand{\QQ}[1][]{\ifthenelse{\equal{#1}{}}{\ensuremath{\mathbf{Q}}}{\ensuremath{\mathbf{Q}\left( #1 \right)}}}
\newcommand{\QQbf}[1][]{\ifthenelse{\equal{#1}{}}{\ensuremath{\mathbf{Q}}}{\ensuremath{\mathbf{Q}\left( #1 \right)}}}
\def\shift{\theta}
\newcommandx{\PPcoupling}[1][1={\QQ,\QQ'}]{\mathbb{P}_{#1}}
\newcommandx{\PEcoupling}[1][1={\QQ,\QQ'}]{\mathbb{E}_{#1}}
\newcommandx{\as}[1][1=\PP]{\ensuremath{#1\, -\mathrm{a.s.}}}
\newcommand{\x}{\ensuremath{x}}
\newcommandx{\rate}[1][1={q,\lyapW}]{\rho_{#1}}
\def\ratev{\rho}
\def\ratewas{\varrho}
\def\cmconstv{c}
\def\deltawas{\delta_*}
\def\vartconstwas{c_{\MKK}}
\def\ie{i.e.}
\def\boundmetric{\kappa_{\MKK}}
\def\setfunction{\mathsf{F}}
\def\normlike{\Psi}

\newtheorem{theorem}{Theorem}
\crefname{theorem}{theorem}{Theorems}
\Crefname{Theorem}{Theorem}{Theorems}

\def\iid{i.i.d.}

\newtheorem{lemma}{Lemma}
\crefname{lemma}{lemma}{lemmas}
\Crefname{Lemma}{Lemma}{Lemmas}

\newtheorem{corollary}{Corollary}
\crefname{corollary}{corollary}{corollaries}
\Crefname{Corollary}{Corollary}{Corollaries}

\newtheorem{proposition}{Proposition}
\crefname{proposition}{proposition}{propositions}
\Crefname{Proposition}{Proposition}{Propositions}

\newaliascnt{definition}{theorem}

\aliascntresetthe{definition}
\crefname{definition}{definition}{definitions}
\Crefname{Definition}{Definition}{Definitions}

\newaliascnt{definition-proposition}{theorem}

\aliascntresetthe{definition-proposition}
\crefname{definition-proposition}{definition-proposition}{definitions-propositions}
\Crefname{Definition-Proposition}{Definition-Proposition}{Definitions-Propositions}

\newtheorem{remark}{Remark}
\crefname{remark}{remark}{remarks}
\Crefname{Remark}{Remark}{Remarks}

\crefname{example}{example}{examples}
\Crefname{Example}{Example}{Examples}

\crefname{figure}{figure}{figures}
\Crefname{Figure}{Figure}{Figures}

\crefname{table}{table}{tables}
\Crefname{Table}{Table}{Tables}

\Crefname{assumption}{\textbf{H}\hspace{-1pt}}{\textbf{H}\hspace{-1pt}}
\crefname{assumption}{\textbf{H}}{\textbf{H}}

\newtheorem{assumptionSGD}{\textbf{A-SGD}\hspace{-1pt}}
\Crefname{assumptionSGD}{\textbf{A-SGD}\hspace{-1pt}}{\textbf{A-SGD}\hspace{-1pt}}
\crefname{assumptionSGD}{\textbf{A-SGD}}{\textbf{A-SGD}}

\newtheorem{assumptionpCN}{\textbf{A-pCN}\hspace{-1pt}}
\Crefname{assumptionpCN}{\textbf{A-pCN}\hspace{-1pt}}{\textbf{A-pCN}\hspace{-1pt}}
\crefname{assumptionpCN}{\textbf{A-pCN}}{\textbf{A-pCN}}

\newtheorem{assumptionC}{\textbf{C}\hspace{-1pt}}
\Crefname{assumptionC}{\textbf{C}\hspace{-1pt}}{\textbf{C}\hspace{-1pt}}
\crefname{assumptionC}{\textbf{C}}{\textbf{C}}

\newtheorem{assumption}{\textbf{A}\hspace{-2pt}}
\Crefname{assumption}{\textbf{A}\hspace{-2pt}}{\textbf{A}\hspace{-2pt}}
\crefname{assumption}{\textbf{A}\hspace{-2pt}}{\textbf{A}\hspace{-2pt}}

\newtheorem{assumptionW}{\textbf{W}\hspace{-3pt}}
\Crefname{assumptionW}{\textbf{W}\hspace{-3pt}}{\textbf{W}\hspace{-3pt}}
\crefname{assumptionW}{\textbf{W}}{\textbf{W}}

\Crefname{probleme}{\textbf{Problem}\hspace{-3pt}}{\textbf{Problem}\hspace{-3pt}}
\crefname{probleme}{\textbf{Problem}}{\textbf{Problem}}

\Crefname{assumptionG}{\textbf{G}\hspace{-4pt}}{\textbf{G}\hspace{-4pt}}
\crefname{assumptionG}{\textbf{G}}{\textbf{G}}
\def\couplingmeasure{\mathcal{C}}
\newcommandx{\wassersym}[1][1=\distance]{\mathbf{W}_{#1}}
\newcommandx{\wasser}[4][1=\distance,4=]{\mathbf{W}_{#1}^{#4}\left(#2,#3\right)}

\def\plusinfty{+\infty}
\newcommand{\txts}[1]{\textstyle #1}
\def\mff{\mathfrak{F}}
\def\scrE{\mathscr{E}}
\def\nsets{\nset^*}

\newcommand{\parenthese}[1]{\left( #1\right)}
\newcommand{\defEns}[1]{\left\{ #1\right\}}
\newcommand{\parentheseDeux}[1]{\left[ #1\right]}
\def\Sigmabf{\boldsymbol{\Sigma}}
\def\pcost{p_{\cost}}

\def\msh{\mathsf{H}}

\def\mch{\mathcal{H}}

\def\potU{\Phi_{\msh}}
\def\alphaH{\alpha_{\msh}}
\def\rhoH{\rho_{\msh}}
\def\muH{\mu_{\msh}}

\newcommandx{\normH}[2][1=]{\ifthenelse{\equal{#1}{}}{\left\Vert #2 \right\Vert_{\msh}}{\left\Vert #2 \right\Vert^{#1}_{\msh}}}
\newcommandx{\normHLigne}[2][1=]{\ifthenelse{\equal{#1}{}}{\Vert #2 \Vert_{\msh}}{\Vert #2\Vert^{#1}_{\msh}}}
\newcommandx{\normHLine}[2][1=]{\ifthenelse{\equal{#1}{}}{\Vert #2 \Vert_{\msh}}{\Vert #2\Vert^{#1}_{\msh}}}

\newcommandx{\norm}[2][1=]{\ifthenelse{\equal{#1}{}}{\left\Vert #2 \right\Vert}{\left\Vert #2 \right\Vert^{#1}}}
\newcommandx{\normLigne}[2][1=]{\ifthenelse{\equal{#1}{}}{\Vert #2 \Vert}{\Vert #2\Vert^{#1}}}
\newcommandx{\normLine}[2][1=]{\ifthenelse{\equal{#1}{}}{\Vert #2 \Vert}{\Vert #2\Vert^{#1}}}

\newcommand{\ball}[2]{\operatorname{B}(#1,#2)}
\newcommand{\ballH}[2]{\operatorname{B}_{\msh}(#1,#2)}

\def\varepsilonH{\varepsilon_{\msh}}

\def\objf{\mathrm{f}}

\def\Rsgd{R}

\def\lbprobpcn{p_1}
\def\Lippcn{\mathsf{L}}
\def\pcngamma{\gamma}
\def\Filtr{\mathcal{F}}
\def\RandSpace{\Omega}
\def\Rassumapcn{R}
\def\tstar{t^{\star}}
\def\constdriftpcnfirst{\mathsf{b}_1}
\def\constdriftpcnsecond{\mathsf{b}_2}
\def\constKone{K_1}
\def\Constexpmoment{\mathsf{D}}
\def\smallconst{\zeta}
\def\prop{z}
\def\eventA{A}
\def\eventB{B}
\def\eventC{C}
\def\gausc{c_\tau}
\def\betagaus{\beta}
\def\alphagaus{\alpha_{\tau}}
\def\fieldH{H}
\def\muY{\mu_{\Yset}}
\def\YSGD{\mathrm{Y}}
\def\MKSGD{\MK}
\def\MKKSGD{\MKK}
\def\Ccoco{C_S}
\def\thetas{\theta^{\star}}

\def\constprfirst{c_1}
\def\constprsecond{c_2}
\def\Lf{L_{\objf}}
\def\muf{\mu_{\objf}}
\def\kapf{\kappa_{\objf}}
\def\sgvarfac{\sigma_{\objf}^2}
\def\tsgvarfac{\tilde{\sigma}_{\objf}^2}
\def\tsgstdfac{\tilde{\sigma}_{\objf}}

\def\RpCN{\mathsf{R}}
\def\alphalpCN{\bar{\alpha}_{\msh}}

\def\rpCNconst{\bar{r}}

\def\LipHessianf{L_{\nabla \objf}}

\def\dims{\mathrm{d}}

%% file: introduction.tex
In this paper, we derive concentration inequalities for linear statistics of Markov chains under a geometric Foster-Lyapunov drift condition and either a minorization condition (which implies $V$-equivalent geometric ergodicity) or a local Wasserstein contraction (which implies convergence in weighted Wasserstein distance). For sums of  independent variables or martingale difference sequences, a wealth of inequalities have been developed, see, e.g.,  \cite{bercu2015concentration}. At the same time, less attention has been paid to the Markov case. In particular, most existing results are either not quantitative in terms of the dependence on the initial condition or the mixing rate of the Markov chain (see \Cref{sec:related-works} for an overview), or they can only be applied to restricted functions. In this paper, the Bernstein and Rosenthal type inequalities for Markov chains are obtained using the cumulant techniques (\cite{leonov:sirjaev:1959,saulis:statulevicius:1991}) with explicit dependence on the initial distribution of the Markov chain and its mixing rate. In the following, we briefly discuss the Bernstein and Rosenthal type inequalities for sums of independent random variables.
\par 
Let $(Y_\ell)_{\ell=1}^n$ be a sequence of independent centered random variables and set $S_n= \sum_{\ell=1}^n Y_{\ell}$. Under \emph{Bernstein's condition}, that is,
\begin{equation}
  \label{eq:Bernstein_condition}
  |\PE[Y_{\ell}^k]|  \leq (k!/2) \PVar(Y_{\ell}) c^{k-2}\eqsp, \text{ for all } \ell \in\{1, \ldots, n\} \text{ and integers } k \geq 3  \eqsp,
\end{equation}
the following two-sided Bernstein inequality holds: for any $t > 0$ and $n \in \nset$,
\begin{equation}
\label{eq:Bernstein-Rio}
\PP( |S_n| \geq t) \leq 2 \exp\left\{ - \frac{t^2/2}{\PVar(S_n) + c t} \right\} \eqsp.
\end{equation}
The condition \eqref{eq:Bernstein_condition} can be further relaxed, see e.g. \cite[Chapter~2, Theorem~2.9]{bercu2015concentration}. Bernstein's condition and inequality can be further generalized by cumulant expansion, as suggested by \cite{bentkus:1980}. Recall that the $k$-th cumulant of a random variable $Y$ is defined as
\[
\Gamma_k(Y) = \frac{1}{\rmi^k} \left.\frac{\rmd^k}{\rmd t^k}\bigl(\log \PE[\rme^{\rmi t Y}]\bigr)\right\vert_{t=0} \eqsp.
\]
It is shown in \cite{bentkus:1980} that if there exist $\upgamma \geq 1$ and $B \geq 0$ such that
\begin{equation}
\label{eq:bernstein-condition-cumulant}
|\Gamma_k(S_n)| \le (k!/2)^\upgamma \PVar(S_n) B^{k-2}  \quad \text{for all integers  $k \geq 2$} \eqsp,
\end{equation}
then for all $t \geq 0$, the following \emph{Bernstein-type bound} holds
\begin{equation}
\label{eq: Bernstein type bound}
 \PP(|S_n| \geq t) \le 2 \exp \bigg \{ -\frac{t^2/2}{\PVar(S_n) + B^{1/\upgamma} t^{(2\upgamma - 1)/\upgamma}} \bigg \} \eqsp.
\end{equation}
In \cite[Theorem 3.1]{saulis:statulevicius:1991} it is shown that the condition \eqref{eq:bernstein-condition-cumulant} is satisfied if, for example, the following generalization of \eqref{eq:Bernstein_condition} holds: there exists $K > 0 $, so that
\begin{equation}
\label{eq:bernstein_bound_generalized}
|\PE[Y_{\ell}^k]|  \leq (k!/2)^{\upgamma} \PVar(Y_\ell) K^{k-2}\eqsp, \text{ for all } \ell \in\{ 1, \ldots, n\} \text{ and  integers } k \geq 2\eqsp.
\end{equation}
Moreover, it can be shown that \eqref{eq:bernstein_bound_generalized} also holds for subexponential random variables. More precisely, define the Orlicz norm $\| Y_{\ell} \|_{\psi_{1/\upgamma}} = \inf\{t > 0: \PE [\psi_\alpha(|Y_{\ell}|/t)] \le 1\}$, where for $\alpha > 0$ we set $\psi_\alpha(x) = \rme^{x^\alpha} - 1$.
Following \cite{Adamczak2008} and \cite{Lecue2012}, we can show that if $\|\max_{\ell \in \{1,\dots,n\}} Y_\ell\|_{\psi_{1/\upgamma}} < \infty$ and $\upgamma \geq 1$, then there are some universal constants $a, b > 0$ (depending on $\upgamma$ but not on the distribution of $Y_{\ell}$) such that for any $t > 0$ and $n \in \nset$,
\begin{equation}
\label{eq:gene_Bernstein type bound}
\PP( |S_n| \geq t) \le 2 \exp \bigg \{-\frac{t^2}{a \PVar(S_n) + b  \| \max_{\ell \in \{1,\dots,n\}} Y_\ell \|_{\psi_{1/\upgamma}}^{1/\upgamma} t^{(2\upgamma -1)/\upgamma} } \bigg \}.
\end{equation}
A simple example in \cite{Adamczak2008} shows that in general $\| \max_{\ell \in \{1,\dots,n\}} Y_\ell \|_{\psi_{1/\upgamma}}$ cannot be replaced by $\max_{ \ell \in \{1,\dots,n\}} \|Y_\ell \|_{\psi_{1/\upgamma}}$ without introducing an additional factor $\log(n)$.
\par 
Moment inequalities also play an important role in studying the properties of sums of random variables. For sums of martingale difference sequences, \cite{pinelis_1994} shows the following version of Rosenthal's inequality (see \cite{Rosenthal1970}): for $q \geq 2$,
\begin{equation}
\label{eq:ros_independent_pinelis}
    \PE[|S_n|^{q}] \le C^q \bigl\{ q^{q/2} \PVar(S_n)^{q/2} +  q^q \PE[|\max_{\ell \in \{1,\dots,n\}} Y_{\ell}|^q]\bigr\} \eqsp,
\end{equation}
where $C$ is a universal constant. An important property of \eqref{eq:ros_independent_pinelis} is that if it is satisfied for every $q \geq 2$ and $ \| \max_{\ell \in \{1,\dots,n\}} Y_\ell \|_{\psi_{1}}  < \plusinfty$, then $S_n$ satisfies an Bernstein' inequality of the form \eqref{eq:gene_Bernstein type bound} with $\upgamma =1$ can be established.
Assuming that $|Y_\ell|  \le M$, using Markov's inequality one easily obtains, for example
\begin{equation}
    \PP(|S_n| \geq t) \le \rme^2 \exp \bigg\{ -  \bigg(\frac{t}{C \rme \PVar^{1/2}(S_n)}\bigg)^2 \wedge \frac{t}{C \rme M} \bigg\} \eqsp.
\end{equation}
\par 
In this paper we are concerned with the derivation of bounds of the form \eqref{eq: Bernstein type bound} and \eqref{eq:ros_independent_pinelis} for $S_n$ as an additive functional of a Markov chain. More precisely, we consider a Markov kernel $\MK$ on $(\Xset,\Xsigma)$ for which we assume that there is a unique stationary distribution $\pi$. Let $g:\Xset \mapsto \rset$ be a measurable function satisfying $\pi(\absLigne{g}) < \infty$. We set up Rosenthal- and Bernstein-type inequalities for the sums
\begin{equation}
\label{eq:def_S_n_g}
S_n = \sum_{\ell=0}^{n-1}  \{g(X_\ell) - \pi(g)\} \eqsp,
\end{equation}
where $\sequence{X}[n][\nset]$ is a Markov chain with Markov kernel $\MK$ and initial distribution $\xi$. We focus on Markov kernels that converge geometrically to the stationary distribution, either in the $V$-total variation or in the weighted Wasserstein distance. Although we do not use regeneration techniques, the results we obtain have in common with those presented in \cite{clemenccon2001moment,Adamczak2008,adamczak2015exponential,bertail2018new,lemanczyk2020general}
the goal of obtaining bounds with explicit and computable constants. This requirement is crucial to obtain, for example, deviation estimates for Markov Chain Monte Carlo or finite-time bounds for stochastic approximation algorithms with Markovian noise. The proof method we use is based on the \cite{bentkus:1980} (see also \cite{saulis:statulevicius:1991}) outlined cumulant expansion techniques and was further developed in \cite{doukhan2007probability} for weakly dependent processes. In the stationary case (when the initial condition $\xi$ is equal to $\pi$), one of the main steps of the proof is to bound centered moments associated with $\{g(X_{\ell})\}_{\ell=0}^{n-1}$ (see \Cref{sec:cumul-centr-moments} for the corresponding definitions). This connection follows from the Leonov-Shiryaev formula, see \cite{leonov:sirjaev:1959}. As far as we know, this is the first application of the Leonov-Shiryaev formula to Markov chains. 
\par 
Unlike \cite{doukhan2007probability}, which considers the case of weakly dependent sequences, sharper estimates can be obtained using the Markov property. Finally, we also treat the case of an arbitrary initial distribution $\xi$. We derive results for the non-stationary case using coupling methods (distributional in the $V$-uniform geometrically ergodic case and a Markov coupling in the weighted Wasserstein case).
\par 
The paper is organized as follows. The main results are presented in \Cref{sec:main-results}. We divide it into two parts corresponding to two sets of conditions. In \Cref{sec:geom-v-ergod} we consider $V$-geometrically ergodic Markov chains, while in \Cref{sec:geom-ergod-mark} we consider Markov chains that converge geometrically in weighted Wasserstein distances. We discuss and compare our results with the literature in \Cref{sec:related-works}. The proofs are postponed to \Cref{sec:proofs}.

\paragraph{Notations}
Let $(\Xset,\Xsigma)$ be a measurable space and $\MK$ a Markov kernel on $(\Xset,\Xsigma)$. For a measurable function $V: \Xset \to \coint{1,\infty}$, define $\mrl_V$ as a set of all measurable functions $g: \Xset \to \rset$, which $\Vnorm[V]{g}= \sup_{x \in \Xset} \{| g(x) | / V(x) \} < \infty$. The
$V$-norm of a signed measure $\xi$ on $(\Xset,\Xsigma)$ is defined by
$\Vnorm[V]{\xi} = \int_{\Xset} V(x) \rmd \abs{\xi}(x)$, for
$V : \rset \to \coint{1,\plusinfty}$, where $\abs{\xi}$ is the absolute value of $\xi$. In the case $V \equiv 1$, the $V$-norm is the total variation norm and is denoted by $\tvnorm{\cdot}$. Equivalently (see \cite[Theorem D.3.2]{douc:moulines:priouret:soulier:2018} for details), $\Vnorm[V]{\xi}$ can be defined as $\Vnorm[V]{\xi} = \sup\{ \xi(g) \, :\, \Vnorm[V]{g}\leq 1\}$.

For each probability measure $\xi$ on $(\Xset,\Xsigma)$ we denote by
$\PP_{\xi}$ (resp. $\PE_{\xi}$) the probability (resp. the expected value) on the canonical space $(\Xset^\nset,\Xsigma^{\otimes \nset})$ such that the canonical process is
$\sequence{X}[n][\nset]$ is a Markov chain with initial probability $\xi$ and Markov kernel $\MK$. By convention, we set
$\PE_{x} = \PE_{\delta_x}$ for all $x \in \Xset$.

Denote for any $q \in \coint{1,\infty}$ the $2q$-th moment of the
standard Gaussian distribution on $\rset$ by  $\momentGq[q] = (2q)! / (q! 2^{q}) = 2^{q}\Gammabf((2q+1)/2)/\uppi^{1/2}$, where $\Gammabf$ is the Gamma function.

%% file: examples.tex
\section{Applications}
\label{sec:applications}

There are a wealth of examples of $V$-uniformly geometrically ergodic Markov chains satisfying \Cref{assG:kernelP_q} and \Cref{assG:kernelP_q_smallset};  for example \citep[Chapter~15]{meyn:tweedie}, \citep{roberts:rosenthal:2004}, and \citep[Chapters~2,15]{douc:moulines:priouret:soulier:2018}. Using for example \citep{roberts:rosenthal:2004}, we can write Bernstein inequalities for additive functions of Markov Chains Monte Carlo methods in $\rset^d$. Such examples are classic and come close to those given in \cite{adamczak2015exponential}, so we prefer to focus on examples that demonstrate the results of \Cref{sec:geom-ergod-mark}. Our first example is an application to an Monte Carlo algorithm in Hilbert space. Our second example is an analysis of averaging methods for a stochastic approximation algorithm. We verify for each of these examples that the assumptions we consider in \Cref{sec:geom-ergod-mark} are satisfied and therefore the corresponding results can be applied.

\paragraph{The pre-conditioned Crank-Nicolson (pCN) algorithm}
pCN introduced in \cite{beskos:roberts:stuart:voss:2008,cotter:roberts:stuart:white:2013} is a Markov chain Monte Carlo (MCMC) method which aims at sampling
from a target distribution $\pi$ defined on a Hilbert space $\msh$ with norm $\normH{\cdot}$ and
its Borel $\sigma$-field $\mch$. This method has been applied for Bayesian inference in function
spaces and other infinite-dimensional models,
\cite{stuart:2010,cotter:roberts:stuart:white:2013,buitanh:ghattas:2014,eberle:2014,agapiou:roberts:vollmer:2018};
see also \cite{beskos:pinski:sanzserna:stuart:2011,ottobre:et:al:2016,rudolf:sprungk:2018,hosseini2018spectral} for generalizations and extensions.

Let $\muH$ be Gaussian measure $\muH$ on
$(\msh,\mch)$ with mean zero. Assume that the target distribution
$\pi$ has a density with respect to $\muH$  of the form $\rmd \pi / \rmd \muH \propto \exp(-\potU )$, for some potential function $\potU:\msh \to \rset$.  pCN then consists in defining the Markov chain $\sequence{X}[k][\nset]$ by the following recurrence:
\begin{equation}
  \label{eq:def_pCN}
  X_{k+1} = X_k \indiacc{U_{k+1} > \alphaH(X_k,Z_{k+1}) }+  \defEns{\rhoH X_k + (1-\rhoH^2)^{1/2} Z_{k+1}} \indiacc{U_{k+1} \leq \alphaH(X_k,Z_{k+1})}\eqsp.
\end{equation}
Here $\rhoH \in\ooint{0,1}$, $\sequence{Z}[k][\nset]$ and $\sequence{U}[k][\nset]$ are independent sequences of \iid~random variables with distribution $\muH$ and uniform on $\ccint{0,1}$ respectively, defined on the probability space $(\RandSpace,\Filtr,\PP)$ and for any $x,z \in \msh$,
\begin{equation}
\label{eq:accept_pcn}
\alphaH(x,z) = 1\wedge \exp\parenthese{-\potU(\rhoH x+ (1-\rhoH^2)^{1/2} z) + \potU(x)} \eqsp.
\end{equation}

It has been established in \cite{beskos:roberts:stuart:voss:2008,cotter:roberts:stuart:white:2013} that the Markov kernel associated with $\sequence{X}[k][\nset]$ is reversible with respect to $\pi$. Further,  \cite{hairer:stuart:vollmer:2012} provides the following conditions (see \cite[Assumptions~2.10-2.11]{hairer:stuart:vollmer:2012}) on $\potU$ implying that this kernel is geometrically ergodic with respect to some Wasserstein semi-metric. Denote $\ballH{x}{R} = \{z \in \msh \,: \, \normH{x-z} \leq R\}$ for any $x \in\msh$ and $R\geq 0$
\begin{assumptionpCN}
\label{assum:d-small-set-pCN}
There exist  $\alphalpCN >-\infty$, $\rpCNconst >0$, and $a \in \ooint{1/2,1}$
such that for all $x \in \msh$, $\normH{x} \geq \RpCN = \bigl(2 \rpCNconst/ (1 - \rhoH)\bigr)^{1/(1-a)}$,
$\inf_{z \in \ballH{\rhoH x}{\rpCNconst \normH{x}^a}} \alphaH(x, z) \geq \exp \left(\alphalpCN\right)$.
\end{assumptionpCN}
\begin{assumptionpCN}
\label{assum:potU-lipshitz}
$\potU$ is a Lipschitz function with Lipschitz constant $\Lippcn$, and $\exp(- \potU)$ is $\muH$-integrable.
\end{assumptionpCN}
\begin{proposition}
\label{prop:pCN}
Assume \Cref{assum:d-small-set-pCN}  and  \Cref{assum:potU-lipshitz}. Then:
\begin{itemize}
\item \Cref{assG:kernelP_q} is satisfied with $V(x)= \exp(\normH{x})$ and $\lambda$, $b$ given in \eqref{eq:drift_constants_pcn};
\item \Cref{ass:cost_fun} is satisfied with $\cost(x,x') = 1 \wedge [\normH{x-x'}/\varepsilonH]$ and $\pcost=1$, where $\varepsilonH$ is defined in \eqref{eq:varepsilon_H_def};
\item \Cref{assG:kernelP_q_contractingset_m} is satisfied  with $\CKset = \ballH{0}{\Rassumapcn} \times \ballH{0}{\Rassumapcn}$  with $\Rassumapcn = \log\bigl\{4b/(1-\lambda) - 1 \bigr\}$, $m = \lceil \log(\varepsilonH/(4\Rassumapcn))/\log\rhoH \rceil$, $\boundmetric= 1$, and $\minorwas$ defined in \eqref{eq:epsilon_pcn_def}.
\end{itemize}
\end{proposition}
\begin{proof}
The proof is postponed to \Cref{subsec:prop:pCN}.
\end{proof}

We may therefore apply our results to obtain Bernstein-type inequality for sample average $S_n(g) = n^{-1} \sum_{\ell = 0}^{n-1} g(X_\ell)$, where $g \in  \Lclass_{\beta, W}$ for some $\beta > 0$.

 \paragraph{Stochastic gradient descent (SGD) for strongly convex objective function}
 As a second example, we consider now SGD with fixed stepsize  applied to minimize a smooth objective function $\objf : \rset^{\dims} \to \rset$.  We suppose that there
 exists a measurable space $(\Yset,\Ysigma)$ endowed with a probability
 measure $\muY$ and a measurable function $\fieldH: \rset^{\dims} \times \Yset \to \rset^{\dims}$, such that
 $\nabla \objf(\theta) = \int_{\Yset} \fieldH_{\theta}(y) \rmd \muY(y)$
 for any $\theta \in\rset^{\dims}$. Based on \iid~samples $(\YSGD_k)_{k \in\nset}$ from $\muY$, the iterates of SGD with fixed stepsize $\gamma >0$ define a Markov chain given by the recursion
 \begin{equation}
   \label{eq:def_SGD}
   \theta_{k+1} = \theta_k - \gamma  \fieldH_{\theta_k}(\YSGD_{k+1}) \eqsp.
 \end{equation}
Denote by $\MKSGD_{\gamma}$ the Markov kernel associated with this recursion.
Consider the following assumption. 
\begin{assumptionSGD}
\label{ass:sgd_field}
\begin{enumerate}[wide, labelwidth=0pt, labelindent=0pt, itemsep=0mm, label=(\roman*)]
\item $\objf$ is $\muf$-strongly convex and twice continuously differentiable  with $\nabla \objf$ $\Lf$-Lipschitz: for any $\theta,\theta' \in\rset^{\dims}$, it holds $\ps{\nabla \objf(\theta) - \nabla\objf(\theta')}{\theta-\theta'} \geq \muf \norm{\theta-\theta'}^2$ and $\sup_{\theta \in \rset^{\dims}} \norm{\nabla^2 \objf(\theta)} < \Lf$, where $\norm{\nabla^2 \objf(\theta)}$ denotes the operator norm of the Hessian of $\objf$ at $\theta\in \rset^{\dims}$.
\item For $\muY$-almost every $y\in\Yset$, $\theta \mapsto \fieldH_{\theta}(y)$ is co-coercive, \ie~there exists $\Ccoco >0$ such that $\ps{\fieldH_{\theta}(y) - \fieldH_{\theta'}(y)}{\theta-\theta'} \geq \Ccoco \norm{\fieldH_{\theta}(y) - \fieldH_{\theta'}(y)}^2$ for any $\theta,\theta' \in \rset^{\dims}$
\end{enumerate}
\end{assumptionSGD}
Note that under \Cref{ass:sgd_field}, $\objf$ admits a unique minimizer denoted  $\thetas$. In the sequel we consider the following classical light-tail condition on the gradient noise; see \cite{hsu2012tail,harvey2019tight} and for equivalence between definitions.
\begin{assumptionSGD}
\label{ass:sgd_noise_exp_mom}
The gradient noise is uniformly norm sub-gaussian with variance factor $\sgvarfac < \infty$: for all $\theta \in \rset^{\dims}$
 and $t \in \rset_+$,
\[
\PP\left( \normLigne{\fieldH_{\theta}(\YSGD) - \nabla \objf(\theta)} \geq t \right) \leq  2 \exp(-t^2 /(2 \sgvarfac)) \eqsp.
\]
\end{assumptionSGD}
Denote by $\kapf= \muf \Lf/ (\muf+\Lf)$ the condition number of the function $\objf$.
\begin{proposition}
\label{theo:SGD}
Assume \Cref{ass:sgd_field} and \Cref{ass:sgd_noise_exp_mom}. Pick $\gamma \in \ocint{0,\gamma_{\objf}}$ where $\gamma_{\objf}= 1/2 \wedge \kapf/2 \wedge (\muf+\Lf)^{-1}$. Then the following statements hold:
\begin{enumerate}
\item \label{item:theo:SGD-1} \Cref{ass:cost_fun} is satisfied with $\cost(\theta,\theta') = 1 \wedge \norm{\theta-\theta'}^2$ and $\pcost=2$;
\item \label{item:theo:SGD-2} \Cref{assG:kernelP_q} is satisfied with drift function $V(\theta)= \exp(1 + \norm[2]{\theta-\theta^*}/\tsgvarfac)$, constants
\begin{equation*}
\lambda = \rme^{-\gamma\kapf/(2\tsgvarfac)}, b = \gamma\bigl(\kapf^{-1} + 2 \gamma + \kapf/(2\tsgvarfac)\bigr)\exp\left(2 + (2\tsgvarfac)^{-1} + (2\gamma \kapf+1)\kapf^{-2}\right)\eqsp,
\end{equation*}
where $\tsgvarfac = 2\sgvarfac(\rme+1)/(\rme-1)$;
\item \label{item:theo:SGD-3} \Cref{assG:kernelP_q_contractingset_m} is satisfied with $\CKset= \ball{0}{R} \times \ball{0}{\Rsgd}$, $\Rsgd = \log\bigl\{4b/(1-\lambda) - 1 \bigr\}$, $\boundmetric= 1$, $\varepsilon= 2\muf \gamma (1-\gamma \Lf/2)$, and $m = \lceil \log(4\Rsgd^2)/\log(1/(1-\minorwas)) + 1 \rceil$.
\end{enumerate}
\end{proposition}
\begin{proof}
The proof is postponed to \Cref{sec:proof_drift_sgd}.
\end{proof}
We apply our results to the Polyak-Ruppert averaged estimator $\hat{\theta}_n = n^{-1} \sum_{k=0}^{n-1} \theta_k$ of $\thetas$  \cite{ruppert1988efficient,polyak1992acceleration}.
It follows from \Cref{theo:SGD} that under \Cref{ass:sgd_field} and \Cref{ass:sgd_noise_exp_mom}, for any $\gamma \in \ocint{0,\gamma_\objf}$, $\MKSGD_{\gamma}$ has a unique invariant distribution $\pi_\gamma$
It is well-known that in general $\bar{\theta}_{\gamma} = \int_{\rset^d} \theta \rmd \pi_{\gamma}(\theta) \neq  \thetas$. However, under the same lines as \cite[Theorem~4]{dieuleveut2020bridging}, we obtain a non-asymptotic bound on  $\norm{\bar{\theta}_{\gamma}- \theta^*}$.

%
\begin{lemma}
  \label{propo:bias}
  Assume \Cref{ass:sgd_field} and \Cref{ass:sgd_noise_exp_mom}. In addition, suppose that the Hessian of $\objf$ is $\LipHessianf$-Lipschitz, \ie~for any $\theta,\theta' \in \rset^d$,
$\norm{\nabla^2\objf(\theta)-\nabla^2 \objf(\theta')} \leq \LipHessianf\norm{\theta-\theta'}$.
  Then, for any $\gamma \in \ooint{0,1/\Lf}$, it holds $    \norm{\thetas - \bar{\theta}_{\gamma}} \leq \gamma \sgvarfac /[\muf^2(1-\gamma \Lf)]$.
\end{lemma}
\begin{proof}
The proof is postponed to \Cref{sec:proof-crefpropo:bias}.
\end{proof}
\Cref{propo:bias} shows that it is enough to obtain high probability bounds on $\normLigne{\hat{\theta}_n -   \bar{\theta}_{\gamma}}$ in order to obtain non-asymptotic convergence guarantees and high probability bounds on $\normLigne{\hat{\theta}_n - \thetas}$.

%% file: proof.tex
\section{Proofs}
\label{sec:proofs}
Our proof strategy is inspired by the paper \cite{doukhan2007probability}, which is based on \cite{bentkus:1980} (see also \cite{saulis:statulevicius:1991}). Before proceeding to the proof, we first introduce some  notation and definitions.

\subsection{Cumulants and central moments}
\label{sec:cumul-centr-moments}
We begin with the definitions of cumulants and central moments of a random vector, which play an essential role in the proofs. We present these notions in a general framework. Let $(\Omega,\mcf,\PP)$ be a probability space and $W=(W_1,\dots,W_n)$ be an $n$-dimensional random vector. For any  subset $I =(i_1 , \ldots , i_k) \in \{1,\ldots,n\}^k$, $W_I$ stands for the $k$-dimensional random variable $(W_{i_1},\ldots,W_{i_k})$. Note that it is possible that some of the indices $i_j$ coincides.

\paragraph{Cumulants} Recall that the characteristic function of random vector $W$ is defined for $u \in \rset^n$ as $\varphi_W(u) = \expeLigne{\rme^{\rmi \ps{u}{W}}}$. Let $\bnu= (\nu_1,\dots, \nu_n) \in \nset^n$. Assuming that $\PE[\prod_{i=1}^n |W_i|^{\nu_i}] < \infty$, define the mixed cumulant of $W$ as
\[
\Gamma^{(\bnu)}(W)=\left.\frac{1}{\rmi^{|\bnu|}} \frac{\partial^{|\bnu|}}{\partial u_{1}^{\nu_{1}} \ldots \partial u_{n}^{\nu_{n}}} \ln \varphi_W(u) \right|_{u=0} \eqsp,
\]
where $|\bnu| = \nu_1 + \dots + \nu_n$. If $\nu_1=\dots=\nu_n=1$, we simply write $\Gamma(W)= \Gamma^{(1,\dots,1)}(W)$.
Note that for any $n$-dimensional random vector $W$, collection of indices $I =(i_1,\ldots,i_k) \in \{1,\ldots,n\}^k$ and permutation $\upsigma : I \to I$,
\begin{equation}
\label{eq:inv_permuation_cumulant}
\Gamma(W_I) = \Gamma(W_{\upsigma(I)}) \eqsp.
\end{equation}
\paragraph{Centered moments} Assume that $\PE[|W_\ell|^n] < \infty$ for $\ell \in \{1,\dots,n\}$. Set $Z_{n+1} = 1$ and define $Z_{\ell}$, for $\ell= n, \dots, 2$ by the backward recursion
 $Z_{{\ell}} = W_{\ell} Z_{{\ell+1}} - \PE[ W_{\ell} Z_{{\ell+1}}]$.
The \emph{centered moment} of $W$ is then  defined by
\begin{equation}
\label{eq:def_centered_moment}
\PEC[W] = \PEC[W_1,\dots,W_n] = \PE[W_1 Z_{2}] \eqsp.
\end{equation}
Note that $\PEC[W]$ is a scalar. Moreover, contrary to the cumulants, the centered moment of $W$ is not invariant by permutation of its component.

Let $I =(i_1,\dots,i_k) \in \{1,\ldots,n\}^k$ be an ordered subset, satisfying $i_1 \leq \dots \leq i_{k}$. Then \cite[Lemma~3]{Statul1970} or \cite[Lemma~1.1]{saulis:statulevicius:1991} allows to express the cumulant $\Gamma(W_I)$  in terms of centered moments:
\begin{equation}
\label{eq:cumulantsviamoments}
\Gamma(W_I) = \sum_{r=1}^k (-1)^{r-1} \sum\nolimits_{\bigcup_{p=1}^r I_p = I} N_r(I_1, \ldots, I_r) \prod_{p=1}^r \PEC[W_{I_p} ]\eqsp.
\end{equation}
In the formula above $N_r(I_1,\dots,I_r)$ are non-negative integers defined in \cite[Appendix~2]{saulis:statulevicius:1991} and $\sum_{\bigcup_{p=1}^r I_p = I}$ denotes the summation over all the sets $ \{I_1,\ldots,I_r\}$  such that there exists a partition $J_1,\ldots,J_r$ of $\{1,\ldots,k\}$ satisfying for any $\ell \in \{1,\ldots,r\}$, $I_\ell = (i_{j_1},\ldots,i_{j_{n_\ell}})$ with $J_{\ell} = \{j_1,\ldots,j_{n_{\ell}} \}$, $j_1 < \ldots < j_{n_{\ell}}$. It is shown in  \cite[Eq. 4.43]{saulis:statulevicius:1991}
that
\begin{equation}
\label{eq:formula-summation}
\sum_{r=1}^k \sum\nolimits_{\bigcup_{p=1}^r I_p = I} N_r(I_1, \ldots, I_r)  = (k-1)! \eqsp.
\end{equation}
We now specify these definitions to Markov chain to derive  an expression for the $q$-th moment of the random variables $S_n$ defined by \eqref{eq:def_S_n_g} for an integer $n \geq 1$ and a measurable function $g$, integrable with respect to $\pi$. This expression is the cornerstone of our approach. Define
\begin{equation}
\label{eq:def_Y_ell}
\bar{g} = g - \pi(g) \eqsp, \quad   Y_{\ell} = \bar{g}(X_\ell) \eqsp,
\end{equation}
so that $S_n = \sum_{\ell =0}^{n-1} Y_{\ell}$.
Assume that for $\ell \in \{0, \ldots, n-1\}$, $\PE_{\pi}[\abs{Y_{\ell}}^{2q}] < \infty$  so that $\PE_{\pi}[|S_n|^{2q}] < \infty$. For $u \in \{1,\ldots,q\}$ and $(k_1, \ldots , k_u) \in \{1, \ldots, 2q \}^u$, let $\bk_u = (k_1, \ldots, k_u)$, $|\bk_u| = \sum_{p=1}^u k_p$ and $\bk_u! = \prod_{p=1}^u k_p!$. Then, by the Leonov-Shiraev formula \cite[Eq.~1.53]{saulis:statulevicius:1991}
\begin{equation}
\label{eq:leonov-shiraev}
    \PE_\pi[|S_n|^{2q}] = \sum_{u=1}^{2q} \frac{1}{u!} \sum_{|\bk_u| = 2q} \frac{(2q)!}{{\boldsymbol k_u}!} \prod_{p=1}^u \Gamma_{\pi, k_p}(S_n) \eqsp,
\end{equation}
where $\Gamma_{\pi, k}(S_n)$ denotes the $k$-th order cumulant of $S_n$ under the stationary probability $\PP_\pi$. Using \cite[Eq.~1.47]{saulis:statulevicius:1991} for any $k \in \{1,\ldots,2q\}$, we may express $\Gamma_{\pi,k}(S_n)$ as the sum of cumulants over all the $k$-tuple $(Y_{t_1},\dots,Y_{t_k})$ under $\PP_\pi$: $\Gamma_{\pi,k}(S_n)=  \sum_{0 \le t_1, \ldots, t_k \le n-1} \Gamma_{\pi}(Y_{t_1}, \ldots, Y_{t_k})$ and using \eqref{eq:inv_permuation_cumulant}, we get
\begin{equation}
  \label{eq:expression_gamma_pi_k_S_n}
\Gamma_{\pi,k}(S_n)= (k!) \sum_{0 \le t_1\leq  \cdots \leq t_k \le n-1} \Gamma_{\pi}(Y_{t_1}, \ldots, Y_{t_k}) \eqsp.
\end{equation}
Since for $i \in \{1,\ldots,n-1\}$, $\Gamma_{\pi}(Y_\ell) = \PE_{\pi}[Y_\ell] = 0$, terms in \eqref{eq:leonov-shiraev} corresponding to indices $u \geq q+1$ vanishes. Moreover, the only non-zero term in the sum corresponding to $u = q$ corresponds to the multi-index $\boldsymbol k_q = (2,\ldots,2)$. Thus, with the definition of $\momentGq[q]$, the expression \eqref{eq:leonov-shiraev} simplifies to
\begin{equation}
\label{eq:qmomentexpansion}
\PE_\pi[|S_n|^{2q}] = \momentGq[q] \{ \PVar[\pi](S_n) \}^q + \sum_{u=1}^{q-1} \frac{1}{u!} \sum_{|\bk_u| = 2q} \frac{(2q)!}{{\boldsymbol k_u}!} \prod_{p=1}^u \Gamma_{\pi, k_p}(S_n) \eqsp.
\end{equation}
Now we aim to bound $\PE_\pi[|S_n|^{2q}]$ based on the expression \eqref{eq:qmomentexpansion}. To do it, we first establish  bounds on centered moments of the random vectors $(Y_{t_1},\ldots,Y_{t_k})$ with $(t_1,\ldots,t_k) \in\{0,\ldots,n-1\}^k$, and then deduce bounds for the $k$-the order cumulants $\Gamma_{\pi,k}(S_n)$ using the relations \eqref{eq:cumulantsviamoments} and \eqref{eq:expression_gamma_pi_k_S_n}.
\par 
The first step of the proof relies on some lemmas on centered moments associated with the sequence of measurable functions $\{h_{\ell}\}_{\ell=1}^{k}$ satisfying $\pi(|h_{\ell}|^{k})<\infty$. We consider the centered moments for 
\[
(h_1(X_{t_1}), \ldots, h_k(X_{t_k})), \quad (t_1,\dots,t_k) \in \{0,\dots,n-1\}^{k}, \quad t_1 \leq \dots \leq t_k\eqsp.
\]
Define the sequence $\{Z^h_{\ell}\}_{\ell=2}^{k+1}$:
 \begin{equation}
 \label{eq:z_ell_definition}
 Z^h_{k+1} = 1\eqsp, \quad Z^h_{\ell} = h_\ell(X_{t_\ell}) Z^h_{\ell+1} - \PE_\pi[ h_\ell(X_{t_\ell}) Z^h_{\ell+1}]\eqsp,\quad  \ell= 2, \dots, k\eqsp.
 \end{equation}
 The dependence of the sequence $\{Z^h_\ell\}_{\ell=2}^{k+1}$ on $\{h_\ell \}_{\ell=1}^k$ and $\{t_\ell\}_{\ell=1}^k$ is implicit.
\begin{lemma}
\label{lem:centred_moments_markov_property}
Assume that $\MK$ has a unique invariant distribution $\pi$. Let $k \in \nset$ and $\{h_i \}_{i=1}^k$ be a family of real measurable functions on $\Xset$
such that $\pi(|h_i|^k) < \infty$. Then, for any $(t_1,\dots,t_{k}) \in \{0,\ldots,n-1\}^{k}$, $t_1 \leq \dots\leq t_{k}$, it holds
\begin{equation*}
\PEC_\pi[h_1(X_{t_1}), \ldots, h_{k}(X_{t_{k}})] = \PEC_\pi[h_1(X_{t_1}), \ldots , h_{k-1}(X_{t_{k-1}}) \tilde h_{k}(X_{t_{k-1}})] \eqsp,
\end{equation*}
where $\tilde h_{k}(x) = \MK^{t_{k} - t_{k-1}} h_{k}(x) - \pi(h_{k})$.
\end{lemma}
\begin{proof}
  Let $(t_1,\dots,t_{k}) \in \{0,\ldots,n-1\}^{k}$, $t_1 \leq \dots\leq t_{k}$.
Set $\mff_{k-1} = \sigma\{X_{t_1}, \ldots, X_{t_{k-1}}\}$. Using the definition \eqref{eq:def_centered_moment} and the tower property, we obtain
\begin{equation}
  \label{eq:def_centered_moment_lem}
     \PEC_{\pi}[h_1(X_{t_1}), \ldots, h_{k}(X_{t_{k}}) ] = \PE_\pi[h_1(X_{t_1}) Z^h_{2}] =\PE_\pi[h_1(X_{t_1}) \PE[Z^h_{2}|\mff_{k-1}]] \eqsp.
\end{equation}
It remains to establish that 
\[
\PE_\pi[h_1(X_{t_1}) \PE[Z^h_{2}|\mff_{k-1}]] = \PEC_\pi[h_1(X_{t_1}), \ldots , h_{k-1}(X_{t_{k-1}}) \tilde h_{k}(X_{t_{k-1}})]\eqsp.
\]
Taking the conditional expectation with respect to $\mff_{k-1}$ in the definition \eqref{eq:z_ell_definition}  
of $\{Z^h_{\ell}\}_{\ell=2}^{k+1}$, we get for any $\ell \in \{1, \ldots, k-2\}$,
 setting  $\tilde Z^h_{\ell+1} = \CPE[\pi]{Z^h_{\ell+1}}{\mathfrak F_{k-1}}$,
\begin{equation}
\label{eq: tilde z_0}
    \tilde Z^h_\ell  = h_\ell(X_{t_\ell}) \tilde Z^h_{\ell+1} - \PE_\pi[ h_\ell(X_{t_\ell}) \tilde Z^h_{\ell+1}] \eqsp.
\end{equation}
For $\ell = k-1$ taking into account that $Z^h_{k+1} = 1$ and $\PE[h_{k}(X_{t_{k}}) - \pi(h_k) | \mathfrak F_{k-1}] = \tilde h_k(X_{t_{k-1}})$,
\begin{align}
\nonumber
\tilde Z^h_{k-1} = \PE[Z^h_{k-1}|\mff_{k-1}] &= h_{k-1}(X_{t_{k-1}}) \CPE[]{h_{k}(X_{t_{k}}) - \pi(h_k)}{\mathfrak F_{k-1}} \\
\nonumber
& \phantom{h_{k-1}(X_{t_{k-1}}) xxxxxx}
- \PE_\pi[ h_{k-1}(X_{t_{k-1}}) \CPE[]{h_{k}(X_{t_{k}}) - \pi(h_k)}{\mathfrak F_{k-1}}] \\
& = h_{k-1}(X_{t_{k-1}}) \tilde h_{k}(X_{t_{k-1}})  - \PE_\pi[ h_{k-1}(X_{t_{k-1}}) \tilde h_{k}(X_{t_{k-1}})]
\label{eq: tilde z_1}
\eqsp.
\end{align}
By \eqref{eq: tilde z_0}-\eqref{eq: tilde z_1} and the definition of centred moments \eqref{eq:def_centered_moment}, we get setting $\tilde{Z}^h_k=1$, that
\begin{equation*}
\PEC_\pi[h_1(X_{t_1}), \ldots , h_{k-1}(X_{t_{k-1}}) \tilde h_{k}(X_{t_{k-1}})] = \PE_\pi[h_1(X_{t_1}) \tilde Z^h_{2}] \eqsp.
\end{equation*}
Combining this result with \eqref{eq:def_centered_moment_lem} completes the proof. 
\end{proof}

\subsection{Upper-bounding cumulants from central moments}
\label{sec:cumulants_upper_bound_geom_ergodicity}
Throughout this section we assume that $\MK$ has a unique invariant distribution $\pi$ and fix some integer $q \geq 1$. Let $\lyapW: \Xset \to \coint{1,\infty}$ be a measurable function and $\setfunction_{\lyapW}$ be a set of
measurable functions. Finally, let $\normlike_{\lyapW}$ be a non-negative functional defined on $\setfunction_{\lyapW}$. The objective is to compute a bound for $\Gamma_{\pi, k}(S_n)$ for $k \leq 2q$, where $S_n= \sum_{\ell=1}^n \bar{g}(X_\ell)$ and $g \in \setfunction_{\lyapW}$.

Starting from \eqref{eq:expression_gamma_pi_k_S_n}, we establish a bound on $\Gamma_\pi(Y_I)$ setting $Y_I = (Y_{t_1},\ldots,Y_{t_k})$ with $ (t_1, \ldots, t_k)\in \{0,\ldots,n-1\}^{k}$, $t_1 \leq \dots\leq t_{k}$.
To this end, we use \eqref{eq:cumulantsviamoments} and the following assumption.
\begin{assumptionW}[$q,\lyapW,\normlike$]
\label{assum:central_moments_bound}
There exist $\arate,\ConstD \geq 0$ $\rate \in \coint{0,1}$ such that for any $k \in \{2,\dots,2q\}$, $k-$tuple $I = (t_1,\dots,t_k) \in \{0,\dots,n-1\}^{k}$, $t_1 \leq \dots \leq t_k$, and any family of measurable functions $\{h_{\ell}\}_{\ell=1}^{k} \subset \setfunction_{\lyapW}$
\begin{equation}
\label{eq:centred_moments_generic_assumption}
|\PEC_\pi[h_1(X_{t_1}), \ldots, h_k(X_{t_k})]| \leq \ConstD^{k} (k!)^{\arate}\left\{\prod_{j=1}^{k} \normlike_{\lyapW}(h_j)\right\}\rho_{q,\lyapW}^{\maxgap(I)}\eqsp,
\end{equation}
where $\maxgap(I) := \max_{j \in \{1,\dots,k-1\}}[t_{j+1}-t_{j}]$
\end{assumptionW}
The proof of the following results can be adapted from \cite{doukhan2007probability}. 
\begin{lemma}
\label{lem:cumulant_bounds_generic}
Assume \Cref{assum:central_moments_bound}($q,\lyapW,\normlike$). Then, for any  $k \in \{2,\ldots,2q\}$,  $I=(t_1,\dots,t_{k}) \in \{0,\ldots,n-1\}^{k}$, $t_1 \leq \dots\leq t_{k}$, and function $g \in \setfunction_{\lyapW}$, it holds that
\begin{equation}
  |\Gamma_{\pi}(Y_I)| \leq \ConstD^{k} (k!)^{\arate} \bigl\{\normlike_{\lyapW}(\bar{g})\bigr\}^{k} \sum_{r=1}^{k}\sum_{\bigcup_{\ell=1}^r I_{\ell} = I} N_r(I_1, \ldots, I_r)    \rho_{q,\lyapW}^{\sum_{\ell=1}^{r}\maxgap(I_{\ell})}  \eqsp.
\end{equation}
\end{lemma}
\begin{proof}
Eq.~\eqref{eq:cumulantsviamoments}
implies that
\begin{equation}
  \label{eq:1:lem:cumlants_generic}
|\Gamma_{\pi}(Y_I)| \leq \sum_{r=1}^k \sum\nolimits_{\bigcup_{\ell=1}^r I_{\ell} = I} N_r(I_1, \ldots, I_r) \prod_{\ell=1}^r \bigl|\PEC_{\pi}[Y_{I_{\ell}} ]\bigr|\eqsp.
\end{equation}
Using \cref{assum:central_moments_bound}($q,\lyapW,\normlike$), we obtain
$|\PEC_\pi[Y_{I_{\ell}}] | \leq  \ConstD^{\card{I_{\ell}}} (\card{I_{\ell}}!)^{\arate} \rate^{\maxgap(I_{\ell})}\{\normlike_{\lyapW}(\bar{g})\}^{\card{I_\ell}}$.
The proof is completed by plugging this bound in \eqref{eq:1:lem:cumlants_generic} and using $\prod_{\ell=1}^r \card{I_\ell}! \leq k!$.
\end{proof}

\begin{lemma}
\label{lem:cumulant_bounds_final_generic}
Assume \Cref{assum:central_moments_bound}$(q,\lyapW,\normlike)$. Then,  for any $k \in \{2,\dots,2q\}$,and function $g \in \setfunction_{\lyapW}$,
\begin{equation}
\label{eq: Gpik estimate}
 |\Gamma_{\pi, k}(S_n)| \leq  n \rate^{-1} \log^{1-k}\{1/\rate\} \ConstD^{k}  \bigl\{\normlike_{\lyapW}(\bar{g})\bigr\}^{k} (k!)^{3+\arate} \eqsp.
\end{equation}
\end{lemma}
\begin{proof}
By \eqref{eq:expression_gamma_pi_k_S_n}, we get $ \bigl|\Gamma_{\pi,k}(S_n)\bigr| \leq k! \sum_{0\leq t_1 \leq \cdots \leq t_k \leq n-1} \bigl|\Gamma_{\pi}(Y_{t_1}, \ldots, Y_{t_k})\bigr|$.
Denoting $I(n,k) = \{ (t_1, \ldots, t_k): 0 \le t_1 \le \ldots \le t_k \le n-1  \}$ and using \Cref{lem:cumulant_bounds_generic}, we get
\begin{equation}
\label{eq:bound_cumulant_S_n_generic}
|\Gamma_{\pi, k}(S_n)|
  \leq (k!)^{\arate + 1} \ConstD^k \bigl\{\normlike_{\lyapW}(\bar{g})\bigr\}^{k} 
 \sum_{I \in I(n,k)} \sum_{r=1}^k \sum_{\bigcup_{\ell=1}^r I_\ell = I} N_r(I_1, \ldots, I_r) \rate^{\sum_{\ell=1}^r \maxgap(I_{\ell})}
\end{equation}
For any $m \in \{2,\dots,k\}$, $\gapindex \in \{0,\dots,m-1\}$, we set $I(n,k,m)=\bigcup_{\gapindex=0}^{n-1} I(n,k,m,\gapindex)$, where
\begin{equation}
  \label{eq:def_I_n_k_m_r_0_generic}
        I(n,k,m,\gapindex): = \{ (t_1, \ldots, t_k) \in I(n,k): t_{m} - t_{m-1} = \gapindex = \max_{i\in\{2,\ldots,k\}}(t_{i} - t_{i-1}) \} \eqsp.
\end{equation}
Then $I(n,k) = \bigcup_{m=2}^{k} I(n,k,m)$. Note that for any $I = (t_1,\ldots,t_k) \in I(n,k,m,\gapindex)$ and any partition $ \{I_1,\ldots,I_r\}$, $\bigcup_{\ell=1}^r I_\ell = I$, it holds that $\sum_{\ell=1}^{r}\maxgap(I_{\ell}) \geq \gapindex=\maxgap(I)$.
Using \eqref{eq:formula-summation}, we get
\begin{multline*}
\sum_{I \in I(n,k)} \sum_{r=1}^k \sum_{\bigcup_{\ell=1}^r I_\ell = I} N_r(I_1, \ldots, I_r) \rate^{\sum_{\ell=1}^r \maxgap(I_{\ell})} \leq k! \sum_{\gapindex=0}^{n-1} \rate^{\gapindex} \sum_{m=2}^k \card{I(n,k,m,\gapindex)}\eqsp.
\end{multline*}
Plugging this bound in \eqref{eq:bound_cumulant_S_n_generic} yields
\begin{equation}
\label{lem:cumulants:1:generic}
 |\Gamma_{\pi,k}(S_n)|
\le (k!)^{\arate + 2} \ConstD^k \bigl\{\normlike_{\lyapW}(\bar{g})\bigr\}^{k} \sum_{m=2}^{k}\sum_{\gapindex=0}^{n-1} \rate^{\gapindex} \card{I(n,k,m,\gapindex)} \eqsp.
\end{equation}
For any $(t_1,\ldots,t_k) \in I(n,k,m,\gapindex)$, using \eqref{eq:def_I_n_k_m_r_0_generic}, we get $t_{m-1} \in \{0,\dots,n-1-\gapindex\}$ and $\max_{i\neq m}\{t_{i}-t_{i-1}\} \leq \gapindex$.
Therefore, $\card{I(n,k,m,\gapindex)} \leq (n-1-\gapindex)(\gapindex+1)^{k-2}$. Combining \eqref{lem:cumulants:1:generic} and \eqref{eq:bound_cumulant_S_n_generic} completes the proof together with
\begin{align*}
\sum_{\gapindex=0}^{n-1} \rate^{\gapindex} (\gapindex+1)^{k-2} \leq \frac{1}{\rate} \int_{0}^n \rate^{s+1}(s+1)^{k-1} \rmd s  \leq \frac{1}{\rate}\biggl(\frac{1}{\log{1/\rate}}\biggr)^{k-1}(k-2)!\eqsp.
\end{align*}
\end{proof}

Based on \Cref{lem:cumulant_bounds_final_generic} and \eqref{eq:qmomentexpansion}, we can now establish a bound on $\PE_{\pi}[|S_n|^{2q}]$.

\begin{lemma}\label{th:generic_th_rosenthal}
Assume \Cref{assum:central_moments_bound}($q,\lyapW,\normlike$). Then, for any $g \in \setfunction_{\lyapW}$, it holds that
\begin{equation*}
\PE_\pi[|S_n|^{2q}] \leq \momentGq[q] \{\PVar[\pi](S_n)\}^q +  \ConstC^{2q}_{q,\lyapW} \ConstD^{2q}  \sum_{u=1}^{q-1} \ConstB_{\arate}(u,q) \biggl(\frac{n \log{(1/\rate)}}{\rate}\biggr)^{u}\eqsp,
\end{equation*}
where $\ConstC_{q,\lyapW}= \{ \log(1/\rate) \}^{-1}$ and $\ConstB_{\arate}(u,q)$ is defined in~\eqref{eq: B_u_q_def_new}.
\end{lemma}
\begin{proof}
Denote the second term in the right-hand side of~\eqref{eq:qmomentexpansion} by $R_{\pi, n}$, i.e.,
\begin{equation*}
    R_{\pi, n} = \sum_{u=1}^{q-1} \frac{1}{u!} \sum_{|\bk_u| = 2q} \frac{(2q)!}{{\boldsymbol k_u}!} \prod_{p=1}^u \Gamma_{\pi, k_p}(S_n) = \sum_{u=1}^{q-1} \frac{1}{u!} \sum_{\bk_u \in \scrE_{u,q}} \frac{(2q)!}{{\boldsymbol k_u}!} \prod_{p=1}^u \Gamma_{\pi, k_p}(S_n) \eqsp,
\end{equation*}
where we have used in the last equality $\Gamma_{\pi,1}(S_n) = 0$ and  $\scrE_{u,q} = \{\bk_u= (k_1,\ldots,k_u) \in \nset^u \, : \, \sum_{i=1}^u k_i = 2q\, ,\, \min_{i \in\iint{1}{u}} k_i \geq 2\}$. Applying \Cref{lem:cumulant_bounds_final_generic}, we get
\begin{equation*}
|R_{\pi,n}|  \leq  \ConstD^{2q} \bigl\{\normlike_{\lyapW}(\bar{g})\bigr\}^{2q} \sum_{u=1}^{q-1} n^{u} \frac{(2q)!}{u!}\biggl(\frac{1}{\rate}\biggr)^{u} \biggl(\frac{1}{\log{1/\rate}}\biggr)^{2q-u} \sum_{\boldsymbol{k}_u \in \scrE_{u,q}} \prod_{i=1}^u (k_i!)^{\arate+2}\eqsp,
\end{equation*}
which completes the proof by the definition of $\ConstB_{\arate}(u,q)$.
\end{proof}

\subsection{Proof of \Cref{th:rosenthal_V_q}}
\label{sec:proof-ros_v_q}
In this section, we show that assumptions \Cref{assG:kernelP_q} and \Cref{assG:kernelP_q_smallset} imply \Cref{assum:central_moments_bound}$(q,V^{1/(2q)},\Vnorm[V^{1/2q}]{\cdot})$ for any $q \in \nset$. Applying \Cref{lem:rate_UGE}, we get that for any $x \in \Xset$, $0 < \alpha  \leq 1$, and $n \in \nset$,
\begin{equation}
\label{eq:V-geometric-better-rate}
\Vnorm[V^\alpha]{\MK^n(x, \cdot) - \pi}
 \leq 2 \{\cmconstv \ratev^n \pi(V)   V(x) \}^{\alpha}
\eqsp.
\end{equation}

\begin{lemma}
\label{lem:centered_moments_Z_old}
Assume \Cref{assG:kernelP_q}, \Cref{assG:kernelP_q_smallset}, and let $s \in \nset$. Then for any $k \in \{1,\ldots,s\}$,  $(t_1,\dots,t_{k}) \in \{0,\ldots,n-1\}^{k}$, $t_1 \leq \dots\leq t_{k}$,   $(p_1,\dots,p_k) \in \nset^k$ satisfying  $p_i \geq 1$ for $i \in \{1,\dots,k\}$ and $\sum_{i=1}^k p_i \le s$, and functions $\{h_{\ell}\}_{\ell=1}^{k} \subset  \mrl_{V^{p_\ell/s}}$, we have 
\begin{multline}
\label{eq:centred_moments_one_Z_0}
  |\PEC_\pi[h_1(X_{t_1}), \ldots, h_k(X_{t_k})] |
   \\\leq 2^{k-1} \{\cmconstv \pi(V)\}^{\sum_{\ell=1}^k\sum_{j=\ell}^k p_{j}/s} \ratev^{\sum_{j=2}^{k}(t_j - t_{j-1})p_{j}/s} \prod_{\ell = 1}^k   \|h_\ell\|_{V^{p_{\ell}/s}} \eqsp.
\end{multline}
\end{lemma}
\begin{proof}
 The proof is based on induction on $k \in \{1,\dots,s\}$.  For $k = 1$, we get :
\begin{equation*}
|\PEC_\pi[h_1(X_{t_1})]| \leq \pi(V^{p_1/s}) \|h_1\|_{V^{p_{1}/s}} \overset{(a)}{\leq} (\cmconstv \pi(V))^{p_{1}/s} \|h_1\|_{V^{p_{1}/s}}\eqsp,
\end{equation*}
where in (a) we used Jensen's inequality and the fact that $\cmconstv \geq 1$.
Assume that~\eqref{eq:centred_moments_one_Z_0} holds for some $k \in \{1,\dots,s-1\}$. Let $t_1 \leq \dots\leq t_{k+1}$,   $(p_1,\dots,p_{k+1})\in \nset^{k+1}$ satisfying  $p_i \geq 1$ for $i \in \{1,\dots,k+1\}$ and $\sum_{i=1}^{k+1} p_i \le s$, and functions $\{h_{\ell}\}_{\ell=1}^{k+1} \in  \mrl_{V^{p_\ell/s}}^{k+1}$. By \Cref{lem:centred_moments_markov_property},
\begin{equation*}
   \PEC_\pi[h_1(X_{t_1}), \ldots, h_k(X_{t_k}), h_{k+1}(X_{t_{k+1}})] = \PEC_\pi[h_1(X_{t_1}), \ldots , h_k(X_{t_k}) \tilde h_{k+1}(X_{t_k})],
  \end{equation*}
  where $\tilde h_{k+1}(x) = \MK^{t_{k+1} - t_k} h_{k+1}(x) - \pi(h_{k+1})$.
  Since $h_{k+1} \in \mrl_{V^{p_{k+1}/s}}$, we apply~\eqref{eq:V-geometric-better-rate} with $\alpha = p_{k+1}/ s$. Thus,
  $$
  |\tilde h_{k+1}(x)| \leq 2(\cmconstv \pi(V))^{p_{k+1}/s} \ratev^{(t_{k+1} - t_k)p_{k+1}/s} V^{p_{k+1}/s}(x) \|h_{k+1}\|_{V^{p_{k+1}/s}}\eqsp.
  $$
  Hence, using that $h_{k} \in \mrl_{V^{p_{k}/s}}$, we obtain
\begin{equation*}
\|h_k \tilde h_{k+1}\|_{V^{(p_k+p_{k+1})/s}}  \leq
 2 \bigl(\cmconstv \pi(V)\bigr)^{p_{k+1}/s}\ratev^{(t_{k+1} - t_k)p_{k+1}/s}
 \|h_{k}\|_{V^{p_{k}/s}}
 \|h_{k+1}\|_{V^{p_{k+1}/s}} \eqsp.
\end{equation*}
Then, applying the induction hypothesis to $\bar{h}_i = h_i \in \mrl_{V^{\bar{p}_i/s}}$, $\bar{p}_i = p_i$, $i \in\{1,\ldots,k-1\}$, $\bar{h}_k = h_k \tilde h_{k+1} \in \mrl_{V^{\bar{p}_k/s}}$, $\bar{p}_k = p_k + p_{k+1}$ completes the proof.
\end{proof}
\begin{corollary}
\label{coro:centered_moments_V_class}
Assume \Cref{assG:kernelP_q}, \Cref{assG:kernelP_q_smallset}. Then for any $q \in \nset$, \Cref{assum:central_moments_bound}($q,V^{1/(2q)},\Vnorm[V^{1/2q}]{\cdot}$) is satisfied with $\ConstD[q,V^{1/(2q)}] = 2 \cmconstv \pi(V)$, $\arate[q,V^{1/(2q)}] = 0$, and $\rho_{q,V^{1/(2q)}} = \ratev^{1/2}$, where $\ratev$ is defined in \eqref{eq:V-geometric-coupling-general}.
\end{corollary}
\begin{proof}
  Let $k \in \{1,\ldots,q\}$ and
 $(t_1,\dots,t_{k}) \in \{0,\ldots,n-1\}^{k}$, $t_1 \leq \dots\leq t_{k}$.
  Define $\maxind \in \{2,\ldots,k\}$ such that $t_{\maxind} - t_{\maxind-1} = \max_{j \in \{2,\ldots,k\}} [t_j - t_{j-1}]$. For $i \in \{1,\ldots,k\} \setminus \{\maxind\}$, we set $p_{i} = 1$, and put $p_{\maxind} = q$.
  Now we apply \Cref{lem:centered_moments_Z_old} with the mentioned choice of $(p_1,\dots, p_k)$ and $s = 2q$.  Note that $ h_i \in \mrl_{V^{p_i/(2q)}}$ for any $i \in \{1,\ldots,k\}$ and $\sum_{i=1}^k p_i \leq 2q$. Moreover, $\|h_{\maxind}\|_{V^{1/2}} \leq \|h_{\maxind}\|_{V^{1/(2q)}}$ since $q \geq 1$ and $V(x) > 1$. Therefore, the application of  \Cref{lem:centered_moments_Z_old} concludes the proof.
\end{proof}

\begin{proof}[Proof of \Cref{th:rosenthal_V_q}] The proof now follows from \Cref{th:generic_th_rosenthal} combined with \Cref{coro:centered_moments_V_class}.
\end{proof}

\subsection{Proof of \Cref{theo:changeofmeasure} and \Cref{theo:changeofmeasure-1}}
\label{sec:non-stationary-extension}
To go from an arbitrary initialization to the stationary case, we  use a distributional coupling argument (see \cite{thorisson1986maximal} and \cite[Chapter~19]{douc:moulines:priouret:soulier:2018}).
Denote by $\QQ$ and $\QQ'$ two
probability measures on the canonical space  $(\Xset^\nset,\Xsigma^{\otimes \nset})$.
Fix $x^*\in \Xset$ and denote $\bar{\nset} = \nset \cup \{\infty\}$. For any $\Xset$-valued stochastic process $\Xcoupling=\sequence{\Xcoupling}[n][\nset]$ and any
$\bar \nset$-valued random variable $T$, define the $\Xset$-valued stochastic process $\shift_T \Xcoupling$
by $\shift_T \Xcoupling=\sequencen{\Xcoupling_{T+k}}[k \in \nset]$ on $\{T<\infty\}$ and $\shift_T \Xcoupling=(\x^*,x^*,
x^*,\ldots)$ on $\{T=\infty\}$.
Let $\Xcoupling=\sequence{\Xcoupling}[n][\nset]$, $\Xcoupling'=\sequence{\Xcoupling'}[n][\nset]$ be  $\Xset$-valued stochastic
processes and $T$, $T'$ be $\bar \nset$-valued random variables defined on the
probability space $(\Omega,\mcf,\PPcoupling)$.

We say that $\left\{(\Omega,\mcf,\PPcoupling,\Xcoupling,T,\Xcoupling',T') \right\}$ is a distributional coupling of $(\QQ,\QQ')$ if
\begin{description}
  \item[DC-1] for all $\msa \in \Xsigma^{\otimes \nset}$, $\PPcoupling(\Xcoupling \in \msa)=\QQ(\msa)$ and $\PPcoupling(\Xcoupling' \in \msa)=\QQ'(\msa)$,
  \item[DC-2] $(\shift_T \Xcoupling, T)$ and $(\shift_{T'} \Xcoupling', T')$ have the same distribution  under $\PPcoupling$.
\end{description}
The random variables $T$ and $T'$ are called the coupling times.  The distributional coupling is said to be \emph{successful} if $\PPcoupling(T<\infty)=1$.

For any measure $\mu$ on $(\Xset^\nset,\Xsigma^{\otimes \nset})$ and any
$\sigma$-field $\mcg \subset \Xsigma^{\otimes \nset}$, we denote by $\restric{\mu}{\mcg}$ the
restriction of the measure $\mu$ to $\mcg$. Moreover, for all $n \in \nset$, define the
$\sigma$-field $\mcg_n=\set{\shift_n^{-1}(\msa)}{\msa \in \Xsigma^{\otimes \nset}}$.
  A distributional coupling $(\Xcoupling,\Xcoupling')$ of $(\QQ,\QQ')$ with coupling times $(T,T')$ is maximal if for
  all $n\in\nset$,
$$
\tvnorm{\restric{\QQ}{\mcg_n}-\restric{\QQ'}{\mcg_n}}=2 \PPcoupling(T >n)\eqsp.
$$
By \cite[Theorem~19.3.9]{douc:moulines:priouret:soulier:2018}, for any two probabilities $\mu, \nu$  on $(\Xset,\Xsigma)$, we have $\tvnorm{\restric{\PP_\mu}{\mcg_n}-\restric{\PP_\nu}{\mcg_n}}$ $=\tvnorm{\mu \MK^n-\nu \MK^n}$  and
there exists a successful maximal distributional coupling of $(\PP_\mu,\PP_\nu)$ denoted by $\left\{(\Omega,\mcf,\PPcoupling[\mu,\nu],\Xcoupling,T,\Xcoupling',T') \right\}$. By \cite[Lemma~19.3.8]{douc:moulines:priouret:soulier:2018}, the distributional coupling $\PPcoupling[\mu,\nu]$ satisfies, for any nonnegative function $V$,
\begin{equation}
\label{eq:bound-coupling}
\PEcoupling[\mu,\nu][V(\check{X}_n) \indiacc{T > n}]= (\mu \MK^n - \nu \MK^n)^+ V \eqsp,
\quad \PEcoupling[\mu,\nu][V(\check{X}_n') \indiacc{T' > n}]= (\nu \MK^n - \mu \MK^n)^+ V \eqsp,
\end{equation}
where for any signed measure $\mu$ on $(\Xset,\Xsigma)$, $\mu^+$ denotes its positive part in the corresponding Jordan decomposition.
By construction,  $\PE_\xi[|S_n|^{2q}]= \PEcoupling[\xi,\pi][ | \sum_{k=0}^{n-1} g(\Xcoupling_k)|^{2q}]$ and $\PE_\pi[|S_n|^{2q}]= \PEcoupling[\xi,\pi][ | \sum_{k=0}^{n-1} g(\Xcoupling'_k) |^{2q}]$.
Denote $S_{T,n}= \sum_{k=0}^{n-1} |g(\Xcoupling_k)| \indiacc{T > k}$ and $S_{T',n}= \sum_{k=0}^{n-1} |g(\Xcoupling'_k) |\indiacc{T' > k}$.
\begin{lemma}
\label{lem:bound_sn_coupling_dist}
Assume \Cref{assG:kernelP_q}, \Cref{assG:kernelP_q_smallset}, and let $\xi$ be a probability measure on $(\Xset,\Xsigma)$. Then for any family of real measurable function $g$ on $\Xset$ it holds
\begin{enumerate}[label=(\alph*)]
\item   \label{lem:bound_sn_coupling_dist_1} for any $q \in\nsets$,
\begin{equation*}
\PE_\xi[|S_n|^{2q}]
\leq 2^{2q-1} \PE_\pi[|S_n|^{2q}]
+ 2^{4q-2} \PEcoupling[\xi,\pi]\Bigl[ S_{T',n}^{2q} \Bigr]
+ 2^{4q-2} \PEcoupling[\xi,\pi]\Bigl[S_{T,n}^{2q}  \Bigr] \eqsp.
\end{equation*}
\item \label{lem:bound_sn_coupling_dist_2} for any $t \geq 0$,
\begin{equation*}
\PP_{\xi}(|S_n| \geq t) \leq \PP_{\pi}(|S_n| \geq t/4) + \PPcoupling[\xi,\pi]( S_{T',n} \geq t/4) +
\PPcoupling[\xi,\pi]( S_{T,n} \geq t/2)\eqsp.
\end{equation*}
\end{enumerate}
\end{lemma}
\begin{proof}
Since
\begin{equation}
\label{eq:decomposition_distr_coupling}
\begin{split}
\sum_{k=0}^{n-1} g(\Xcoupling_k)
&= \sum_{k=0}^{n-1} g(\Xcoupling_k) \indiacc{T \geq n} + \sum_{k=0}^{n-1} g(\Xcoupling_k) \indiacc{T \leq n-1} \\
&= \sum_{k=0}^{n-1} g(\Xcoupling_k) \indiacc{T \geq n} + \sum_{k=0}^{T-1} g(\Xcoupling_k) \indiacc{T \leq n-1} +
\sum_{k=0}^{n-T-1} g(\shift_T \Xcoupling_{k}) \indiacc{T \leq n-1}\eqsp,
\end{split}
\end{equation}
we have
\begin{align}
\label{eq:coupling-1}
&\Bigl| \sum_{k=0}^{n-1} g(\Xcoupling_k) \Bigr|
\leq  S_{T,n} + \left| \sum_{k=0}^{n-T-1} g(\shift_T \Xcoupling_{k}) \indiacc{T \leq n-1}\right| \\
\label{eq:coupling-2}
&\Bigl| \sum_{k=0}^{n-T'-1} g(\shift_{T'} \Xcoupling'_{k}) \indiacc{T' \leq n-1} \Bigr|
\leq S_{T',n} + \left| \sum_{k=0}^{n-1} g(\Xcoupling'_k) \right|
\eqsp.
\end{align}
Now \ref{lem:bound_sn_coupling_dist_1} follows from Minkowski's inequality and (DC-1), (DC-2). Similarly, the proof of \ref{lem:bound_sn_coupling_dist_2} uses the same decomposition \eqref{eq:decomposition_distr_coupling}-\eqref{eq:coupling-2} and the union bound.
\end{proof}

\begin{lemma}
  \label{lem:prob_ineq_non_statio_v_norm}
  Assume \Cref{assG:kernelP_q},\Cref{assG:kernelP_q_smallset}, and let $\gamma \geq 0$, $\xi$ be a probability measure on $(\Xset,\Xsigma)$.   Then, for any real measurable function $g \in \mrl_{W^{\gamma}}$ on $\Xset$, $t \geq 0$,  it holds with  $\varpi_{\gamma} = 1/(1+\gamma)$, that
  \begin{align*}
&    \PPcoupling[\xi,\pi](S_{T,n} \geq t) +     \PPcoupling[\xi,\pi](S_{T',n} \geq t) \\
    &\leq 2 \parenthese{\rme^{\log(\rho)t^{\varpi_{\gamma}}/(4\ConstM[n,W^{\gamma}]^{\varpi_{\gamma}}\varpi_{\gamma})}\ratev^{-1/2} +   \rme^{-(1+\gamma) t^{\varpi_{\gamma}}/(2\ConstM[n,W^{\gamma}]^{\varpi_{\gamma}}\gamma)} (1-\ratev)^{-1} }\cmconstv \{ \xi(V) + \pi(V) \} \eqsp.
  \end{align*}
\end{lemma}
\begin{proof}
    Without loss of generality, we can assume that  $ \| g \|_{W^{\gamma}} = 1$. We first assume that $\gamma >0$. Note that by Young's inequality for products, we have that for any $u_1,u_2 \in \rset_+$,
  \begin{equation}
    \label{eq:2}
    u_1^{\varpi_{\gamma}} u_2^{\varpi_{\gamma}} \leq \varpi_{\gamma} u_1 + (1-\varpi_{\gamma}) u_2^{\varpi_{\gamma}/(1-\varpi_{\gamma})} \eqsp.
  \end{equation}
  Then, we get since $\varpi_{\gamma}/(1-\varpi_{\gamma}) = 1/\gamma$ that
  \begin{align}
  S_{T,n}^{\varpi_{\gamma}} &\leq \varpi_{\gamma} (T \wedge n) + (1-\varpi_{\gamma}) \parentheseDeux{\frac{S_{T,n}}{T\wedge n}}^{\varpi_{\gamma}/(1-\varpi_{\gamma})} \\
        \label{eq:3}
    &   \leq \varpi_{\gamma}  (T \wedge n) + (1-\varpi_{\gamma}) \max_{k \in \iint{0}{n-1}}  \{\log V(\Xcoupling_k) \indiacc{T > k} \} \eqsp.
  \end{align}
  Similarly, it is easy to verify that \eqref{eq:3} holds for $\gamma =0$. Therefore, we obtain
  \begin{multline}
    \label{eq:bound_exp_distriub_coupling}
    \PPcoupling[\xi,\pi](S_{T,n} \geq t ) \leq
    \PPcoupling[\xi,\pi]\parenthese{ \varpi_{\gamma} T \geq t^{\varpi_{\gamma}}/2 } \\
    +         \PPcoupling[\xi,\pi]\parenthese{ (1-\varpi_{\gamma})  \max_{k \in \iint{0}{n-1}}  \{\log V(\Xcoupling_k) \indiacc{T > k} \} \geq t^{\varpi_{\gamma}}/2}\eqsp.
  \end{multline}
  We bound the two terms in the right-hand side separately. Setting $\lambda_\ratev = - \log(\ratev)/2$ and using that $\indiacc{T = k} = \indiacc{T > k-1} - \indiacc{T > k}, k \in \nset$ and $V(x) \geq 1, x \in \Xset$, we get
  \begin{align}
    \nonumber
    &\PPcoupling[\xi,\pi]\parenthese{ \varpi_{\gamma} T \geq  t^{\varpi_{\gamma}}/2} \leq
    \rme^{-\lambda_{\ratev}t^{\varpi_{\gamma}}/(2\varpi_{\gamma})}  \PEcoupling[\xi,\pi] \parentheseDeux{ \rme^{\lambda_{\ratev} T}} \\
    \nonumber
    &  \qquad \qquad    \leq  \rme^{-\lambda_{\ratev}t^{\varpi_{\gamma}}/(2\varpi_{\gamma})} \Bigl\{ 1+ (\ratev^{-1/2}-1)  \sum_{k=0}^{\plusinfty}  \ratev^{-k/2} \PEcoupling[\xi,\pi] \parentheseDeux{V(\Xcoupling_k)  \indiacc{T > k}} \Bigr\}\\
        \label{eq:bound_exp_distriub_coupling_1}
    & \qquad \qquad \txts  \overset{(a)}{\leq} \rme^{-\lambda_{\ratev}t^{\varpi_{\gamma}}/(2\varpi_{\gamma})}  \Bigl\{ 1+ (\ratev^{-1/2}-1) \sum_{k=0}^{\plusinfty}  \ratev^{-k/2}  \{(\xi \MK^k - \pi)^+(V) \} \Bigr\}\eqsp,
  \end{align}
where (a) is due to \eqref{eq:bound-coupling}. We complete the bound using that $(\xi \MK^k - \pi)^+(V) \leq \cmconstv \{ \xi(V) + \pi(V) \} \ratev^k$ due to \eqref{eq:V-geometric-coupling-general}. Now, applying \eqref{eq:bound-coupling} again, we obtain
  \begin{align}
\nonumber
  &\PPcoupling[\xi,\pi]\parenthese{ (1-\varpi_{\gamma})  \max_{k \in \iint{0}{n-1}}  \{\log V(\Xcoupling_k) \indiacc{T > k} \} \geq t^{\varpi_{\gamma}}/2}\\
  \nonumber
& \qquad \leq \exp\parenthese{-\frac{t^{\varpi_{\gamma}}}{2(1-\varpi_{\gamma})}} \sum_{k=0}^{\plusinfty} \PEcoupling[\xi,\pi] \parentheseDeux{V(\Xcoupling_k)  \indiacc{T > k}} \\
& \qquad
  \leq  \exp\parenthese{-\frac{t^{\varpi_{\gamma}}}{2(1-\varpi_{\gamma})}} \sum_{k=0}^{\plusinfty}  \{(\xi \MK^k - \pi)^+(V) \} \eqsp,
\label{eq:bound_exp_distriub_coupling_2}
\end{align}
and we again complete the bounds using \eqref{eq:bound-coupling}. The statement now follows by combining \eqref{eq:bound_exp_distriub_coupling_1} and \eqref{eq:bound_exp_distriub_coupling_2} in
  \eqref{eq:bound_exp_distriub_coupling}. We proceed similarly for $\PPcoupling[\xi,\pi](S_{n,T'}  \geq t)$.
\end{proof}

\begin{proof}[Proof of \Cref{theo:changeofmeasure}]
By Minkowski's inequality, we get
\begin{align*}
\Bigl( \PEcoupling[\xi,\pi][   S_{T,n}^{2q}  ] \Bigr)^{1/2q}
\leq \Vnorm[V^{1/(2q)}]{g} \sum_{k=0}^{n-1} \left( \PEcoupling[\xi,\pi][V(\Xcoupling_k) \indiacc{T > k}] \right)^{1/2q}.
\end{align*}
Using \eqref{eq:bound-coupling}, we get
$ \PEcoupling[\xi,\pi][V(\Xcoupling_k) \indiacc{T > k}]  \leq
\Vnorm[V]{\pi - \xi \MK^k}^{1/2q}$. Note also that the same upper bound holds for $\PEcoupling[\xi,\pi][S_{T',n}^{2q}]$ by (DC-2). Now, combining \Cref{lem:bound_sn_coupling_dist}-\ref{lem:bound_sn_coupling_dist_1} with \eqref{eq:V-geometric-coupling-general}, we get 
\begin{equation*}
\PE_\xi\big[ \big|S_n \big|^{2q} \big] \leq 2^{2q-1} \PE_\pi\big[ \big|S_n \big|^{2q} \big]  + 2^{4q-1} \Vnorm[V^{1/(2q)}]{g}^{2q} \cmconstv \{ \xi(V) + \pi(V) \} (1 - \ratev^{1/2q})^{-2q}  \eqsp.
\end{equation*}
To conclude, we note that for $x \in (0,1)$, it holds that
\[
\frac{1}{\log{(1/x)}} = \frac{1}{\log(1 + (1-x)/x)} \geq \frac{x}{1-x}\,,
\]
and apply the above inequality with $x = \ratev^{1/2q}$.
\end{proof}

\begin{proof}[Proof of \Cref{theo:changeofmeasure-1}]
Using that $( \sum_{k=0}^{p-1} a_k )^{2q} \leq p^{2q-1} \sum_{k=0}^{p-1} a_k^{2q}$ for any $a_k \geq 0$, and
\begin{align*}
\PEcoupling[\xi,\pi]\Bigl[ S_{T,n}^{2q} \Bigr] &= \PEcoupling[\xi,\pi]\Bigl[ \bigl(\sum_{k=0}^{(n-1) \wedge (T-1)} |g(\Xcoupling_k)| \bigr)^{2q} \Bigr]
\leq \|g\|_{W^{\gamma}}^{2q} \sum_{k=0}^{n-1}  \PEcoupling[\xi,\pi][T^{2q-1} W(\Xcoupling_k)^{2 \gamma q} \indiacc{T > k}]\\
& \txts  \overset{(a)}{\leq} (1/2) \|g\|_{W^{\gamma}}^{2q} \PEcoupling[\xi,\pi][T^{4q-2}]  + (1/2) \|g\|_{W^{\gamma}}^{2q} (4 q \gamma/ \rme)^{4 q \gamma} \sum_{k=0}^{n-1}  \{(\xi \MK^k - \pi)^+(V) \}\eqsp.
\end{align*}
In (a) we used \eqref{eq:bound-coupling} combined with the bound $\sup_{x \in \Xset}W^{4\gamma q}(x) / V(x) \leq (4\gamma q/\rme)^{4\gamma q}$, which holds since $V(x) \geq \rme$. Since $\PEcoupling[\xi,\pi][T^{4q-2}] \leq  \rme^{-1} \sum_{k=0}^{\infty} (k+1)^{4q-2} \PEcoupling[\xi,\pi][V(\Xcoupling_k) \indiacc{T \geq k}]$, one more application of \eqref{eq:bound-coupling} yields
\begin{align*}
&\PEcoupling[\xi,\pi]\Bigl[ S_{T,n}^{2q} \Bigr]
+ \PEcoupling[\xi,\pi]\Bigl[ S_{T',n}^{2q} \Bigr]
\\
&\leq  \|g\|_{W^{\gamma}}^{2q} \rme^{-1} \sum_{k=0}^\infty (k+1)^{4q-2} \Vnorm[V]{\xi \MK^k - \pi} +
\|g\|_{W^{\gamma}}^{2q}  (4 q \gamma/ \rme)^{4 q \gamma} \sum_{k=0}^{\infty} \Vnorm[V]{\xi \MK^k - \pi} \\
&\txts  \overset{(b)}{\leq} \cmconstv \{ \xi(V) + \pi(V) \} \|g\|_{W^{\gamma}}^{2q} \left( \rme^{-1} \ratev^{-1} \{ \log(1/\ratev) \}^{1- 4 q} (4q-2) ! +  (4 q \gamma/ \rme)^{4 q \gamma} (1- \ratev)^{-1} \right)
\eqsp.
\end{align*}
In (b) we used \eqref{eq:V-geometric-coupling-general} together with an upper bound
\begin{equation*}
\sum_{k=0}^\infty (k+1)^{4q-2} \ratev^{k} \leq \ratev^{-1}\int_{0}^{+\infty}x^{4q-2}\ratev^{x}\,\rmd x = \ratev^{-1}(\log{1/\ratev})^{1-4q}(4q-2)!\eqsp.
\end{equation*}
The rest of the proof follows from \Cref{lem:bound_sn_coupling_dist}-\ref{lem:bound_sn_coupling_dist_1} combined with 
\[
\frac{1}{1-\ratev} \leq \frac{\{ \log(1/\ratev) \}^{-1}}{\ratev}\,.
\]
\end{proof}

\begin{proof}[Proof of \Cref{theo:prob_ineq_V_norm}]
  The proof follows from
  \Cref{lem:bound_sn_coupling_dist}-\ref{lem:bound_sn_coupling_dist_2}
  and \Cref{lem:prob_ineq_non_statio_v_norm}.
\end{proof}

\subsection{Proof of \Cref{th:rosenthal_log_V}}
\label{sec:proof-ros_log_V}
Similarly to \Cref{sec:proof-ros_v_q}, we show that, given $\gamma \geq 0$, assumptions \Cref{assG:kernelP_q} and \Cref{assG:kernelP_q_smallset} imply \Cref{assum:central_moments_bound}$(q,W^{\gamma},\Vnorm[W^\gamma]{\cdot})$ for any $q \in \nset$, where $W(x) = \log{V(x)}$. \Cref{assG:kernelP_q} implies that for all $x \in \Xset$ we have $W(x) \geq 1$.

\begin{lemma}
\label{coro:centered_moments_W_class_new}
Assume \Cref{assG:kernelP_q}, \Cref{assG:kernelP_q_smallset}. Then for any $q \in \nset$ and $\gamma \geq 0$, \Cref{assum:central_moments_bound}$(q,W^{\gamma},\Vnorm[W^\gamma]{\cdot})$ is satisfied with constants $\ConstD[q,W^{\gamma}] = 2^{1+\gamma}\gamma^{\gamma}\cmconstv\pi(V)$, $\arate[q,W^{\gamma}] = \gamma$, and $\rho_{q,W^{\gamma}} = \ratev^{1/2}$, where $\ratev$ is defined in \eqref{eq:V-geometric-coupling-general}.
\end{lemma}
\begin{proof}
  Let $k \in \{1,\ldots,q\}$ and
 $(t_1,\dots,t_{k}) \in \{0,\ldots,n-1\}^{k}$, $t_1 \leq \dots\leq t_{k}$.
Let $\maxind \in \{2,\ldots,k\}$ be an index of the largest gap in $(t_1,\dots,t_{k})$, that is, $t_{\maxind} - t_{\maxind-1} = \max_{j \in \{2,\ldots,k\}} [t_j - t_{j-1}]$. If such index $\maxind$ is not unique, we choose the largest one. For $i \in \{1,\ldots,k\} \setminus \{\maxind\}$, we set $p_{i} = 1$, and $p_{\maxind} = k$. Note that for $i \in \{1,\dots,k\}$, $h_i \in \mrl_{W^{\gamma}}$ implies $h_i \in \mrl_{V^{p_i/(2k)}}$.
Now we apply \Cref{lem:centered_moments_Z_old} with the mentioned choice of $(p_1,\dots, p_k)$ and $s = 2k$, and obtain
\begin{multline*}
|\PEC_\pi[h_1(X_{t_1}), \ldots, h_k(X_{t_k})] |
\\\leq 2^{k-1}  (\cmconstv \pi(V))^{\sum_{\ell=1}^k\sum_{j=\ell}^k p_{j}/(2k)} \ratev^{\sum_{j=2}^{k}(t_j - t_{j-1})p_{j}/(2k)} \prod_{\ell = 1}^k   \|h_\ell\|_{V^{1/(2k)}} \eqsp.
\end{multline*}
Here we used that $\|h_{\maxind}\|_{V^{1/2}} \leq \|h_{\maxind}\|_{V^{1/(2k)}}$ since $k \geq 1$ and $V(x) > 1$. To complete the proof it remains to note that $\sum_{i=1}^k p_i \leq 2k$ and
\begin{align*}
\|h_\ell\|_{V^{1/(2k)}}
= \sup_{x \in \Xset}\biggl\{\frac{|h(x)|}{V^{1/(2k)}(x)}\biggr\} \leq \sup_{x \in \Xset}\biggl\{\frac{|h(x)|}{W^{\gamma}(x)}\biggr\} \sup_{x \in \Xset}\biggl\{\frac{W^{\gamma}(x)}{V^{1/(2k)}(x)}\biggr\} \leq (2\gamma k / \rme)^{\gamma}\|h_\ell\|_{W^{\gamma}}\eqsp.
\end{align*}
Combining the previous inequalities,
\begin{align*}
|\PEC_\pi[h_1(X_{t_1}), \ldots, h_k(X_{t_k})] | \leq \bigl(2^{1+\gamma}\gamma^{\gamma}\cmconstv\pi(V)\rme^{-\gamma}\bigr)^{k}k^{\gamma k}\ratev^{\maxgap(I)/2} \prod_{\ell = 1}^k   \|h_\ell\|_{W^{\gamma}}\eqsp,
\end{align*}
which completes the proof together with the elementary inequality $k^{k} \leq k!\rme^{k}, k \in \nset$.
\end{proof}

\begin{proof}[Proof of \Cref{th:rosenthal_log_V}] The proof now follows from \Cref{coro:centered_moments_W_class_new} and \Cref{th:generic_th_rosenthal}.
\end{proof}

\subsection{Proof of \Cref{th:rosenthal_log_V_cor_2}}
\label{sec:proof_bernstein_bound}
We first prove the bound \eqref{eq:bernstein_mc}. \Cref{lem:cumulant_bounds_final_generic,coro:centered_moments_W_class_new} imply that for any $k \geq 3$,
\begin{equation*}
\begin{split}
|\Gamma_{\pi, k}(S_n)|
&\leq \ratev^{-1/2} 2^{k-1}\{\log(1/\ratev)\}^{1-k} \ConstD[q,W^{\gamma}]^{k} (k!)^{3+\gamma} n \|h_\ell\|_{W^{\gamma}}^{k}  \\
&\leq \biggl(\frac{k!}{2}\biggr)^{3+\gamma}\PVar[\pi](S_n) \biggl( \frac{n \ratev^{-1/2} \{\log(1/\ratev)\}^{-1} \ConstD[q,W^{\gamma}]^{2} \|\bar{g}\|_{W^{\gamma}}^{2}}{\PVar[\pi](S_n)} \vee 1\biggr) \biggl( \frac{2 \ConstD[q,W^{\gamma}] \|\bar{g}\|_{W^{\gamma}}}{\log(1/\ratev)} \biggr)^{k-2} \\
&\leq \biggl(\frac{k!}{2}\biggr)^{3+\gamma}\PVar[\pi](S_n) \ConstJ^{k-2} \eqsp,
\end{split}
\end{equation*}
where $\ConstD[q,W^{\gamma}] = 2^{1+\gamma}\gamma^{\gamma}\cmconstv\pi(V)$ and $\ConstJ$ is given in \eqref{eq:const_B_n_definition_main}. We conclude the statement using \cite[Lemma~2.1]{bentkus:1980} (see also \cite[Equation~(24)]{doukhan2007probability}). Next we show \eqref{eq:high_prob_bound_W_ergodic}. From the bound above we get for $t \geq 0$ that
\begin{equation*}
\PP_{\pi}(|S_n| \geq t) \leq 2\exp\biggl\{-\frac{t^2/2}{\PVar[\pi](S_n) + \ConstJ^{1/(3+\gamma)} t^{2-1/(3+\gamma)}}\biggr\}\eqsp.
\end{equation*}
Since for any $t$ it holds either $\PVar[\pi](S_n) + \ConstJ^{1/(3+\gamma)} t^{2-1/(3+\gamma)} \leq 2\PVar[\pi](S_n)$ or $\PVar[\pi](S_n) + \ConstJ^{1/(3+\gamma)} t^{2-1/(3+\gamma)} \leq 2\ConstJ^{1/(3+\gamma)} t^{2-1/(3+\gamma)}$,
the previous bound implies
\begin{align*}
\PP_{\pi}(|S_n| \geq t) \leq 2\exp\biggl\{-\frac{t^2}{4\PVar[\pi](S_n)}\biggr\} + 2\exp\biggl\{-\frac{t^{1/(3+\gamma)}}{4\ConstJ^{1/(3+\gamma)}}\biggr\} = T_1 + T_2\eqsp.
\end{align*}
To complete the proof we choose $t$ such that $T_1 \leq \delta/2$ and $T_2 \leq \delta/2$ for $\delta \in (0,1)$.


\subsection{Weak Harris Theorem}
\label{sec:proof-ros_W_q}
\begin{proof}[Proof of Proposition~\ref{prop:wasser:geo}]
Set $\gamma = p/2q$. Note that $\MKK^m$ satisfies the geometric drift condition
\begin{equation}
\label{eq:geometric-drift-condition_m}
\MKK^m \bar{V} \leq \bar{\lambda}_m \bar{V} + b_m \indi{\CKset}\eqsp,
\end{equation}
where $\CKset$ is defined in \Cref{assG:kernelP_q_contractingset_m}. For $\delta  \geq 0$, set $\bar{V}_{\delta}=\bar{V}+\delta$ and
\begin{equation}
\label{eq:rho_tilde_def}
    \tilde{\rho}_{\gamma, \delta} = \sup_{(x,x') \in \Xset^2} \lrb{(1-\minorwas \indi{\CKset}(x,x'))^{1/2} \lr{ \frac{K^m\bar V^{\gamma}_\delta(x,x')}
    {\bar{V}^{\gamma}_\delta(x,x')}}^{1/2}} \eqsp ,
\end{equation}
  which is finite since $\MKK^m$ satisfies~\eqref{eq:geometric-drift-condition_m}. Furthermore, H\"older's inequality and \Cref{assG:kernelP_q_contractingset_m} yield
  \begin{align*}
    \MKK^m(\metricc^{1/2} \bar{V}_\delta^{\gamma/2})
    & \leq (\MKK^m\metricc)^{1/2} (\MKK^m\bar V^{\gamma}_\delta)^{1/2} \\
    & \leq \lrb{(1-\epsilon \indi{\CKset})^{1/2} \lr{{\MKK^m\bar V^{\gamma}_\delta}/{\bar{V}^{\gamma}_\delta}}^{1/2} }
      \metricc^{1/2} \bar{V}^{\gamma/2}_\delta \leq \tilde{\rho}_{\gamma, \delta} \metricc^{1/2} \bar{V}_\delta^{\gamma/2} \eqsp.
  \end{align*}
  Using $\bar{V} \leq \bar{V}_\delta$ and a straightforward induction, we obtain for any $k \geq 1$
  \begin{equation}
    \label{eq:wasser:geo:K_n}
    \MKK^{mk} (\metricc^{1/2} \bar V^{\gamma/2})\leq \MKK^{mk} (\metricc^{1/2} \bar V_\delta^{\gamma/2}) \\
    \leq \tilde{\rho}_{\gamma,\delta}^k \metricc^{1/2} \bar {V}_\delta^{\gamma/2}\eqsp.
  \end{equation}
  Let $n = k m + \ell$, $\ell \in \{0,\dots, m-1\}$. Note that the drift condition \Cref{assG:kernelP_q} together with Jensen's inequality imply
  \begin{align*}
  \MKK^\ell\bar V^{\gamma}_\delta \leq (\MKK^\ell\bar V_\delta)^{\gamma} \leq \biggl(\lambda^{\ell}\bar{V} + b(1-\lambda^\ell)/(1-\lambda)+\delta\biggr)^{\gamma} \leq (1 + b/(1-\lambda) + \delta)^{\gamma} \bar{V}^{\gamma}\eqsp.
  \end{align*}
  Combining this with H\"older's inequality yields
  \begin{equation*}
    \MKK^\ell(\metricc^{1/2} \bar{V}_\delta^{\gamma/2})
    \leq (\MKK^\ell\metricc)^{1/2} (\MKK^\ell\bar V^{\gamma}_\delta)^{1/2} \\
    \leq \boundmetric^{m/2} (1 + b/(1-\lambda) + \delta)^{\gamma/2}  \metricc^{1/2} \bar V^{\gamma/2} \eqsp.
  \end{equation*}
This inequality and \eqref{eq:wasser:geo:K_n}   imply
\begin{equation}
  \label{eq:prop:wasser:geo_1}
      \MKK^{km + \ell}(\metricc^{1/2} \bar{V}_\delta^{\gamma/2}) \leq \tilde{\rho}_{\gamma, \delta}^k \MKK^\ell(\metricc^{1/2}\bar {V}_\delta^{\gamma/2}) \le \tilde{\rho}_{\gamma,\delta}^k \boundmetric^{m/2} (1 + b/(1-\lambda) + \delta)^{\gamma/2}  \metricc^{1/2} \bar V_{\delta}^{\gamma/2} \eqsp.
    \end{equation}
We now provide a lower bound on $ \tilde{\rho}_{\gamma,\delta}$.
Applying~\eqref{eq:geometric-drift-condition_m}, we obtain
  \begin{equation}
    \label{eq:wasser:geo:third}
    \frac{\MKK^m\bar V^{\gamma}_\delta}{\bar V^{\gamma}_\delta} \leq \{ \varphi(\bar{V}) \}^{\gamma} \indi{\CKset} + \{\psi(\bar{V}) \}^{\gamma} \indi{\CKset^c} \eqsp,
  \end{equation}
  with $\varphi(v) = (\bar{\lambda}_m  v + b_m + \delta)/(v+\delta), \psi(v)= (\bar{\lambda}_m  v   + \delta)/(v+\delta)$. Since $\bar{V} \geq \indi{\CKset} +\bar{d}\indi{\CKset^c}$ and functions $\varphi$ and $\psi$ are decreasing on $[1;+\infty)$, we get
  \begin{align*}
    \frac{\MKK^m\bar V^{\gamma}_\delta}{\bar V^{\gamma}_\delta} \leq
      \left\{\frac{\bar{\lambda}_m+b_m+\delta}{1+\delta} \right\}^{\gamma} \indi{\CKset}
    + \left\{\frac{\bar{\lambda}_m + \bar{d}+\delta}{\bar{d}+\delta} \right\}^{\gamma}  \indi{\CKset^c}\eqsp.
  \end{align*}
  The previous inequality yields
  \begin{equation*}
    (1-\minorwas \indi{\CKset})^{1/2} \lr{{\MKK^m\bar{V}^{\gamma}_\delta}/{\bar{V}^{\gamma}_\delta}}^{1/2}
    \leq \lrb{(1-\minorwas)^{1/2} \lr{\frac{\bar \lambda_m+
          b_m+\delta}{1+\delta}}^{\gamma/2}}\indi{\CKset} + \lr{\frac{\bar \lambda_m \bar{d}+\delta}{\bar{d}
        +\delta}}^{\gamma/2} \indi{\CKset^c}\eqsp.
  \end{equation*}
Due to \eqref{eq:rho_tilde_def}, the previous inequality implies $\tilde{\rho}_{\gamma,\delta} \leq \rho_{\delta}^\gamma$, where
\begin{equation}
      \label{eq:def:rho_0}
      \rho_{\delta}=\lrb{(1-\minorwas)^{1/2} \lr{\frac{\bar \lambda_m+b_m+\delta}{1+\delta}}^{1/2}}
      \vee \lr{\frac{\bar \lambda_m \bar{d}+\delta}{\bar{d}+\delta}}^{1/2}\eqsp.
    \end{equation}
Now we choose $\delta = \deltawas$ defined in~\eqref{eq:delta_star_def}, and complete the proof setting $\ratewas = \rho_{\deltawas}$ in \eqref{eq:def:rho_0}.
\end{proof}

\begin{proof}[Proof of \Cref{cor:wasserstein-convergence}]
Using \cite[Corolllary~20.4.1]{douc:moulines:priouret:soulier:2018} and \eqref{eq: constraction} (with $p=2q$),  we get for any initial distribution $\xi, \xi'$ and $\gamma \in \couplingmeasure(\xi,\xi')$,
\begin{align}
  \nonumber
    \wasser[\cost]{\xi \MK^n}{\xi'\MK^n} &\leq \wasser[\cost^{1/2} \bar V^{1/2}]{\xi \MK^n}{\xi' \MK^n}
    \leq \int_{\Xset \times \Xset} \wasser[\cost^{1/2} \bar V^{1/2}]{\MK^n(x,\cdot)}{\MK^n(x',\cdot)}
      \gamma(\rmd x\rmd x')\eqsp, \\
    \nonumber
    &\leq \ratewas^{n}  \boundmetric^{m/2} \vartconstwas  \int_{\Xset \times \Xset} \gamma(\rmd x\rmd x') \bar{V}^{1/2}(x,x') \\
  \label{eq:cor:wasserstein-convergence_1}
    &\leq (1/\sqrt{2}) \ratewas^{n}   \boundmetric^{m/2} \vartconstwas \{ \xi(V^{1/2}) + \xi'(V^{1/2}) \} \eqsp.
\end{align}
We now show existence and uniqueness of $\pi$. Equation \eqref{eq:cor:wasserstein-convergence_1} implies that for some fixed $x_0 \in\Xset$
   \[
   \wasser[(\distance \wedge 1)^{\pcost}]{\MK^n(x_0,\cdot)}{\MK^{n+1}(x_0,\cdot)}  \leq (1/\sqrt{2}) \ratewas^{n}   \boundmetric^{m/2} \vartconstwas \{ V^{1/2}(x_0) + \MK V^{1/2}(x_0) \} \eqsp.
   \]
   Hence the sequence $\{\MK^n(x_0,\cdot)\}_{n=1}^\infty$ is a Cauchy sequence in the complete metric space of probability measure equipped with the Wasserstein distance $\wassersym[(\distance \wedge 1)^{\pcost}]$ (see \cite[Theorem~20.1.8]{douc:moulines:priouret:soulier:2018}). Therefore, there exists a probability measure $\pi$ on $(\Xset,\Xsigma)$ such that $\lim_{n \to \infty} \wasser[(\distance \wedge 1)^{\pcost}]{\MK^n(x_0,\cdot)}{\pi}=0$. It is easily seen that $\pi= \pi \MK$ (see e.g. \cite[Theorem~20.2.1]{douc:moulines:priouret:soulier:2018}), and the existence of invariant distribution is show. \Cref{assG:kernelP_q} implies for a stationary distribution $\pi$ that $\pi(V) \leq b / (1-\lambda)$ (see \cite[Lemma~14.1.10]{douc:moulines:priouret:soulier:2018}).
   \par
   Second, if $\pi$ and $\pi'$ are two invariant probability measures, \eqref{eq:cor:wasserstein-convergence_1} implies
$\wasser[\cost]{\pi}{\pi'}=0$. Hence, by \Cref{ass:cost_fun}, $\wassersym[(\distance \wedge 1)^{\pcost}](\pi,\pi')=0$ and finally $\pi=\pi'$. Now we complete the proof of \eqref{eq:wasser:geo:bound:pi} setting $\xi' = \pi$ in \eqref{eq:cor:wasserstein-convergence_1}.
\end{proof}

\subsection{Proof of \Cref{th:rosenthal_V_poly_wasserstein}}
\label{sec:proof:rosenthal_V_poly_wasserstein}
We preface the proof by technical lemmas.
\begin{lemma}
\label{lem: product of two funct}
Let $h_1 \in \Lclass_{\beta_1, \lyapW}$ and $h_2 \in \Lclass_{\beta_2, \lyapW}$ with $\beta_1, \beta_2 \in \rset_{+}$. Then 
\begin{equation*}
\Nnorm[\beta_1 + \beta_2, \lyapW]{h_1 h_2} \leq 2^{1+\beta_1 + \beta_2} \Nnorm[\beta_1, \lyapW]{h_1}  \Nnorm[\beta_2, \lyapW]{h_2}\eqsp.
\end{equation*}
\end{lemma}
\begin{proof}
Fix arbitrary $x, y \in \Xset$. Then
\begin{equation*}
\begin{split}
|h_1(x) h_2(x) - h_1(y) h_2(y)| &\leq |h_1(x) - h_1(y)||h_2(x)| + |h_2(x) - h_2(y)||h_1(y)| \\
& \leq \Nnorm[\beta_1, \lyapW]{h_1} \metricc^{1/2}(x,y) \bar \lyapW^{\beta_1}(x,y) \Nnorm[\beta_2, \lyapW]{h_2} \lyapW^{\beta_2}(x) \\
& + \Nnorm[\beta_2, \lyapW]{h_2} \metricc^{1/2}(x,y) \bar \lyapW^{\beta_2}(x,y) \Nnorm[\beta_1, \lyapW]{h_1} \lyapW^{\beta_1}(y)\eqsp.
\end{split}
\end{equation*}
Since for any $y \in \Xset$, $\lyapW(x) \leq 2\bar \lyapW(x,y)$ we get
\begin{equation*}
 |h_1(x) h_2(x) - h_1(y) h_2(y)| \le 2^{ 1 + \beta_1 + \beta_2} \Nnorm[\beta_1, \lyapW]{h_1}  \Nnorm[\beta_2, \lyapW]{h_2} \metricc^{1/2}(x,y) \bar \lyapW^{\beta_1+\beta_2}(x,y) \eqsp.
\end{equation*}
Similarly,
\begin{equation*}
|h_1(x) h_2(x)| \leq \Nnorm[\beta_1, \lyapW]{h_1}  \Nnorm[\beta_2, \lyapW]{h_2} \lyapW^{\beta_1 + \beta_2}(x)\eqsp.
\end{equation*}
The last two inequalities imply the statement.
\end{proof}

\begin{lemma}
\label{lem: P^n h class}
Assume \Cref{assG:kernelP_q} and \Cref{assG:kernelP_q_contractingset_m}. For any $p,q \in \nset$, $p \leq 2q$, $h \in \Lclass_{p/(4q),V}$ and $n \in \nset$, it holds that $\MK^n h - \pi(h) \in \Lclass_{p/(4q),V}$ with $\Nnorm[p/(4q), V]{\MK^n h} \leq \boundmetric^{m/2} \constlemqnpi^{p/(2q)}  \Nnorm[p/(4q),V]{h} \ratewas^{pn/(2q)}$, where the constant $\constlemqnpi$ is given by
\begin{equation}
\label{eq:kappa_alpha_def}
\constlemqnpi = \pi(V)^{1/2} \vartconstwas/\sqrt{2}\eqsp,
\end{equation}
and $\vartconstwas$ defined in \eqref{eq:def:rho}.
\end{lemma}
\begin{proof}
Applying \Cref{prop:wasser:geo} (equation \eqref{eq: constraction}), we get
\begin{align}
    |\MK^n h(x) - \MK^n h(x')| &\le \int \MKK^n(x,x', \rmd y \rmd y') |h(y) - h(y')| \nonumber\\
    & \leq \Nnorm[p/(4q), V]{h} \int \MKK^n(x,x'; \rmd y \rmd y') \metricc^{1/2}(y,y') \bar V^{p / (4q)}(y,y') \nonumber \\
    & \leq \Nnorm[p/(4q), V]{h} \boundmetric^{m/2} \vartconstwas^{p/(2q)} \ratewas^{pn/(2q)} \metricc^{1/2}(x,x') \bar{V}^{p/(4q)}(x,x')\eqsp. \label{eq:q_n_diff_bound}
\end{align}
Moreover, integrating \eqref{eq:q_n_diff_bound} w.r.t. $\pi$ and using that $\metricc(x,x') \leq 1$, we get
\begin{align}
| \MK^n h(x)  - \pi(h)|
&\leq \Nnorm[p/(4q), V]{h} \boundmetric^{m/2} \vartconstwas^{p/(2q)} \ratewas^{pn/(2q)} \int \bar{V}^{p/(4q)}(x,x') \pi(\rmd x') \nonumber \\
&\leq \boundmetric^{m/2} \vartconstwas^{p/(2q)} (\pi(V)/2)^{p/(4q)} \Nnorm[p/(4q), V]{h} \ratewas^{pn/(2q)} V^{p /(4q)}(x) \eqsp. \label{eq:q_n_pi_diff_bound}
\end{align}
In the last inequality we used that $V(x) + V(x') \leq V(x)V(x')$ since $V(x) \geq \rme$ for any $x \in \Xset$, and
\begin{align*}
\int \bar{V}^{p/(4q)}(x,x') \pi(\rmd x') = (V(x)/2)^{p/(4q)}\int V(x')^{p/(4q)} \pi(\rmd x') \leq (\pi(V)/2)^{p/(4q)}V^{p/(4q)}(x)\eqsp.
\end{align*}
Combining \eqref{eq:q_n_pi_diff_bound} and \eqref{eq:q_n_diff_bound} completes the proof.
\end{proof}

Based on \Cref{lem: product of two funct} and \Cref{lem: P^n h class}, we check Assumption \cref{assum:central_moments_bound}($q,\lyapW$) with the functional class $\Lclass_{1/(4q),V}$. The lemma below is an adaptation of \Cref{lem:centered_moments_Z_old}.
\begin{lemma}
\label{lem:centered_moments_Wasserstein_new}
Assume \Cref{assG:kernelP_q} and \Cref{assG:kernelP_q_contractingset_m}. Let $q \in \nset$. Then for any $k \in \{1,\ldots,2q\}$,  $(t_1,\dots,t_{k}) \in \{0,\ldots,n-1\}^{k}$, $t_1 \leq \dots\leq t_{k}$,   $(p_1,\dots,p_k) \in \nset^k$ satisfying  $p_i \geq 1$ for $i \in \{1,\dots,k\}$ and $\sum_{i=1}^k p_i \leq s$, and functions $\{h_{\ell}\}_{\ell=1}^{k}$ satisfying  $h_{i} \in  \Lclass_{p_i/(2s), V}$, $i \in \{1,\dots,k\}$, it holds
\begin{equation}
\label{eq:centred_moments_wasserstein_first_bound}
|\PEC_\pi[h_1(X_{t_1}), \ldots, h_k(X_{t_k})] |
\leq (2 \boundmetric^{m/2})^{k-1}  (2\constlemqnpi)^{\sum_{\ell=1}^k\sum_{j=\ell}^k p_{j}/s}
\ratewas^{\sum_{j=2}^{k}(t_j - t_{j-1})p_{j}/s} \prod_{\ell = 1}^k \Nnorm[p_{\ell}/(2s), V]{h_{\ell}} \eqsp,
\end{equation}
where $\constlemqnpi$ is defined in \eqref{eq:kappa_alpha_def}.
\end{lemma}
\begin{proof}
The proof is based on induction on $k \in \{1,\dots,2q\}$.  For $k = 1$, we get
\begin{equation*}
|\PEC_\pi[h_1(X_{t_1})]|
\leq \pi(V^{p_1/(2s)}) \|h_1\|_{V^{p_{1}/(2s)}}
\leq \{\pi(V)\}^{p_{1}/(2s)}
\Nnorm[p_1/(2s),V]{h_1}\eqsp,
\end{equation*}
where for the second inequality we used Jensen's inequality. Assume that~\eqref{eq:centred_moments_wasserstein_first_bound} holds for some $k \in \{1,\dots,2q-1\}$. Let $t_1 \leq \dots\leq t_{k+1}$,   $(p_1,\dots,p_{k+1})\in \nset^{k+1}$ satisfying  $p_i \geq 1$ for $i \in \{1,\dots,k+1\}$ and $\sum_{i=1}^{k+1} p_i \leq s$, and functions $h_{\ell} \in  \Lclass_{p_{\ell}/(2s),V}, \ell \in \{1,\dots,k+1\}$. Applying \Cref{lem:centred_moments_markov_property},
\begin{equation*}
\PEC_\pi[h_1(X_{t_1}), \ldots, h_k(X_{t_k}), h_{k+1}(X_{t_{k+1}})] = \PEC_\pi[h_1(X_{t_1}), \ldots , h_k(X_{t_k}) \tilde h_{k+1}(X_{t_k})]\eqsp,
\end{equation*}
where $\tilde h_{k+1}(x) = \MK^{t_{k+1} - t_k} h_{k+1}(x) - \pi(h_{k+1})$.
Since $h_{k+1} \in \Lclass_{p_{k+1}/(2s),V}$, we obtain using \Cref{lem: P^n h class} that $\tilde{h}_{k+1} \in \Lclass_{p_{k+1}/(2s),V}$ with
\begin{equation*}
\Nnorm[p_{k+1}/(2s),V]{\tilde h_{k+1}} \leq \boundmetric^{m/2} \constlemqnpi^{p_{k+1}/s}  \Nnorm[p_{k+1}/(2s), V]{h_{k+1}} \ratewas^{p_{k+1}(t_{k+1}-t_k)/s}\eqsp.
\end{equation*}
Hence, applying \Cref{lem: product of two funct} with $\lyapW = V$, $\beta_1 = p_{k}/(2s)$, and $\beta_2 = p_{k+1}/(2s)$
\begin{multline*}
\Nnorm[(p_k + p_{k+1})/(2s),V]{h_k \tilde h_{k+1}}
\leq 2^{1 + (p_k+p_{k+1})/(2s)} \boundmetric^{m/2} \constlemqnpi^{p_{k+1}/s} \ratewas^{p_{k+1}(t_{k+1}-t_k)/s} \\ \times \Nnorm[p_k/(2s), V]{h_k}  \Nnorm[p_{k+1}/(2s),V]{h_{k+1}} \eqsp.
\end{multline*}
Then, applying the induction hypothesis to $\bar{h}_i = h_i \in \Lclass_{\bar{p}_{i}/(2s),V}$, $\bar{p}_i = p_i$, $i \in\{1,\ldots,k-1\}$, $\bar{h}_k = h_k \tilde h_{k+1} \in \Lclass_{\bar{p}_{k}/(2s),V}$, $\bar{p}_k = p_k + p_{k+1}$ completes the proof.
\end{proof}

\begin{corollary}
\label{coro:centered_moments_wasserstein_new}
Assume \Cref{assG:kernelP_q} and \Cref{assG:kernelP_q_contractingset_m}. Then for any $q \in \nset$, \Cref{assum:central_moments_bound}($q,V,\Nnorm[1/(4q),V]{\cdot}$) is satisfied with $\ConstD[q,V] = 4 \boundmetric^{m/2}\constlemqnpi$, $\arate[q,V] = 0$ and $\rho_{q,V} = \ratewas^{1/2}$, where $\ratewas$ and $\constlemqnpi$ are defined in \eqref{eq:def:rho} and \eqref{eq:kappa_alpha_def}, respectively.
\end{corollary}
\begin{proof}
  Let $k \in \{1,\ldots,q\}$ and
 $(t_1,\dots,t_{k}) \in \{0,\ldots,n-1\}^{k}$, $t_1 \leq \dots\leq t_{k}$.
  Define $\maxind \in \{2,\ldots,k\}$ such that $t_{\maxind} - t_{\maxind-1} = \max_{j \in \{2,\ldots,k\}} [t_j - t_{j-1}]$. For $i \in \{1,\ldots,k\} \setminus \{\maxind\}$, we set $p_{i} = 1$, and put $p_{\maxind} = q$.
  Now we apply \Cref{lem:centered_moments_Wasserstein_new} with the mentioned $(p_1,\dots, p_k)$ and $s = 2q$.  Note that $ h_i \in \Lclass_{p_i/(4q), V}$ for any $i \in \{1,\ldots,k\}$ and $\sum_{i=1}^k p_i \leq 2q$. Moreover, $\Nnorm[1/4, V]{h_{\maxind}} \leq \Nnorm[1/(4q),V]{h_{\maxind}}$ since $q \geq 1$ and $V(x) > 1$. Therefore, the application of  \Cref{lem:centered_moments_Wasserstein_new} concludes the proof.
\end{proof}

\begin{proof}[Proof of \Cref{th:rosenthal_V_poly_wasserstein}] The proof now follows from \Cref{lem:centered_moments_Wasserstein_new} and \Cref{coro:centered_moments_wasserstein_new}.
\end{proof}

\subsection{Proof of \Cref{theo:changeofmeasure_wasser}}
\label{sec:proof-crefth_wass_change_mease}
Set $S_n' = \sum_{k=0}^{n-1} g_k(X_k')$. Using \Cref{assG:kernelP_q_contractingset_m}, we get
\begin{equation}
  \label{eq:sec:proof-crefth_wass_change_mease_1}
    \PE_\xi\big[ \big|S_n \big|^{2q} \big] = \PE_{\xi,\pi}^{\MKK}\big[ \big|S_n \big|^{2q} \big]   \leq 2^{2q-1} \PE_\pi\big[ \big|S_n \big|^{2q} \big]  +
     2^{2q-1} \PE_{\xi,\pi}^{\MKK}\big[ \big|S_n-S_n' \big|^{2q} \big]\eqsp,
  \end{equation}
where $\MKK$ is a kernel coupling. Using the Minkowski inequality and \Cref{prop:wasser:geo} completes the proof.

\subsection{Proof of \Cref{th:rosenthal_log_V_wasserstein}}
\label{sec:proof_ros_W_log}
\begin{lemma}
\label{lem:centered_moments_log_wasserstein}
Assume \Cref{assG:kernelP_q} and \Cref{assG:kernelP_q_contractingset_m}. Then for any $q \in \nset$ and $\gamma \geq 0$, \Cref{assum:central_moments_bound}($q,W^{\gamma},\Nnorm[1,W^\gamma]{\cdot}$) is satisfied with $\ConstDW[q,W^\gamma] = 2^{2+2\gamma}\gamma^{\gamma}\boundmetric^{m/2}\constlemqnpi$, $\arate[q,W^\gamma] = \gamma$ and $\rho_{q,W^{\gamma}} = \ratewas^{1/2}$, where $\ratewas$ and $\constlemqnpi$ are defined in \eqref{eq:def:rho} and \eqref{eq:kappa_alpha_def}, respectively.
\end{lemma}

\begin{proof}
Let $k \in \{1,\ldots,q\}$ and $I = (t_1,\dots,t_{k}) \in \{0,\ldots,n-1\}^{k}$, $t_1 \leq \dots\leq t_{k}$. Define $\maxind \in \{2,\ldots,k\}$ such that $t_{\maxind} - t_{\maxind-1} = \max_{j \in \{2,\ldots,k\}} [t_j - t_{j-1}]$. For $i \in \{1,\ldots,k\} \setminus \{\maxind\}$, we set $p_{i} = 1$, and put $p_{\maxind} = k$. Now we apply \Cref{lem:centered_moments_Wasserstein_new} with the mentioned $(p_1,\dots, p_k)$ and $s = 2k$. Proceeding as in \Cref{coro:centered_moments_wasserstein_new} with $\sum_{i=1}^k p_i \leq 2k$,
\begin{align*}
|\PEC_\pi[h_1(X_{t_1}), \ldots, h_k(X_{t_k})] |
\leq (4 \boundmetric^{m/2} \constlemqnpi)^{k} \ratewas^{\maxgap(I)/2} \prod_{\ell = 1}^k   \Nnorm[1/(4k), V]{h_{\ell}} \eqsp.
\end{align*}
To complete the proof it remains to note that for functions $h_{\ell} \in \Lclass_{1,W^{\gamma}}$ it holds
\begin{equation*}
\Nnorm[1/(4k), V]{h_{\ell}} \leq (4\gamma k/\rme)^{\gamma}\Nnorm[1, W^{\gamma}]{h_{\ell}}\eqsp.
\end{equation*}
\end{proof}

\begin{proof}[Proof of \Cref{th:rosenthal_log_V_wasserstein}] The proof now follows from \Cref{lem:centered_moments_Wasserstein_new} and \Cref{lem:centered_moments_log_wasserstein}.
\end{proof}
\subsection{Proof of \Cref{theo:changeofmeasure-1_wasser}}
\label{sec:proof-crefth-1_wass:theo:changeofmeasure-1_wasser}
Without loss of generality, we can assume that $\Nnorm[1, W^{\gamma}]{\bar{g}} = 1$. We set for any $x,x'\in\Xset$, $\bar{W}_{\gamma}(x,x') = 2^{-1}(W^{\gamma}(x)+ W^{\gamma}(x'))$.  Proceeding as in the proof of \Cref{theo:changeofmeasure_wasser}, \eqref{eq:sec:proof-crefth_wass_change_mease_1} holds and  we only need to bound $ \PE_{\xi,\pi}^{\MKK}\big[ \big|S_n-S_n' \big|^{2q} \big]$ with $S_n' = \sum_{k=0}^{n-1} g(X_k')$. Denote for any $k \in\{0,\ldots,n-1\}$, $c_k = \cost^{1/2}(X_k,X_k')$, $\Sigmabf_k = \sum_{l=0}^{k-1} c_l$ and $\Delta g_k = g(X_k)-g(X_k')$. First, we have
by Jensen inequality,
  \begin{align}
    \nonumber
    \big|S_n-S_n' \big|^{2q} &= \abs{\sum_{k=0}^{n-1} \{\Delta g_k\}}^{2q} \leq \Sigmabf_{n}^{2q-1} \defEns{\sum_{k=0}^{n-1} c_k \{\Delta g_k / c_k\}^{2q} }  \\
    \nonumber                             &\leq 2^{-1} \Sigmabf_{n}^{4q-2} \sum_{k=0}^{n-1} c_k  +  2^{-1} \sum_{k=0}^{n-1} c_k \{\Delta g_k / c_k\}^{4q}\\
                              \nonumber   & \leq 2^{-1} \Sigmabf_{n}^{4q-1}  +  2^{-1}\sum_{k=0}^{n-1} c_k\{ W^{4q \gamma}(X_k) + W^{4q\gamma}(X_k') \}\\
\label{eq:1:theo:changeofmeasure-1_wasser}                         & \leq 2^{-1} \Sigmabf_{n}^{4q-1}  +  \sqrt{2} \parenthese{8q \gamma/\rme}^{4q \gamma} \sum_{k=0}^{n-1} c_k\bar{V}^{1/2}(X_k,X_k') \eqsp,
  \end{align}
  where we used that  $(a+b)^{4q} \leq 2^{4q-1}\{a^{4q}+ b^{4q}\}$.
In addition by an easy induction on $\ell \in \{0,\ldots,n\}$, using that $\sup_{k \in \{0,\ldots,n-1\}} c_k \leq 1$, we have $ \Sigmabf_{\ell}^{4q-1} \leq \sum_{k=0}^{\ell-1} c_k (k+1)^{4q-1} \leq \sum_{k=0}^{\ell-1} c_k (k+1)^{4q-1} \bar{V}^{1/2}(X_k,X_k')$. Plugging this bound for $\ell=n$ in \eqref{eq:1:theo:changeofmeasure-1_wasser} and using \Cref{prop:wasser:geo} with $p=2q$ completes the proof.

\subsection{Proof of \Cref{th:rosenthal_log_V_cor_2_wasserstein}}
\label{sec:proof_bernstein_bound_wasserstein}
We proceed as in the proof of \Cref{th:rosenthal_log_V_cor_2}. Indeed, for any $k \geq 3$, \Cref{lem:centered_moments_log_wasserstein} implies
\begin{equation*}
\begin{split}
|\Gamma_{\pi, k}(S_n)|
&\leq \ratewas^{-1} 2^{k-1}\log^{1-k}\{1/\ratewas\} \ConstDW[q,W^{\gamma}]^{k} (k!)^{3+\gamma} n \Nnorm[1, W^{\gamma}]{\bar{g}}^{k} \\
&\leq \biggl(\frac{k!}{2}\biggr)^{3+\gamma}\PVar[\pi](S_n) \biggl( \frac{n \ratewas^{-1/2} \{\log(1/\ratewas)\}^{-1} \ConstDW[q,W^\gamma]^{2} \Nnorm[1, W^{\gamma}]{\bar{g}}^{2}}{\PVar[\pi](S_n)} \vee 1\biggr) \biggl( \frac{2 \ConstDW[q,W^\gamma] \Nnorm[1, W^{\gamma}]{\bar{g}}}{\log(1/\ratewas)} \biggr)^{k-2} \\
&\leq \biggl(\frac{k!}{2}\biggr)^{3+\gamma}\PVar[\pi](S_n)\,\ConstJW^{k-2}\eqsp,
\end{split}
\end{equation*}
with $\ConstDW[q,W^\gamma] = 2^{3/2+2\gamma}\gamma^{\gamma}\boundmetric^{m/2} \vartconstwas \{\pi(V)\}^{1/2}$ and $\ConstJW$ given in \eqref{eq:const_J_n_definition_main_was}. We conclude using \cite[Lemma~2.1]{bentkus:1980} (see also \cite[Equation~(24)]{doukhan2007probability}).

\subsection{Proof of \Cref{th:rosenthal_log_V_cor_2_wasserstein_non_statio}}
\label{sec:proof-crefth:r_th:rosenthal_log_V_cor_2_wasserstein_non_statio}
Without loss of generality, we can assume that $\Nnorm[1, W^{\gamma}]{\bar{g}} = 1$. Let $t \geq 0$, $g \in \Lclass_{1, W^\gamma}$ and  $\xi$ be a probability measure on $(\Xset,\Xsigma)$ satisfying $\xi(V^{1/2}) < \infty$.  First note that setting $S_n' = \sum_{k=0}^{n-1} g(X_k')$, we have using that $\MKK$ is a coupling kernel for $\MK$.
  \begin{equation}
    \label{eq:rosenthal_log_V_cor_2_wasserstein_non_statio_00}
  \PP_{\xi}(|S_n| \geq t) \leq \PP_{\pi}(|S_n| \geq t/2) + \PP_{\xi,\pi}^{\MKK}(|S_n-S_n'| \geq t/2)\eqsp,
\end{equation}
Set $c_k = \cost^{1/2}(X_k,X_k')$, $\Delta g_k = g(X_k)-g(X_k')$, $\Sigmabf^{(1/2)}_{\ell} = \sum_{\ell=0}^{k-1}c_{\ell}^{1/2}$ for $k \in\{0,\ldots,n-1\}$.
We distinguish the two cases $\gamma =0$ and $\gamma >0$.

First assume $\gamma >0$.
Then, we have setting $\varpi_{\gamma} = 1/(1+\gamma)$, using Young inequality with $1/\varpi_{\gamma} >1$ and since $\varpi_{\gamma}/(1+\varpi_{\gamma}) = 1/\gamma$ and $\bar{W}(x,x') = \{W(x) + W(x')\}/2$,
\begin{align}
  \nonumber
  |S_n-S_n'|^{\varpi_{\gamma}} &\leq \varpi_{\gamma} \Sigmabf^{(1/2)}_n  +(1-\varpi_{\gamma}) \defEns{\frac{1}{\Sigmabf^{(1/2)}_n} \sum_{k=0}^{n-1} \Delta g_k}^{1/\gamma} \\
  \nonumber
                               & \leq \varpi_{\gamma} \Sigmabf^{(1/2)}_n  +(1-\varpi_{\gamma})\max_{k\in\{0,\ldots,n-1\}}\{ \Delta g_k/c_k^{1/2} \}^{1/\gamma} \\
    \label{eq:rosenthal_log_V_cor_2_wasserstein_non_statio_2}
   & \leq \varpi_{\gamma} \Sigmabf^{(1/2)}_n  +2 (1-\varpi_{\gamma}) \max_{k\in\{0,\ldots,n-1\}}\{ c_k^{1/(2\gamma)} [W(X_k)+ W(X_k')]\} \eqsp.
\end{align}
where we have used in \eqref{eq:rosenthal_log_V_cor_2_wasserstein_non_statio_2} that $(a+b)^u \leq 2^{(u-1)_+}(a^u+b^u)$ for $a,b \geq 1$, $u \geq 0$. It is easy to verify that \eqref{eq:rosenthal_log_V_cor_2_wasserstein_non_statio_2} still holds for $\gamma =0$.
Then, we get that for $\tilde{t} \geq 0$,
\begin{align}
   \label{eq:rosenthal_log_V_cor_2_wasserstein_non_statio_2_5}
  &\PP_{\xi,\pi}^{\MKK}(|S_n-S_n'| \geq \tilde{t}) \leq   \PP_{\xi,\pi}^{\MKK}(\varpi_{\gamma} \Sigmabf^{(\beta)}_n \geq \tilde{t}^{\varpi_{\gamma}}/2)\\
  \nonumber
&  \qquad \qquad \txts +   \PP_{\xi,\pi}^{\MKK}( (1-\varpi_{\gamma}) \max_{k\in\{0,\ldots,n-1\}}\{ c_k^{1/(2\gamma)} [W(X_k)+ W(X_k')]\} \geq \tilde{t}^{\varpi_{\gamma}}/4) \eqsp.
\end{align}
We now bound separately the two terms in the right hand side. Using that $\rme^{\uplambda_1 \sum_{k=0}^{n-1}a_k} \leq \uplambda_1 \sum_{k=0}^{n-1} a_k \rme^{\uplambda_1(k+1)} + 1$ for any $\{a_k\}_{k=0}^{n-1} \in\ccint{0,1}^n$ and Jensen's inequality, 
\begin{align*}
\txts  \PP_{\xi,\pi}^{\MKK}(\Sigmabf^{(1/2)}_n \geq \tilde{t}^{\varpi_{\gamma}}/(2\varpi_{\gamma})) &\leq \txts \rme^{- \uplambda_1 \tilde{t}^{\varpi_{\gamma}}/(2\varpi_{\gamma})}
                                                                                                \PE_{\xi,\pi}^{\MKK}[\rme^{\uplambda_1 \Sigmabf^{(1/2)}_n }] \\
                                 \txts                                            & \txts
  \leq \rme^{-\uplambda_1  \tilde{t}^{\varpi_{\gamma}}/(2\varpi_{\gamma})} (1+ \uplambda_1 \sum_{k=0}^{n-1} \PE_{\xi,\pi}^{\MKK}[  c_k]^{1/2} \rme^{\uplambda_1 (k+1)}) \eqsp.
\end{align*}
Using \Cref{prop:wasser:geo} and since $c_k \leq c_k \bar{V}(X_k,X_k')$, we get with $\uplambda_1 = -\log(\ratewas)/4$ that
\begin{align}
  \nonumber
  &  \txts  \PP_{\xi,\pi}^{\MKK}(\Sigmabf^{(1/2)}_n \geq \tilde{t}^{\varpi_{\gamma}}/(2\varpi_{\gamma})) \\
  \label{eq:rosenthal_log_V_cor_2_wasserstein_non_statio_4}
  & \txts \qquad \leq \txts \rme^{\log(\ratewas) \tilde{t}^{\varpi_{\gamma}}/(8\varpi_{\gamma})}\defEns{1+(-\log(\ratewas)/4)\frac{[ \boundmetric^{m/2}  \vartconstwas \{\pi(V^{1/2}) + \xi(V^{1/2})\}]^{1/2}}{\ratewas^{1/4}(1-\ratewas^{1/4})}} \eqsp.
\end{align}

On the other hand,  we have for any $\uplambda_2 >0$,
\begin{align}
  \nonumber
 &\txts \PP_{\xi,\pi}^{\MKK}( (1-\varpi_{\gamma}) \max_{k\in\{0,\ldots,n-1\}}\{ c_k^{1/(2\gamma)} [W(X_k)+ W(X_k')]\} \geq \tilde{t}^{\varpi_{\gamma}}/4) \\
  \label{eq:rosenthal_log_V_cor_2_wasserstein_non_statio_3}
 &\qquad \qquad \txts \leq \exp\parenthese{-\frac{ \uplambda_2 \tilde{t}^{\varpi_{\gamma}}}{4(1-\varpi_{\gamma})}}  \PE_{\xi,\pi}^{\MKK}[\exp( \uplambda_2  \max_{k\in\{0,\ldots,n-1\}}A_k) ] \eqsp,
\end{align}
with $  A_k  =  c_k^{1/(2\gamma)} [W(X_k)+ W(X_k')]$.
Second, using $\rme^{u} -1 \leq u \rme^u$, and $\max_{k} c_k \leq 1$,
\begin{align*}
&  \PE_{\xi,\pi}^{\MKK}[\exp( \uplambda_2  \max_{k\in\{0,\ldots,n-1\}}A_k) ]-1 \leq \uplambda_2   \PE_{\xi,\pi}^{\MKK}[\max\limits_{k\in\{0,\ldots,n-1\}}A_k \exp( \uplambda_2 \max\limits_{k\in\{0,\ldots,n-1\}}A_k)] \\
  & \qquad\qquad\leq \uplambda_2 \sum_{k=0}^{n-1} \PE_{\xi,\pi}^{\MKK}[A_k \exp(\uplambda_2 A_k)] \\
  & \qquad\qquad \leq \uplambda_2 \sum_{k=0}^{n-1} \PE_{\xi,\pi}^{\MKK}[c_k^{1\wedge(2\gamma)^{-1}}[W(X_k)+ W(X_k')]\{V^{2\uplambda_2}(X_k) + V^{2\uplambda_2}(X_k')\}] \eqsp.
\end{align*}
Taking $\uplambda_2 = 8^{-1} \wedge (16 \gamma)^{-1}$, we obtain using Jensen inequality
\begin{align*}
  &  \PE_{\xi,\pi}^{\MKK}[\exp( \uplambda_2  \max_{k\in\{0,\ldots,n-1\}}A_k) ]-1  \leq (8^{-1} \wedge (16 \gamma)^{-1} ) \, 2\, \sup_{a \geq \rme} \{a^{4^{-1}\wedge (8\gamma)^{-1}}\log(a)\}\\
  & \qquad \qquad \times \sum_{k=0}^{n-1} \PE_{\xi,\pi}^{\MKK}[c_k^{1\wedge(2\gamma)^{-1}}\{V^{1/2(1\wedge(2\gamma)^{-1})}(X_k) + V^{1/2(1\wedge(2\gamma)^{-1})}(X_k')\}]\\
  & \leq (2^{-1} \wedge (4 \gamma)^{-1}) \sup_{a \geq \rme} \{a^{4^{-1}\wedge (8\gamma)^{-1}}\log(a)\} \sum_{k=0}^{n-1} \PE_{\xi,\pi}^{\MKK}[c_k\{V(X_k) + V(X_k')\}^{1/2}]^{(1\wedge 1/(2\gamma))} \eqsp.
\end{align*}
Using \Cref{prop:wasser:geo} and plugging the resulting bounds in \eqref{eq:rosenthal_log_V_cor_2_wasserstein_non_statio_00} completes the proof.

\subsection{Proof of \Cref{prop:pCN}}
\label{subsec:prop:pCN}
Let $\gausc = \muH(\ballH{0}{\tau})$. We use the following version of Fernique's theorem; see \citep[Theorem~2.8.5]{bogachev:1998}.
\begin{lemma}
\label{lem:bogachev:fernique}
Let $\muH$ be a centered Gaussian measure on  $(\msh,\mch)$. Then for any $\tau \in \rset_{+}$ such that $\gausc > 1/2$ and $\alphagaus = \log\{\gausc/(1-\gausc)\}/(24\tau^2)$ the following inequality holds
\begin{equation*}
\int_{\msh} \exp\bigl(\alphagaus \normH{x}^2\bigr) \rmd\muH(x) \leq \Constexpmoment_{\tau}\,,
\end{equation*}
where $\Constexpmoment_{\tau} = \gausc\bigl(\gausc/(1-\gausc)\bigr)^{1/24} + \gausc\bigl\{1 - (1/\gausc - 1)^{1 - (1+\sqrt{2})^2/6}\bigr\}^{-1}$. Moreover, for any $\betagaus \in \rset_{+}$ and $K \geq \betagaus/(2\alphagaus)$,
\begin{equation*}
\int_{\{\normH{y} \geq K\}}\exp\{\betagaus\normH{y}\}\rmd\muH(y) \leq \ConstC_{\alphagaus,\betagaus}\exp\{-\alphagaus K^2 + \betagaus K\}\eqsp,
\end{equation*}
where $\ConstC_{\tau,\betagaus} = \Constexpmoment_{\tau}\biggl(1 + \frac{\sqrt{\pi}\betagaus}{2\sqrt{\alphagaus}}\biggr)$.
\end{lemma}
\begin{proof}
The first part of the statement follows from \cite[Theorem~2.8.5]{bogachev:1998} by making the constants explicit. The second part of the statement follows from \cite[Proposition~$A.1.$]{hairer:stuart:vollmer:2012}.
\end{proof}

We first check the drift condition $\Cref{assG:kernelP_q}$. The proof essentially follows from \cite[Lemma~3.2]{hairer:stuart:vollmer:2012}, once again making  the constants explicit.
\begin{lemma}
Under the assumptions of \Cref{prop:pCN}, \Cref{assG:kernelP_q} holds with the constants
\begin{equation}
\label{eq:drift_constants_pcn}
\begin{split}
& \lambda = 1 - \muH(\ballH{0}{\constKone \rpCNconst^{a}})\bigl(1-
\exp{\bigl(-(1-\rhoH)\RpCN/2\bigr)}\bigr)\exp\left(\alphalpCN\right), \quad b = \constdriftpcnfirst \vee \constdriftpcnsecond, \\
& \constdriftpcnfirst = \Constexpmoment_{\tau}\exp\bigl\{\RpCN + (1-\rhoH^2)/(4\alphagaus)\bigr\}, \quad \constdriftpcnsecond = \ConstC_{\alphagaus,(1-\rhoH^2)^{1/2}}\exp\left\{g(\tstar) + (1-\rhoH^2)^{1/2}\constKone\right\}, \\
& g(t) = (\rhoH + (1-\rhoH^2)^{1/2} \constKone)t - \alphagaus \constKone^2 t^{2a}\eqsp, \quad
\tstar = \left(\frac{(1-\rhoH^2)^{1/2} \constKone + \rhoH}{2\alphagaus \constKone^2 a}\right)^{1/(2a-1)}, \\
& \constKone = \rpCNconst/(1-\rhoH^2)^{1/2}\eqsp, \quad \tau = \inf_{t \in \rset}\bigl\{\muH(\ballH{0}{t}) \geq 3/4\bigr\}, \\
& \Constexpmoment_{\tau} = (3/4)\left(3^{1/24} + \bigl\{1 - 3^{(1+\sqrt{2})^2/6 - 1}\bigr\}^{-1}\right)\eqsp, \quad \alphagaus = \log(3)/(24\tau^2)\eqsp.
\end{split}
\end{equation}
\end{lemma}
\begin{proof}
Let $V(x)= \exp(\normH{x})$ and $\prop(x,y) = \rhoH x + (1-\rhoH^2)^{1/2}y$. Note that it holds
\begin{equation}
\label{eq:MK_V_func_image}
\MK V(x) = \int_{\msh}\biggl(\exp\{\normH{x}\}(1-\alphaH(x,y)) + \exp\{\normH{\prop(x,y)}\}\alphaH(x,y)\biggr)\rmd \muH(y)\eqsp.
\end{equation}
Then for $x \in \ballH{0}{\RpCN}$, using $\normH{\prop(x,y)} \leq \normH{x} + (1-\rhoH^2)^{1/2}\normH{y}$, we get
\begin{align}
\MK V(x) &\leq \exp\{\normH{x}\} \int_{\msh} \exp\{(1-\rhoH^2)^{1/2}\normH{y}\} \rmd \muH(y) \nonumber \\
&\overset{(a)}{\leq} \exp\{\RpCN + (1-\rhoH^2)/(4\alphagaus)\} \int_{\msh} \exp\{\alphagaus \normH{y}^2\} \rmd \muH(y) \nonumber \\
&\overset{(b)}{\leq} \Constexpmoment_{\tau}\exp\bigl\{ \RpCN + (1-\rhoH^2)/(4\alphagaus)\bigr\} =: \constdriftpcnfirst \label{eq:const_b_1_pcn_def} \eqsp.
\end{align}
In the above, (a) is due to inequality $\exp\{\gamma t\} \leq \exp\{\gamma^2/(4\Delta) + \Delta t^2\}$, $t, \gamma, \Delta > 0$, and (b) is due to \Cref{lem:bogachev:fernique} applied with $\tau$ given in \eqref{eq:drift_constants_pcn}. Further, using \Cref{assum:d-small-set-pCN}, for $x \not\in \ballH{0}{\RpCN}$, it holds $\rpCNconst \normH{x}^a \leq (1-\rhoH)\normH{x}/2$. This implies
\begin{equation}
\label{eq:bound_V_ball}
\sup_{y \in \ballH{\rhoH x}{\rpCNconst \normH{x}^a}} V(y) \leq \smallconst V(x)\eqsp, \quad \smallconst = \exp{\bigl(-(1-\rhoH)\RpCN/2\bigr)}\eqsp.
\end{equation}
Let us now define $\eventA = \{y \in \msh | (1-\rhoH^2)^{1/2}\normH{y} \leq \rpCNconst\normH{x}^a\}$. Then for $x \not\in \ballH{0}{\RpCN}$ we use \eqref{eq:MK_V_func_image} and split integration over $\msh$ into integration over $\eventA$ and $\msh\setminus\eventA$. For $y \in \eventA, z(x,y) \in \ballH{\rhoH x}{\rpCNconst \normH{x}^a}$, thus \eqref{eq:bound_V_ball} implies
\begin{align*}
& \int_{\eventA}\biggl(\exp\{\normH{x}\}(1-\alphaH(x,y)) + \exp\{\normH{\prop(x,y)}\}\alphaH(x,y)\biggr)\rmd \muH(y) \\
&\leq \int_{\eventA}\biggl(\exp\{\normH{x}\}(1-\alphaH(x,y)) + \smallconst \exp\{\normH{x}\} \alphaH(x,y)\biggr)\rmd \muH(y) \\
&= \exp\{\normH{x}\}\bigl(\muH(\eventA) - (1 - \smallconst)\int_{\eventA}\alphaH(x,y)\rmd \muH(y)\bigr) \\
&\overset{(a)}{\leq}  \exp\{\normH{x}\}\muH(\eventA)\bigl(1 - (1 - \smallconst)\exp\left(\alphalpCN\right)\bigr)\eqsp.
\end{align*}
In the above, (a) is due to lower bound on $\alphaH(x,y)$, which follows from \Cref{assum:d-small-set-pCN}. Setting $\constKone = \rpCNconst/(1-\rhoH^2)^{1/2}$ and using \Cref{lem:bogachev:fernique} with $K = \constKone \normH{x}^a$ and $\betagaus = (1-\rhoH^2)^{1/2}$,
\begin{align*}
& \int_{\msh\setminus\eventA}\biggl(\exp\{\normH{x}\}(1-\alphaH(x,y)) + \exp\{\normH{\prop(x,y)}\}\alphaH(x,y)\biggr)\rmd \muH(y) \\
& \quad \leq (1 - \muH(\eventA))\exp\{\normH{x}\} + \exp\{\rhoH\normH{x}\}\int_{\msh\setminus\eventA}\exp\{\betagaus\normH{y}\} \rmd \muH(y) \\
& \quad \leq (1 - \muH(\eventA))\exp\{\normH{x}\} + \ConstC_{\tau,\betagaus}\exp\biggl\{\rhoH\normH{x} - \alphagaus \constKone^2 \normH{x}^{2a} + \betagaus \constKone \normH{x}^a\biggr\}\eqsp,
\end{align*}
where $\alphagaus$ is defined in \eqref{eq:drift_constants_pcn} and $\ConstC_{\tau,\betagaus}$ defined in \Cref{lem:bogachev:fernique}. Combining the above inequalities and \eqref{eq:MK_V_func_image}, for $x \not\in \ballH{0}{\RpCN}$,
\begin{equation}
 \label{eq:bound_outside_ball}
 \begin{split}
\MK V(x) &\leq \bigl(1 - \muH(A)(1-\smallconst)\exp\left(\alphalpCN\right)\bigr)V(x) \\
&+ \ConstC_{\alphagaus,\betagaus}\sup_{t \geq 0}\exp\biggl\{\rhoH t + \betagaus \constKone t^{a} - \alphagaus \constKone^2 t^{2a}\biggr\}\eqsp.
 \end{split}
\end{equation}
We complete the proof combining \eqref{eq:const_b_1_pcn_def}, \eqref{eq:bound_outside_ball}, and noting that $\muH(A) \geq \muH(\ballH{0}{\constKone \rpCNconst^{a}})$.
\end{proof}

\textbf{Proof of \Cref{ass:cost_fun} and \Cref{assG:kernelP_q_contractingset_m}.} It is easy to see that assumption \Cref{ass:cost_fun} is satisfied with $\cost(x,x') = 1 \wedge [\normH{x-x'}/\varepsilonH]$ and $\pcost=1$ as soon as $\varepsilonH \leq 1$. In our proof of \Cref{assG:kernelP_q_contractingset_m} we use the synchronous coupling suggested in \cite[Section~3.1.2]{hairer:stuart:vollmer:2012},
\begin{align*}
  X_{k+1} &= X_k \indiacc{U_{k+1} > \alphaH(X_k,Z_{k+1}) }+  \defEns{\rhoH X_k + (1-\rhoH^2)^{1/2} Z_{k+1}} \indiacc{U_{k+1} \leq \alphaH(X_k,Z_{k+1})}, \, X_0 = x \\
  X'_{k+1} &= X'_k \indiacc{U_{k+1} > \alphaH(X'_k,Z_{k+1}) }+  \defEns{\rhoH X'_k + (1-\rhoH^2)^{1/2} Z_{k+1}} \indiacc{U_{k+1} \leq \alphaH(X'_k,Z_{k+1})}, X'_0 = x'\eqsp.
\end{align*}
The associated coupling kernel is denoted $\MKK$. The first part of \Cref{assG:kernelP_q_contractingset_m} follows from the following lemma:
\begin{lemma}
\label{lem:pcn_contract_1step}
Let $\cost(x,x') = 1 \wedge [\normH{x-x'}/\varepsilonH]$ with
\begin{equation}
\label{eq:varepsilon_H_def}
\varepsilonH = \pcngamma/(2\Lippcn)\eqsp,
\end{equation}
where we have introduced
\begin{equation}
\label{eq:contact_small_pcn_const}
\begin{split}
&\pcngamma = \left(\lbprobpcn \muH\left(\ballH{0}{(1-\rhoH^2)^{-1/2}\RpCN}\right) \wedge \exp\left(\alphalpCN\right)\muH(\ballH{0}{\constKone \RpCN^a}) \right)(1-\rhoH)/2 \\
&\lbprobpcn = \exp\bigl\{-\sup_{y \in \ballH{0}{2\RpCN+1}}\potU(y) + \inf_{y \in \ballH{0}{2\RpCN+1}}\potU(y)\bigr\}\eqsp,
\end{split}
\end{equation}
where $\constKone$ is defined in \eqref{eq:drift_constants_pcn}. Then, for $x,x' \in \msh, \normH{x-x'} \geq \varepsilonH$, it holds
\begin{equation*}
\MKK \cost(x,x') \leq \cost(x,x')\eqsp.
\end{equation*}
Moreover, for $x,x' \in \msh, \normH{x-x'} < \varepsilonH$, it holds
\begin{equation*}
\MKK \cost(x,x') \leq (1- \pcngamma)\cost(x,x')\eqsp,
\end{equation*}
\end{lemma}
\begin{proof}
Let $\normH{x-x'} \geq \varepsilonH$ then $\cost(x,x') = 1$. The statement follows from $\MKK \cost(x,x') = \PE[\cost(X_1,X'_1)] \leq 1$. Consider the case $\cost(x,x') < 1$. Since $\varepsilonH \leq 1$ then either $x,x' \in \ballH{0}{\RpCN+1}$ or $x,x' \not \in \ballH{0}{\RpCN}$. We consider these cases separately. Let $x,x' \in \ballH{0}{\RpCN+1}$. Introduce the following events, $\eventA = \{(1-\rhoH^2)^{1/2}\normH{Z_1} \leq \RpCN\}$, $\eventB_1 = \{U_{1} \leq \alphaH(x,Z_{1}),U_{1} \leq \alphaH(x',Z_{1})\}$, $\eventB_2 = \{U_{1} > \alphaH(x,Z_{1}),U_{1} > \alphaH(x',Z_{1})\}$ and $\eventB_3 = \RandSpace \setminus (\eventB_1 \cup \eventB_2)$. Note that $\indiacc{\eventB_1}\cost(X_1,X'_1) = \rhoH\cost(x,x')$ and $\indiacc{\eventB_2}\cost(X_1,X'_1) = \cost(x,x')$. We get
\begin{align}
\MKK \cost(x,x')
&= \PE[\indiacc{\eventA \cap \eventB_1}\cost(X_1,X'_1)] + \PE[\indiacc{\eventA \cap \eventB_2}\cost(X_1,X'_1)] + \PE[\indiacc{\overline{\eventA} \cap (\eventB_1 \cup \eventB_2)}\cost(X_1,X'_1)] \nonumber \\
&\qquad + \PE[\indiacc{\eventB_3}\cost(X_1,X'_1)] \nonumber \\
&\overset{(a)}{\leq} \PP(\eventA)\bigl(\PP(\eventB_1|\eventA)\rhoH\cost(x,x') + \PP(\eventB_2|\eventA)\cost(x,x')\bigr) + (1-\PP(\eventA))\cost(x,x') \label{eq:Kc_bound} \\
&\qquad +\int_{\msh}\bigl|\alphaH(x,y) - \alphaH(x',y)\bigr| \rmd \muH(y) \nonumber\eqsp.
\end{align}
Here, (a) follows from the representation
\begin{align*}
\PE[\indiacc{\eventB_3}\cost(X_1,X'_1)]
&= \int_{\msh}\int_{0}^{1}\bigl[\cost(x,\rhoH x' + (1-\rhoH^2)^{1/2}y)\indiacc{\alphaH(x',y) \leq u \leq \alphaH(x,y)} \\
&\qquad +\cost(\rhoH x + (1-\rhoH^2)^{1/2}y,x')\indiacc{\alphaH(x,y) \leq u \leq \alphaH(x',y)}\bigr]\,\rmd u \rmd \muH(y)\eqsp.
\end{align*}
We use $\PP(\eventB_2|\eventA) \leq 1 - \PP(\eventB_1|\eventA)$ together with with $\PP(\eventB_1|\eventA) \geq \lbprobpcn$. The latter follows from \eqref{eq:accept_pcn} and definition of the set $\eventA$. Since $f(t) = 1 \wedge \exp\{t\}$ is $1$-Lipschitz we may use definition \eqref{eq:accept_pcn} and \Cref{assum:potU-lipshitz} to obtain
\begin{align*}
&\int_{\msh}\bigl|\alphaH(x,y) - \alphaH(x',y)\bigr| \rmd \muH(y) \leq \int_{\msh}\biggl| 1 \wedge \exp\parenthese{-\potU(\rhoH x+ (1-\rhoH^2)^{1/2} y) + \potU(x)} - \\
&\qquad 1 \wedge \exp\parenthese{-\potU(\rhoH x'+ (1-\rhoH^2)^{1/2} y) + \potU(x')}\biggr| \rmd \muH(y) \leq  2\Lippcn\normH{x-x'} \leq 2\varepsilonH\Lippcn\cost(x,x')\eqsp.
\end{align*}
Putting together the obtained inequalities, we arrive at an estimate of the form
\begin{align*}
\MKK \cost(x,x') \leq \bigl(1 - \lbprobpcn \muH\bigl(\ballH{0}{\RpCN}\bigr)(1-\rhoH) + 2 \varepsilonH \Lippcn\bigr)\cost(x,x') \leq (1- \pcngamma)\cost(x,x').
\end{align*}
The last inequality follows from the choice of $\varepsilonH$.
\par
Consider the case $x,x' \not \in \ballH{0}{\RpCN}$. Define $\eventC = \{\omega \in \RandSpace | (1-\rhoH^2)^{1/2}\normH{Z_1(\omega)} \leq \rpCNconst(\normH{x} \wedge \normH{x'})^{a}\}$. Repeating the argument \eqref{eq:Kc_bound} we get
\begin{align*}
\MKK \cost(x,x')
&\leq \PP(\eventC)\bigl(\PP(\eventB_1|\eventC)\rhoH\cost(x,x') + \PP(\eventB_2|\eventC)\cost(x,x')\bigr) + (1-\PP(\eventC))\cost(x,x') \\
&\qquad +\int_{\msh}\bigl|\alphaH(x,y) - \alphaH(x',y)\bigr| \rmd \muH(y) \nonumber\eqsp.
\end{align*}
To complete the proof it remains to note that $\PP(\eventB_2|\eventC) \leq 1 - \PP(\eventB_1|\eventC) \le \exp\left(\alphalpCN\right)$, where the last inequality follows from \Cref{assum:d-small-set-pCN}.
\end{proof}

Now we check the second part of \Cref{assG:kernelP_q_contractingset_m}.
\begin{lemma}
Under the assumptions of \Cref{prop:pCN}, it holds
\begin{equation}
\label{eq:K_m_bound}
\MKK^m \cost(x,x') \leq (1 - \minorwas \indi{\CKset}(x,x'))\cost(x,x')\eqsp,
\end{equation}
where $\CKset$, $\Rassumapcn$ and $m$ are defined in \Cref{prop:pCN},
\begin{equation}
\label{eq:epsilon_pcn_def}
\minorwas = \pcngamma \wedge \left(\lbprobpcn \muH(\ballH{0}{\Rassumapcn_{m}})\right)^{m}/2\eqsp, \quad \Rassumapcn_{m} = \frac{\Rassumapcn}{m(1-\rhoH^2)^{1/2}}\eqsp,
\end{equation}
and $\pcngamma$, $\lbprobpcn$ are defined in \eqref{eq:contact_small_pcn_const}.
\end{lemma}
\begin{proof}
Note that in case $(x,x') \not\in \CKset$ we can use \Cref{lem:pcn_contract_1step} which implies $\MKK^{m}\metricc(x,x') \leq \metricc(x,x')$ for any $m \in \nset$. Hence, \eqref{eq:K_m_bound} follows.
\par
Assume that $(x,x') \in \CKset$. Consider first the case $\metricc(x,x') = 1$, that is, $\normH{x-x'} \geq \varepsilonH$. Let $m \in \nset$ be a number to be chosen later and introduce $\eventB_{m} = \{ U_{k} \leq \alphaH(X_k,Z_{k}),U_{k} \leq \alphaH(X'_k,Z_{k}), \, k = 1,\dots,m\}$. Note that $\eventB_{m}$ is an event where first $m$ proposals were accepted both for sequence $(X_{k})_{k \in \nset}$ and $(X'_{k})_{k \in \nset}$. Then, using that $\cost(x,y) \leq \normH{x-y}/\varepsilonH$ and $\cost(x,y) \leq 1, \, x,y \in \msh$, we get
\begin{align}
\MKK^{m}\metricc(x,x')
&= \PE[\cost(X_{m},X'_{m})] = \PE[\cost(X_{m},X'_{m})\indiacc{\eventB_m}] + \PE[\cost(X_{m},X'_{m})\indiacc{\overline{\eventB}_m}] \nonumber \\
&\leq \PP(\eventB_m)\rhoH^{m}\normH{x-x'}/\varepsilonH + 1-\PP(\eventB_m)\eqsp. \label{eq:K_m_bound_intermediate}
\end{align}
We choose $m = \log(\varepsilonH/(4\Rassumapcn))/\log\rhoH$ and use $\metricc(x,x') = 1$. Then \eqref{eq:K_m_bound_intermediate} implies
\begin{equation}
\label{eq:MKK_m_step_bound_m_step_prob}
\MKK^{m}\metricc(x,x') \leq \bigl(1 - \PP(\eventB_{m})/2\bigr)\metricc(x,x')\eqsp.
\end{equation}
It remains to lower bound $\PP(\eventB_{m})$. Recall the definition \eqref{eq:epsilon_pcn_def} of $\Rassumapcn_{m}$. It follows
\begin{align*}
\PP(\eventB_{m}) &\geq \PP\bigl(\eventB_{m} | \cap_{k=1}^{m}\bigl\{\normH{Z_{k}} \leq \Rassumapcn_{m}\bigr\}\bigr)\PP(\normH{Z_1} \leq \Rassumapcn_{m})^{m} \\
&\geq \left(\lbprobpcn \muH(\ballH{0}{\Rassumapcn_{m}})\right)^{m}\eqsp,
\end{align*}
where $\lbprobpcn$ is defined in \eqref{eq:contact_small_pcn_const}. In case $\metricc(x,x') < 1$, \Cref{lem:pcn_contract_1step} implies
\begin{equation}
\label{eq:MKK_m_step_bound_contractive}
\MKK^{m}\metricc(x,x') \leq \MKK \metricc(x,x') \leq (1-\pcngamma)\metricc(x,x')\eqsp.
\end{equation}
Now the statement follows by combining \eqref{eq:MKK_m_step_bound_m_step_prob} and \eqref{eq:MKK_m_step_bound_contractive}.
\end{proof}

\subsection{Proof of \Cref{theo:SGD}}
\label{sec:proof_drift_sgd}
The proof of \Cref{ass:cost_fun} is immediate. We preface the proof of the drift condition \Cref{assG:kernelP_q} by the following instrumental lemma. Recall that $\kapf= \muf \Lf/ (\muf+\Lf)$.
\begin{lemma}
\label{lem:descent-SGD}
Assume \Cref{ass:sgd_field}. Then, for all $\gamma \in \ocint{0,1/(\muf+\Lf)}$ and $k \in \nset$,
\begin{equation*}
\norm{\theta_{k+1} - \thetas}^2 \leq (1 - \gamma \kapf) \norm{\theta_k - \thetas}^2
+ (\gamma / \kapf + 2 \gamma^2) \norm{\fieldH_{\theta_k}(\YSGD_{k+1}) - \nabla \objf(\theta_k)}^2 \eqsp.
\end{equation*}
\end{lemma}
\begin{proof}
Expanding the recurrence \eqref{eq:def_SGD} and using $\nabla \objf(\thetas) = 0$,
\begin{multline*}
\norm{\theta_{k+1}- \thetas}^2 \leq \| \theta_k - \thetas \|^2 - 2 \gamma \ps{\fieldH_{\theta_k}(\YSGD_{k+1}) - \nabla \objf(\theta_k)}{\theta_k - \thetas} - 2 \gamma \ps{\nabla \objf(\theta_k) - \nabla \objf(\thetas)}{\theta_k - \thetas} \\
+ 2 \gamma^2 \norm{\fieldH_{\theta_k}(\YSGD_{k+1}) - \nabla \objf(\theta_k)}^2 + 2 \gamma^2 \norm{\nabla \objf(\theta_k) - \nabla \objf(\thetas)}^2 \eqsp.
\end{multline*}
\cite[Theorem~2.1.12]{nesterov:2004} shows that, for all $(\theta,\theta') \in \rset^{2d}$,
\[
\ps{\nabla \objf(\theta)-\nabla \objf(\theta')}{\theta-\theta'} \geq \kapf \norm{\theta-\theta'}^{2}+\frac{1}{\muf+\Lf}\norm{\nabla \objf(\theta)-\nabla \objf(\theta')}^{2} \eqsp.
\]
Using that $\gamma \leq 1/(\muf+\Lf)$, and $|\ps{a}{b}| \leq (\varepsilon^2/2) \norm{a}^2 + 1/(2 \varepsilon^2) \norm{b}^2$ for $\varepsilon >0$ and $(a,b) \in \rset^{2d}$, we get
\[
\norm{\theta_{k+1}-\thetas}^2 \leq ( 1 - 2 \gamma \kapf + \gamma \varepsilon^2)
\norm{\theta_k - \thetas}^2 + (\gamma/ \varepsilon^2 + 2 \gamma^2) \norm{\fieldH_\theta(\YSGD_{k+1}) - \nabla \objf(\theta_k)}^2 \eqsp.
\]
We complete the proof by taking $\varepsilon^2 = \kapf$.
\end{proof}
Now we are ready to check \Cref{assG:kernelP_q}.

\begin{lemma}
Under the assumptions of \Cref{theo:SGD}, \Cref{assG:kernelP_q} holds with the constants $\lambda$ and $b$ given by 
\begin{equation}
\label{eq:ass_A1_PR_constants}
\begin{split}
\lambda &= \rme^{-\gamma\kapf/(2\tsgvarfac)}\eqsp, \quad \tsgvarfac = 2\sgvarfac(\rme+1)/(\rme-1)\eqsp, \\
b &= \gamma\bigl(1/\kapf + 2 \gamma + \kapf/(2\tsgvarfac)\bigr)\exp\left(2 + (2\tsgvarfac)^{-1} + (2\gamma \kapf+1)/\kapf^2\right)\eqsp.
\end{split}
\end{equation}
\end{lemma}
\begin{proof}
Under \Cref{ass:sgd_noise_exp_mom}, \Cref{subsec:proof:eq:bound_exp_sgd} implies for all $\theta \in \rset^{\dims}$
\begin{equation}
  \label{eq:bound_exp_sgd}
  \PE\parentheseDeux{\exp\parenthese{\normLigne{\fieldH_{\theta}(\YSGD) - \nabla \objf(\theta)}^2/\tsgvarfac}}
  \leq \rme \eqsp, \text{ with } \tsgvarfac = 2\sgvarfac(\rme+1)/(\rme-1) \eqsp.
\end{equation}
In particular, Jensen's inequality implies
  \begin{equation}
    \label{eq:bound_var_sgd}
    \PE[\normLigne{\fieldH_{\theta}(\YSGD) - \nabla \objf(\theta)}^2] \leq \tsgvarfac \eqsp.
  \end{equation}
Applying Jensen's inequality one more time to \eqref{eq:bound_exp_sgd}, and using $\gamma/\kapf + 2\gamma^2 \leq 1$,
\begin{equation*}
\PE\left[\exp\parenthese{(\gamma / \kapf + 2 \gamma^2) \norm{\fieldH_{\theta_k}(\YSGD_{k+1}) - \nabla \objf(\theta_k)}^2/\tsgvarfac}\right] \leq \rme^{\gamma / \kapf + 2 \gamma^2} \eqsp.
\end{equation*}
\Cref{lem:descent-SGD} implies that for any $\theta \in \rset^d$,
\begin{equation}
\label{eq:Q_gamma_image_bound}
\MKSGD_{\gamma} V_{1/\tsgvarfac}(\theta) \leq \exp\parenthese{-(\gamma \kapf/\tsgvarfac) \norm{\theta - \thetas}^2} \rme^{\gamma / \kapf + 2 \gamma^2} V_{1/\tsgvarfac}(\theta)   \eqsp.
\end{equation}
Now we consider the two cases separately. For $\theta \in\rset^d$ satisfying $\norm{\theta-\thetas}^2 \geq M_{\objf} = 1/2 + (2\gamma \kapf+1)\tsgvarfac/\kapf^2$, we have
\begin{equation}
\label{eq:7_drift_sgd_1}
\MKSGD_{\gamma} V_{1/\tsgvarfac}(\theta) \leq \rme^{-\gamma\kapf/(2\tsgvarfac)} V_{1/\tsgvarfac}(\theta) \eqsp.
\end{equation}
On the other hand, for $\theta \in \rset^d$ satisfying $\norm{\theta-\thetas}^2 \leq M_{\objf}$, we have using \eqref{eq:Q_gamma_image_bound} and $\rme^t -1 \leq t \rme^t$,
\begin{align}
\nonumber
\MKSGD_{\gamma} V_{1/\tsgvarfac}(\theta) &\leq \rme^{-\gamma\kapf/(2\tsgvarfac)} V_{1/\tsgvarfac}(\theta) + \{\rme^{\gamma / \kapf + 2 \gamma^2}-\rme^{-\gamma\kapf/(2\tsgvarfac)}\}V_{1/\tsgvarfac}(\theta) \\
\label{eq:8_drift_sgd_2}
& \leq \rme^{-\gamma\kapf/(2\tsgvarfac)} V_{1/\tsgvarfac}(\theta) + \gamma\bigl(1/\kapf + 2 \gamma + \kapf/(2\tsgvarfac)\bigr) \exp(2 + M_{\objf}/\tsgvarfac) \eqsp.
\end{align}
Combining \eqref{eq:7_drift_sgd_1} and \eqref{eq:8_drift_sgd_2} implies the statement.
\end{proof}

To check \Cref{assG:kernelP_q_contractingset_m}, we use the following synchronous coupling construction
\begin{equation*}
 \theta_{k+1} = \theta_k - \gamma  \fieldH_{\theta_k}(\YSGD_{k+1}) \quad \text{and}
 \quad  \theta'_{k+1} = \theta'_k - \gamma  \fieldH_{\theta_k'}(\YSGD_{k+1}) \eqsp,
\end{equation*}
\ie\ we use the same noise $Y_{k+1}$ at each iteration. We denote by $\MKKSGD_\gamma$ the associated coupling kernel.

\begin{lemma}
Under the assumptions of \Cref{theo:SGD}, it holds
\begin{equation*}
\MKKSGD_\gamma^m \cost(\theta,\theta') \leq (1 - \minorwas \indi{\CKset}(\theta,\theta'))\cost(\theta,\theta')\eqsp,
\end{equation*}
where $\CKset$, $\minorwas$, and $m$ are defined in \Cref{theo:SGD}.
\end{lemma}
\begin{proof}
We first note that \cite[Proposition 2]{dieuleveut2020bridging} implies that for any $\theta,\theta' \in \rset^{\dims}$, it holds
\begin{align*}
\MKKSGD_\gamma\norm{\theta - \theta'}^2 \leq \bigl(1 - \minorwas\bigr) \norm{\theta - \theta'}^2\eqsp,
\end{align*}
where $\minorwas = 2\muf \gamma (1-\gamma \Lf/2)$. Hence, in case $\theta,\theta' \in \rset^{\dims}$ with $\cost(\theta,\theta') < 1$ it holds
\begin{align*}
\MKKSGD_\gamma^{m}\cost(\theta,\theta') \leq \MKKSGD_\gamma^{m}\norm{\theta - \theta'}^2 \leq (1-\minorwas)\cost(\theta,\theta')
\end{align*}
for any $m \in \nset$. Consider now $\theta,\theta'$ with $\cost(\theta,\theta') = 1$. Then, with $m = \lceil \log(4\Rsgd^2)/\log(1/(1-\minorwas)) + 1 \rceil$,
\begin{align*}
\MKKSGD_\gamma^{m}\cost(\theta,\theta') \leq \MKKSGD_\gamma^{m}\norm{\theta - \theta'}^2 \leq (1-\minorwas)^{m}\cost(\theta,\theta') \leq 4\Rsgd^2(1-\minorwas)^{m} < (1-\varepsilon)\cost(\theta,\theta')
\end{align*}
\end{proof}

\subsection{Proof of \Cref{propo:bias}}
\label{sec:proof-crefpropo:bias}
  Let $\gamma \in \ooint{0,1/\Lf}$.
  Consider $\theta_0$ with distribution $\pi_{\gamma}$   and $\theta_1$ defined by \eqref{eq:def_SGD}. Then, by \Cref{theo:SGD}, $\theta_1$ has also distribution $\pi_{\gamma}$ which implies taking expectation in \eqref{eq:def_SGD} and rearranging terms that
  \begin{equation}
    \label{eq:4}
    0 = \int_{\rset^d} \nabla \objf(\theta) \rmd \pi_{\gamma}(\theta) \eqsp.
  \end{equation}
  In addition, \cite[]{nesterov:2004} implies that for any $\theta \in\rset^d$,
  \begin{equation*}
    \norm{\nabla \objf(\theta) - \nabla \objf(\thetas) - \nabla^2 \objf(\thetas)\{\theta - \thetas\}} \leq \LipHessianf\norm{\theta-\thetas}/2 \eqsp.
  \end{equation*}
  Plugging this result in \eqref{eq:4} and using $\nabla \objf(\thetas) =0$, we obtain that
  \begin{equation*}
    \label{eq:5}
    \norm{\int_{\rset^d} \nabla^2 \objf(\thetas)\{\theta - \thetas\} \rmd \pi_{\gamma}(\theta)} \leq
    (\LipHessianf/2)\int_{\rset^d} \norm{\theta-\thetas}^2  \rmd \pi_{\gamma}(\theta) \eqsp.
  \end{equation*}
Using   $\norm{\nabla^2 \objf(\thetas)\{\theta - \thetas\} } \geq \muf \norm{\theta-\thetas}$ for any $\theta \in\rset^d$  by \Cref{ass:sgd_field}, Jensen inequality and \Cref{theo:SGD}  complete the proof.

\section{Auxiliary proofs}
We start this part with a technical lemma about subgaussian vectors.
\begin{lemma}
\label{subsec:proof:eq:bound_exp_sgd}
Let $Y \in \rset^d$ be a norm-subgaussian vector with variance factor $\sigma^2 < \infty$. Then
\[
\PE\left[\exp\parenthese{\normLigne{Y}^2/\sigmaconst^2}\right] \leq \rme
\]
with $\sigmaconst^2 = 2\sigma^2(\rme+1)/(\rme-1)$.
\end{lemma}
\begin{proof}
By the definition of norm-subgaussian vector (see \Cref{ass:sgd_noise_exp_mom}), it holds for $c > \sigma\sqrt{2}$ that
\begin{align*}
\PE\left[\exp\parenthese{\normLigne{Y}^2/c^2}\right]
&= \int_{0}^{+\infty}\PP\left(\exp\parenthese{\normLigne{Y}^2/c^2} \geq t\right)\rmd t \\
&\leq 1 + (2/c^2)\int_{0}^{+\infty}\PP\left(\normLigne{Y} \geq u\right)\exp\parenthese{u^2/c^2}u\rmd u \\
&\leq 1 + (4/c^2)\int_{0}^{+\infty}\exp\parenthese{-u^2\bigl(1/(2\sigma^2) - 1/c^2\bigr)}u\rmd u = 1 + \frac{4\sigma^2}{c^2-2\sigma^2}\,.
\end{align*}
Now the proof follows by letting $c^2 = \sigmaconst^2$.
\end{proof}
\begin{lemma}
\label{lem:rate_UGE}
Assume \Cref{assG:kernelP_q} and \Cref{assG:kernelP_q_smallset}. Then for any $\alpha \in (0,1]$, $x \in \Xset$, $n \in \nset$, it holds that
\begin{equation}
\label{eq:V-geometric-better-rate-appendix}
\Vnorm[V^\alpha]{\MK^n(x, \cdot) - \pi}
 \leq 2 \{\cmconstv \ratev^n \pi(V)   V(x) \}^{\alpha}
\eqsp.
\end{equation}
\end{lemma}
\begin{proof}
Let $\xi$ and $\xi'$ be arbitrary measures such that $\xi(V), \xi'(V) < \infty$. Then, using the definition of the $V$-norm and Jensen's inequality,
\begin{align*}
\Vnorm[V^\alpha]{\xi - \xi'} &= \int |\xi - \xi'|(\rmd x) V^{\alpha}(x) = 2\tvnorm{\xi - \xi'} \int \frac{|\xi - \xi'|(\rmd x)}{2\tvnorm{\xi - \xi'}} V^\alpha(x) \\
&\leq \{ 2\tvnorm{\xi - \xi'} \}^{1-\alpha} \left( \int |\xi- \xi'|(\rmd x) V(x) \right)^\alpha \eqsp.
\end{align*}
If we replace $\xi \leftarrow \delta_x P^n$ and $\xi' \leftarrow \pi$, we get 
\begin{equation}
\Vnorm[V^\alpha]{\delta_x P^n - \pi} \leq 
\{ 2\tvnorm{\delta_x P^n - \pi} \}^{1-\alpha} \left( \Vnorm[V]{\delta_x P^n - \pi} \right)^\alpha\eqsp.
\end{equation}
Now \eqref{eq:V-geometric-better-rate-appendix} follows from $\tvnorm{\delta_x P^n - \pi} \leq 1$ and equation~\eqref{eq:V-geometric-coupling-general}.
\end{proof}
\begin{lemma}
\label{lem:geom_ergodicity_variance_bound}
Assume \Cref{assG:kernelP_q}, \Cref{assG:kernelP_q_smallset}, and let $q \in \nsets$. Then for any $g \in \mrl_{V^{1/(2q)}}$ and $n \in \nsets$, it holds that 
\begin{equation}
\label{eq:var_bound}
\PVar[\pi](S_n) \leq 5 n c^{1/2} \ratev^{-1/2} \{\log{1/\ratev}\}^{-1} \pi(V)^{3/2} \Vnorm[V^{1/2}]{\bar{g}}^{2}\eqsp.
\end{equation}
\end{lemma}
\begin{proof}
Applying \Cref{lem:rate_UGE} with $\alpha = 1/2$, we get
\begin{align*}
\PVar[\pi](S_n) 
&= \PVar[\pi](\sum_{i=0}^{n-1}\bar{g}(X_i)) = n \PVar[\pi](g) + \sum_{i=0}^{n-2}\sum_{\ell=1}^{n-2-i}\PE_{\pi}[\bar{g}(X_i)\bar{g}(X_{i+\ell})] \\
&\leq n \PVar[\pi](g) + 2 \Vnorm[V^{1/2}]{\bar{g}} \sum_{i=0}^{n-2}\sum_{\ell=1}^{n-2-i}\PE_{\pi}\bigl[|g(X_{i})| \{c \ratev^{\ell} \pi(V) V(X_i)\}^{1/2}\bigr] \\
&\leq n \PVar[\pi](g) + 2 c^{1/2} \pi(V)^{3/2} \Vnorm[V^{1/2}]{\bar{g}}^{2} \sum_{i=0}^{n-2}\sum_{\ell=1}^{n-2-i} \ratev^{\ell/2} \\
&\leq 5 n c^{1/2} \ratev^{-1/2} \{\log{1/\ratev}\}^{-1} \pi(V)^{3/2} \Vnorm[V^{1/2}]{\bar{g}}^{2}\eqsp.
\end{align*}
To complete the proof it remains to note that $\Vnorm[V^{1/2}]{\bar{g}} \leq \Vnorm[V^{1/(2q)}]{\bar{g}}$.
\end{proof}
We end this section with a technical bound on the scaling of coefficients $\ConstB_{0}(u,q)$ defined in \eqref{eq: B_u_q_def_new}. 
\begin{lemma}
\label{lem:scale_B_gamma}
Let $\ConstB_{0}(u,q)$ be the coefficient defined in \eqref{eq: B_u_q_def_new}. Then for $u \in \{1,\ldots,q-1\}$ it holds that
\begin{equation}
\label{eq:non_unif_bound_B_u_q}
\ConstB_{0}(u,q) \leq c_{1} (2q)^{2q} (2q-u)^{2(2q-u)} \rme^{-(2q-u)}\eqsp,
\end{equation}
where $c_{1} = \rme^{2}\sqrt{2}$. Moreover, it holds that 
\[
\ConstB_{0}(u,q) \leq c_{1} q^{6q} 2^{7q} \rme^{-2q}\eqsp.
\]
\end{lemma}
\begin{proof}
Recall that due to \eqref{eq:bound-Balpha}, it holds that
\[
\ConstB_{0}(u,q)
\leq \frac{(2q)!}{u!} \binom{2q-u-1}{u-1} \bigl((2q-2u+2)!\bigr)^{2} 2^{2(u-1)}\eqsp.
\]
Now, using the bound of \cite[Theorem~2]{guo_gamma_func}, we use the following upper bound on the Gamma function, valid for $x \geq 0$:
\[
\Gamma(x+1) \leq (x+1)^{x+1/2}e^{-x}\,.
\]
Application of this bound together with $1 + x \leq \rme^{x}$ yield 
\begin{align*}
\ConstB_{0}(u,q) 
&\leq \frac{(2q)!}{u!} \frac{(2q-u-1)!}{(u-1)!} (2q-2u+2)! (2q-2u+2)^{2} 2^{2(u-1)} \\
&\leq (2q)^{2q} (2q-2u+2)^{2q-2u+2} (2q-u-1)^{2q-u-1} \rme^{-2q+2} 2^{u+1/2}\,.
\end{align*}
This yields the upper bound \eqref{eq:non_unif_bound_B_u_q}. 
The uniform upper bound can be obtained from the expression above letting $u = 1$.
\end{proof}